\documentclass[a4paper, reqno]{amsart}
\usepackage{color}

\usepackage[T1]{fontenc}
\usepackage{tikz}
\usepackage{mathrsfs}
\usepackage{subfig}
\usepackage{mathabx}

\usepackage{amsmath, amsfonts, amssymb, amsthm, graphics, graphicx}
\usepackage[inline]{enumitem} 
\usepackage{hyperref}
\renewcommand{\\}{\vspace{3mm}}

\newenvironment{edit}{\color{red}}{\color{black}}
\newcommand{\aed}{\begin{edit}}
\newcommand{\zed}{\end{edit}}
\newenvironment{change}{\color{blue}}{\color{black}}
\newcommand{\ach}{\begin{change}}
\newcommand{\zch}{\end{change}}

\providecommand{\aut}{\mathop{\rm Aut \,}\nolimits}
\providecommand{\sym}{\mathop{\rm Sym \,}\nolimits}
\providecommand{\Wr}{\mathop{\rm Wr}\nolimits}

\newcommand{\core}{\mathop{\rm Core}\nolimits}
\newcommand{\sd}{\mathop{\rm sd}\nolimits}

\newcommand{\N}{\mathbb{N}}
\newcommand{\Z}{\mathbb{Z}}

\newcommand{\MinClosedNormal}{B}
\newcommand{\DirectProd}{M}
\newcommand{\sslash}{\mathbin{/\mkern-6mu/}}

\theoremstyle{plain}
\newtheorem{theorem}{Theorem}
\newtheorem*{theorem*}{Theorem}
\newtheorem{lemma}[theorem]{Lemma}
\newtheorem{corollary}[theorem]{Corollary}
\newtheorem*{corollary*}{Corollary}
\newtheorem{proposition}[theorem]{Proposition}
\newtheorem{conjecture}[theorem]{Conjecture}

\theoremstyle{definition}
\newtheorem{definition}[theorem]{Definition}
\newtheorem*{definition*}{Definition}

\newtheorem{remark}[theorem]{Remark}
\newtheorem{example}[theorem]{Example}
\newtheorem*{example*}{Example}
\newtheorem*{remark*}{Remark}

\title{\bf The structure of primitive permutation groups with finite suborbits and t.d.l.c. groups admitting a compact open subgroup that is maximal}

\author{Simon M.~Smith}

\address{School of Mathematics and Physics, University of Lincoln, Lincoln, U.K.
}

\date{\today}

\email{sismith@lincoln.ac.uk}

\thanks{An early part of this research was supported by an Australian Research Council Discovery Early Career Researcher Award (project number DE130101521)}

\begin{document}
\maketitle

\markboth{\textsc{Primitive groups with finite subdegrees}}{\textsc{Primitive groups with finite subdegrees}}

\begin{abstract} This paper is about the structure of infinite primitive permutation groups and totally disconnected locally compact groups (``tdlc groups'').
The permutation groups we investigate are {\em subdegree-finite}, that is, all orbits of point stabilisers are finite. Automorphism groups of connected, locally finite graphs are examples of subdegree-finite permutation groups.

The tdlc groups we investigate all have a maximal subgroup that is compact and open. Tdlc groups {\em with few open subgroups}, recently studied by Pierre-Emmanuel Caprace and Timoth\'{e}e Marquis, are examples of such groups.




We prove a classification result, and use it to show that every closed, subdegree-finite primitive group is a primitive subgroup of a product $H \Wr F_1 \boxtimes F_2 \Wr \cdots \boxtimes F_{m-1} \Wr F_m$, where $H$ is a closed, subdegree-finite and primitive group that is either finite or one-ended and almost topologically simple. The groups $F_i$ are transitive and finite (and subject to some constraints) and $m$ is finite. The product $\boxtimes$ here denotes the recently discovered {\em box product}, and all wreath products here act via their product action.

We apply this permutational result to tdlc groups. If $G$ is a tdlc group then it contains a compact open subgroup $V$. The permutation group induced by the action of $G$ on the coset space $G / V$ is called the {\em Schlichting completion} of the pair $(G,V)$, and is denoted $G \sslash V$. Knowledge of this action of $G$ on $G/V$ underpins many recent influential results in tdlc theory. We show that if $V$ is a maximal subgroup of $G$, and $G$ is non-compact, then $G \sslash V$ is subject to a topological interpretation of our result for primitive groups. 
We use this to show that in this case, if $G \sslash V$ is nondiscrete, preserves no nontrivial homogenous cartesian decomposition on $G / V$, and does not split nontrivially over a compact open subgroup, then the monolith of $G \sslash V$ is a nondiscrete, one-ended, topologically simple, compactly generated tdlc group.
\end{abstract}

\section{Introduction}

\subsection{Local compactness and subdegree-finiteness} \label{subsec:localComp}
When studying infinite objects, it is natural to focus on those with some local finiteness property. For topological groups, this might be that the group topology is {\em locally compact} (that is, every group element lies in a compact neighbourhood); for permutation groups, this might be that the group is {\em subdegree-finite} (that is, all orbits of point stabilisers are finite). In fact these two types of groups are closely related, and their structural properties are intertwined. The results we describe in this paper pertain to both locally compact groups, and to subdegree-finite permutation groups.

Let us look first at locally compact groups. If $G$ is such a group, then it is well-known that the connected component $C$ of the identity is a closed normal subgroup of $G$, and so $G/C$ is a totally disconnected 
and locally compact group (henceforth referred to as a {\em tdlc} group), while of course $C$ is a connected locally compact group. In this way, the structure theory of locally compact groups decomposes into two cases: the connected locally compact case and the tdlc case.
A consequence of the solution of Hilbert's 5th Problem (by Gleason, Montgomery \& Zippin, and Yamabe) is that the broad structure of connected locally compact groups is known: such groups are pro-Lie. For many years the tdlc case was treated as the degenerate case, but following breakthrough work by Willis (see \cite{willis94} and \cite{willis01})
in the 1990s and early 2000s,
it became clear that the structure theory of tdlc groups might be more tractable than first assumed. Since then, there has been a growing body of work in this area.

The arguments contained in some of these cited works use, as a fundamental tool, the following natural permutation representations of tdlc groups (described in more detail in Section~\ref{section:prelims}). If $G$ is a tdlc group then it is known by an old result of van Dantzig (see \cite{vanDantzig} or \cite[Theorem 7.7]{hewitt_ross}) that $G$ always contains a compact open subgroup $V$. There is an obvious action (by multiplication) of $G$ on the set $G/V$ of cosets of $V$ in $G$. The kernel of this action is closed and compact in $G$.
The action induces a group of permutations of $G/V$, and we denote this permutation group by $G \sslash V$. This group $G \sslash V$, viewed as a topological group under the topology of pointwise convergence, is sometimes called the {\em Schlichting completion} for the pair $(G,V)$ (see \cite{reid_wesolek:homomorphisms}). The group $G \sslash V$ is a closed and tdlc subgroup of the group 
of all permutations of $G/V$. As a permutation group, $G \sslash V$ is subdegree-finite.
Obviously, if one understands the structure and permutational behaviour of $G \sslash V$, then one gains a deep understanding of the structure of the topological group $G$.

In this paper we obtain structural results for $G \sslash V$ in the case where $V$ is a maximal subgroup of $G$. The class of tdlc groups with this property includes tdlc groups with few open subgroups, studied by P.~E.~Caprace and T.~Marquis in \cite{caprace:marquis}.
Precise statements of our results are given below, but intuitively we show that the group $G \sslash V$ is a ``large'' subgroup of a group obtained by applying the wreath product (in its product action) and the box product (see Section~\ref{section:boxproduct}) finitely many times to some ``seed'' group, where this seed group is either a finite primitive permutation group or an infinite primitive permutation group that is contained in the automorphism group of a one-ended topologically simple, compactly generated tdlc group. This result can be combined with techniques outlined in the recent book  \cite{PraegerSchneider}, by C.~E.~Praeger and C.~Schneider, to detect one-ended, compactly generated, topologically simple tdlc groups inside infinite primitive permutation groups.

Let us now turn our attention to permutation groups. Let $\Omega$ be a set, and denote the group of all permutations of $\Omega$ by $\sym(\Omega)$ (when $\Omega = \{1,\ldots, m\} \subseteq \N$ we write $S_m$ for brevity). There is a natural topology on the group $\sym(\Omega)$, called the {\em permutation topology}. In this topology, a neighbourhood basis of the identity consists of all pointwise stabilisers of finite subsets of $\Omega$. This topology is equal to the topology of pointwise convergence and the compact-open topology.
Intuitively, in this topology two permutations are close if they agree on a large finite subset of $\Omega$. 
Note that a group $G \leq \sym(\Omega)$ and its closure $\overline{G}$ cannot be distinguished by their action on finite subsets of $\Omega$. We will typically be concerned with closed permutation groups in this paper. To apply our results to an arbitrary permutation group $G$, one must first take its closure.

We are interested in permutation groups $G \leq \sym(\Omega)$ that are  subdegree-finite. Orbits of point stabilisers are called {\em suborbits}, and the cardinality of a suborbit is called a {\em subdegree}. Thus, a subdegree-finite permutation group is a group with only finite suborbits.
These permutation groups arise naturally: any automorphism group of a connected, locally finite graph is closed and subdegree-finite; the natural permutation representations of tdlc groups described above are closed and subdegree-finite. On the other hand, any closed, subdegree-finite permutation group is (under the permutation topology) a tdlc group, with compact open point stabilisers.
In this way, the study of tdlc groups and the study of closed subdegree-finite permutation groups are essentially the same.

Permutation groups can be decomposed in a natural way. This is obvious in the case of intransitive groups (that is, groups with more than one orbit), while for 
transitive group the decomposition is based on the notion of primitivity. A nontrivial transitive group $G$ of permutations of a set $\Omega$ is called {\em imprimitive} if $\Omega$ admits a $G$-invariant equivalence relation whose equivalence classes are neither singletons, nor all of $\Omega$. The action of an imprimitive group $G$ decomposes into two distinct actions: the group on the one hand induces a transitive group $F$ of permutations on the set $\Sigma$ of equivalence classes, and on the other hand the setwise stabiliser of an equivalence class $\Delta$ induces a transitive group $H$ of permutations on $\Delta$. Famously, $G$ can be embedded as a permutation group into the wreath product $H \wr F$ acting via its imprimitive action on $\Delta \times \Sigma$.
If $G$ is not imprimitive, then it is {\em primitive}. Thus, we see that imprimitive permutation groups decompose naturally, whereas primitive permutation groups are in this sense indecomposable.  
Thus, to understand the general structure of permutation groups, one must first start by understanding those that are primitive.

Let $\mathcal{P}$ denote the class of nonregular, closed, primitive, subdegree-finite permutation groups. We exclude regular groups (that is, transitive permutation groups whose point stabilisers are trivial) from this class because they are not interesting: the primitive regular groups are cyclic of prime order. 
Structural information about some groups in $\mathcal{P}$ is already known, for example finite groups (see \cite{liebeck&praeger&saxl:finite_onan_scott}), infinite discrete groups (see \cite{smith:discrete_prim}), and groups whose socle has the FCR property (see \cite{PraegerSchneider}). We discuss these results further in Section~\ref{into:comparisons}.

In this paper we obtain structural information about all infinite groups in $\mathcal{P}$. A precise statement is given below, but intuitively, we show that any infinite group $G \in \mathcal{P}$ is a ``large'' subgroup of a group obtained by applying the wreath product (in its product action) and the box product (see Section~\ref{section:boxproduct} below) finitely many times to either a finite group in $\mathcal{P}$, or to a one-ended group in $\mathcal{P}$ that is almost topologically simple. It is from this result that we derive our corresponding result about tdlc groups.

Note that all groups in $\mathcal{P}$ have degree at least three. For each $G \in \mathcal{P}$, we define
$\sd(G) \in \N$ to be the minimal nontrivial subdegree of $G$; we track this number through our arguments, to show that any iterated sequences of products must terminate after finitely many steps.
 
\subsection{One-ended almost topologically simple groups} \label{subsection:one_ended}

A connected graph $\Gamma$ with bounded valency on which $G$ acts transitively as a group of automorphisms with compact open vertex stabilisers is called a {\em Cayley--Abels graph} for $G$. This notion appears (with different names) in \cite[Beispiel 5.2]{abels-74} and \cite[Section 2]{KronMoller}.  A tdlc group $G$ has a Cayley--Abels graph if and only if $G$ is compactly generated, and any two Cayley--Abels graphs for $G$ are quasi-isometric 
(see \cite{KronMoller}).
Since quasi-isometric graphs have the same ends, it makes sense to talk about the {\em ends} of a compactly generated tdlc group as being the ends of any of its Cayley--Abels graphs. This is a standard tool for working with compactly generated tdlc groups.

For $G \leq \sym(\Omega)$, we have another standard tool: an {\em orbital graph} of $G$ has vertex set $\Omega$ with its edge set consisting of the orbit $\{\alpha, \beta\}^G$, for some $\alpha, \beta \in \Omega$. 
If $G$ is transitive, subdegree-finite and closed, and some finite union $\Gamma$ of orbital graphs of $G$ is connected, then $\Gamma$ is a Cayley--Abels graph for $G$ (under the permutation topology), and so we can talk about the {\em ends} of $G$ as being the ends of $\Gamma$. Of course any orbital graph of a group is an orbital graph of its closure, so we can speak of the ends of permutation groups that are not necessarily closed. A group is {\em one-ended} if it has only one end.

If a group $G \leq \sym(\Omega)$ has a nonabelian, closed normal subgroup $\MinClosedNormal$ that is topologically simple (that is, $\MinClosedNormal$ has no closed normal subgroups) and $\MinClosedNormal \leq G \leq \aut \MinClosedNormal$, we say that $G$ is {\em almost topologically simple}. Note that this definition does not involve the $\aut \MinClosedNormal$ topology; the topology on $\MinClosedNormal$ and $G$ is the permutation topology. In the special case when $G$ (and therefore $\MinClosedNormal$) has precisely one end, we say that $G$ is {\em one-ended almost topologically simple} and abbreviate this to {\em OAS}. Groups of type OAS in $\mathcal{P}$ can be discrete or nondiscrete (see Section~\ref{section:examples}).

\subsection{Box products and fibrelobes} \label{intro:box_prods}
There are two products that play a fundamental role in our structure theory for groups in $\mathcal{P}$: the wreath product in its product action and the box product. These two products are in some sense duals of each other; to see this relationship clearly, we must first recall the fibres of a wreath product, define the lobes of a box product, and introduce the combined notion of a {\em fibrelobe}.

Let $m \geq 2$ and let $X$ be a set containing at least three elements. Suppose
 $G_1 \leq \sym(X)$ and $G_2 \leq S_m$ are two permutation groups. Their wreath product famously has two important actions, one on the set $X \times \{1,\ldots,m\}$ (which we will call the {\em imprimitive action}) and the other on the set $X^{m}$ of functions from $\{1,\ldots,m\}$ to $X$ (which is called the {\em product action}). 
We use both resulting permutation groups frequently, so adopt different notations for them: in this paper $G_1 \wr G_2$ refers to the imprimitive action, and $G_1 \Wr G_2$ to the product action. 
Recall that a {\em fibre} of $G_1 \Wr G_2$ is any subset of $X^m$ of the form $\{(x_1, \ldots, x_{i-1}, y, x_{i+1}, \ldots, x_m) : y \in X\}$. The {\em points} of $G_1 \Wr G_2$ are the  elements of $X^m$.

The {\em box product} of $G_1$ and $G_2$, denoted $G_1 \boxtimes G_2$, was first introduced in \cite{smith:product} and
is defined fully in Section~\ref{section:boxproduct}; here we give only an intuitive sense of it. This intuitive picture is accurate when $G_1$ and $G_2$ are transitive, so let us assume this to be the case. (Although here we have assumed that $G_2$ is a finite group, the box product construction is valid also when both $G_1$ and $G_2$ are infinite.)

We begin by constructing a tree-like graph. Let $\Delta$ be a transitive graph with vertex set $X_1$ such that the connectivity of $\Delta$ is at least two and $G_1 \leq \aut \Delta $ (for example, $\Delta$ can always be taken to be the complete graph on $X_1$ although often a smaller graph will suffice).
Now ``glue'' $m$ copies of $\Delta$ together so they meet only at a single vertex. Now to all vertices $v$ in the resulting graph that currently lie in only one copy of $\Delta$, ``glue'' $m-1$ additional copies of $\Delta$ to $v$, so that $v$ lies in $m$ copies of $\Delta$ that meet only at $v$. Continue this process until we obtain a connected, tree-like graph $\Gamma(\Delta, m)$ in which every vertex lies in precisely $m$ copies of $\Delta$, and any two copies of $\Delta$ in $\Gamma(\Delta, m)$ are either disjoint or intersect at a single vertex. In the left hand part of Figure~\ref{fig:Gamma_D8_S3} we see a depiction of $\Gamma(K_4, 3)$. 

The vertex sets of the copies of $\Delta$ in $\Gamma(\Delta, m)$ are called {\em lobes}.
We say a subgroup $G \leq \aut \Gamma(\Delta, m)$ is {\em locally-$(G_1, G_2)$} if the setwise stabiliser in $G$ of any lobe of $\Gamma(\Delta, m)$ induces $G_1$ on that lobe, and the stabiliser in $G$ of any vertex $\alpha$ of $\Gamma(\Delta, m)$ induces $G_2$ on the set  of lobes of $\Gamma(\Delta, m)$ that contain $\alpha$.
The box product $G_1 \boxtimes G_2$ is a subgroup of $\aut \Gamma(\Delta, m)$ that is maximal subject to being locally-$(G_1, G_2)$. The {\em lobes} and {\em points} of $G_1 \boxtimes G_2$ are, respectively, the lobes and vertices of $\Gamma(\Delta, m)$.
For a precise description of the construction of the box product, see Section~\ref{section:boxproduct}.

\begin{remark} \label{rem:WrAndBox}
The fibres of $G_1 \Wr G_2$ and the lobes of $G_1 \boxtimes G_2$ have much in common. To explore this, let $\boxdot$ denote either $\Wr$ or $\boxtimes$, and call the lobes or fibres of $G_1 \boxdot G_2$ the {\em fibrelobes} of $G_1 \boxdot G_2$. If $\alpha$ is a point of $G_1 \boxdot G_2$ then $\mathcal{F}(\alpha)$ denotes the set of fibrelobes of $G_1 \boxdot G_2$ that contain $\alpha$.

The following statements are famous properties of the wreath product in its product action when $G_1$ and $G_2$ are transitive. Using the concept of a fibrelobe, we see they apply equally to the box product (see Section~\ref{section:boxproduct}).
\begin{itemize}
\item
	$G_1 \boxdot G_2$ permutes its points transitively, and permutes its fibrelobes transitively.
\item
	If $\Delta$ is a fibrelobe of $G_1 \boxdot G_2$, then its setwise stabiliser $(G_1 \boxdot G_2)_{\{\Delta\}}$ induces $G_1$ on $\Delta$.
\item
	If $\alpha$ is a point of $G_1 \boxdot G_2$, then its stabiliser $(G_1 \boxdot G_2)_\alpha$ induces $G_2$ on $\mathcal{F}(\alpha)$.
\item
	If $G_2$ is finite, then $G_1 \boxdot G_2$ is a primitive permutation group if and only if $G_1$ is primitive and not regular.
\end{itemize}

Thus, we can see the ``local'' behaviour of $G_1 \boxdot G_2$ in the induced actions of $(G_1 \boxdot G_2)_{\alpha}$ and $(G_1 \boxdot G_2)_{\{\Delta\}}$. We also see primitivity here as a ``local-to-global'' condition. The following definition captures the notion of a subgroup of $G_1 \boxdot G_2$ being large enough to induce (perhaps as a dense subgroup) the ``local'' behaviour of $G_1 \boxdot G_2$.
\end{remark}

\begin{definition} \label{def:fibrelobe_full}
Let $\boxdot$ denote either $\Wr$ or $\boxtimes$. The the lobes or fibres of $G_1 \boxdot G_2$ are the {\em fibrelobes} of $G_1 \boxdot G_2$.
We say that a subgroup $G \leq G_1 \boxdot G_2$ is {\em fibrelobe-full} if the following hold: 
\begin{enumerate}
\item
	$G$ is transitively permutes the points, and transitively permutes the fibrelobes, of $G_1 \boxdot G_2$.
\item
	If $\Delta$ is a fibrelobe of $G_1 \boxdot G_2$, then the setwise stabiliser $G_{\{\Delta\}}$ induces a group on $\Delta$ whose closure in $\sym(\Delta)$ is $G_1$.
\item \label{fibrelobe:item}
	If $\alpha$ is a point of $G_1 \boxdot G_2$, then the stabiliser $G_\alpha$ induces a group on $\mathcal{F}(\alpha)$ whose closure in $\sym(\mathcal{F}(\alpha))$ is $G_2$.
\end{enumerate}
\end{definition}

\begin{remark} \label{rem:Basic_itteration_of_products}
One can of course apply the products $\Wr$ and $\boxtimes$ repeatedly in an alternating sequence, to obtain ever-larger permutation groups. For example, we might construct a group $((G_1 \Wr G_2) \boxtimes G_3) \Wr G_4$, whose fibrelobes are the fibres of $H \Wr G_4$, where $H = (G_1 \Wr G_2) \boxtimes G_3$. If $G$ is a fibrelobe-full subgroup of $((G_1 \Wr G_2) \boxtimes G_3) \Wr G_4$ and $\Delta$ is a fibrelobe, notice that the closure in $\sym(\Delta)$ of the group induced by $G_{\{\Delta\}}$ on $\Delta$ is $H = (G_1 \Wr G_2) \boxtimes G_3$, and $H$ is obviously fibrelobe-full in $(G_1 \Wr G_2) \boxtimes G_3$. Thus, being fibrelobe-full is in some sense a hereditary property when the products $\boxtimes$ and $\Wr$ are applied repeatedly.
\end{remark}

\begin{remark} Suppose $G_1 \boxdot G_2$ is transitive, closed, with all suborbits finite (for example, $G_1$ and $G_2$ could be closed and transitive, with $G_1$ subdegree-finite and $G_2$ finite). If $G \leq G_1 \boxdot G_2$ is fibrelobe-full, then $G$ is cocompact in $G_1 \boxdot G_2$.
\end{remark}

\subsection{Theorems for permutation groups} Here we present our main results as they apply to permutation groups.

Recall that $\mathcal{P}$ is the class of nonregular, closed, primitive permutation groups whose suborbits are all finite. 
Our first theorem is a complete classification of infinite groups in $\mathcal{P}$.

\begin{theorem} \label{thm:MainTheorem}
If $G \in \mathcal{P}$, then $G$ falls into precisely one of the following categories:
\begin{enumerate}
\item[(FIN)]
	$G$ is a nonregular finite primitive permutation group;
\item[(OAS)] \label{type:topo_simple}
	$G$ is one-ended and almost topologically simple;
\item[(PA)] \label{type:PA}
	$G$ is a one-ended, primitive, fibrelobe-full subgroup of a wreath product $H \Wr F$ acting with its product action, where $F$ is some nontrivial finite transitive group, and $H \in \mathcal{P}$ is infinite of type OAS or BP, and $\sd(H) \leq \sd(G)$, with $G$ discrete if and only if $H$ discrete; 
\item[(BP)] \label{type:BP}
	$G$ is a nondiscrete, $2^{\aleph_0}$-ended, primitive, fibrelobe-full subgroup of a box product $H \boxtimes F$, where $F$ is  some nontrivial finite transitive group, and $H \in \mathcal{P}$ is of type FIN, OAS or PA, and $\sd(H) < \sd(G)$
\end{enumerate}
\end{theorem}

The class OAS can be further subdivided into OAS$_d$, which contains those OAS groups that are discrete, and OAS$_n$, which contains those OAS groups that are nondiscrete. The same subdivision can also be made for the class PA, into PA$_d$ and PA$_n$.
Groups of type BP are always nondiscrete. The discrete groups in $\mathcal{P}$ have already been classified: finite groups in $\mathcal{P}$ were famously classified by the O'Nan--Scott Theorem (see \cite{liebeck&praeger&saxl:finite_onan_scott}); infinite discrete groups in $\mathcal{P}$ are classified in \cite[Theorem 1.1]{smith:discrete_prim}. These results can be used to further refine the statement of Theorem~\ref{thm:MainTheorem}. We discuss this further in Section~\ref{into:comparisons}.\\
 
 Examples of  groups in $\mathcal{P}$ exist for all types: FIN, OAS$_d$, OAS$_n$, PA$_d$, PA$_n$ and BP (see Section~\ref{section:examples}). \\
 
As noted in the introduction, it is known precisely when groups of the form $G_1 \Wr G_2$ and $G_1 \boxtimes G_2$ are primitive. No such knowledge exists for one-ended almost topologically simple groups.

Theorem~\ref{thm:MainTheorem} can be combined with the following theorem and its corollary to give an O'Nan--Scott style theorem for infinite groups in $\mathcal{P}$. 

\begin{theorem} \label{thm:min_normal_subgroups}
Let $G \leq \sym(\Omega)$ be an infinite group in $\mathcal{P}$, and suppose $\alpha \in \Omega$. Then
$G$ has a unique (nontrivial) minimal closed normal subgroup $\MinClosedNormal$ which is nonabelian, and there exist finitely many topologically simple, pairwise isomorphic, nonabelian, infinite permutation groups $K_1, \ldots, K_m$ such that each $K_i$ is closed and normal in $\MinClosedNormal$ and $\DirectProd := K_1 \times \cdots \times K_m$ is a dense subgroup of $\MinClosedNormal$ that is normal in $G$. The group $B$ has an infinite point stabiliser if and only if $G$ has an infinite point stabiliser.
Furthermore, $G_\alpha$ acts on the components of $M$ by conjugation, inducing a transitive subgroup $F \leq S_m$.
\end{theorem}

\begin{remark} The group $F$ in the PA case of Theorem~\ref{thm:MainTheorem} and the group $F$ in Theorem~\ref{thm:min_normal_subgroups}, are the same permutation group. This is obviously not the case for groups of BP type. 
\end{remark}

\begin{corollary} \label{cor:min_normal_subgroups_cor}
In the statement of Theorem~\ref{thm:min_normal_subgroups}, we have that $m=1$ if and only if $G$ is of type OAS or BP, and in this case $M = B = K_1$. When $m > 1$, our group $G$ is of type PA and the minimal closed normal subgroup of $H$ is $K_1$.
\end{corollary}

\begin{remark} \label{rem:RepeatedApplicationOfMainTheorem}
The structural decomposition of an infinite group $G \in \mathcal{P}$ described in Theorem~\ref{thm:MainTheorem} can be repeated. For example, if $G$ is of type (PA), with $G$ a fibrelobe-full subgroup of $H \Wr F$ for some $H \in \mathcal{P}$, then one can apply Theorem~\ref{thm:MainTheorem} now to $H$ and obtain a further decomposition. Iterating the decomposition in this way, we note that products $\Wr$ and $\boxtimes$ must alternate with $\sd$ non-increasing, and each time $\boxtimes$ occurs the value of $\sd$ must strictly decrease. Hence we can be certain that such an iterated decomposition must terminate after finitely many steps in an OAS or finite group in $\mathcal{P}$.

On the other hand, one can start with a group in $H \in \mathcal{P}$ and obtain every-larger groups in $\mathcal{P}$ by repeatedly applying the wreath and box products. Indeed, if $F_1, \ldots, F_n$ are finite transitive groups, then $G:=(((H \Wr F_1 ) \boxtimes F_{2} ) \Wr \cdots \boxtimes F_{n-1}) \Wr F_n$ lies in $\mathcal{P}$. Notice that, as long as some $F_i$ is nontrivial, the resulting group $G$ is never of type OAS, since the application of a box product results in a group with infinitely many ends, and the application of the wreath product results in a group with a non-topologically simple socle.
\end{remark}

Intuitively, we can now see all infinite groups in $\mathcal{P}$ as ``large'' subgroups of groups built from finite (noncyclic) primitive permutation groups or OAS groups in $\mathcal{P}$, together with the repeated application of $\Wr$ and $\boxtimes$ finitely many times. 
We formalise this intuitive view of groups in $\mathcal{P}$ in the following result. 

\begin{theorem} \label{thm:main_product_statement}
If $G \in \mathcal{P}$, then there is a finite sequence $F_1, \ldots, F_n$ of finite transitive groups, with $F_i$ nontrivial for all $1 < i < n$, and a group $H \in \mathcal{P}$ that is of type FIN or OAS, such that $G$ is a primitive 
subgroup of
\[(((H \Wr F_1 ) \boxtimes F_{2} ) \Wr \cdots \boxtimes F_{n-1}) \Wr F_n.\]
\end{theorem}

In the above theorem, the groups $F_1$ or $F_n$ (or both) could be trivial, allowing $G \leq (((H \boxtimes F_{2} ) \Wr \cdots \boxtimes F_{n-1})$ for example as a possible conclusion.

There is a partial converse to Theorem~\ref{thm:main_product_statement}: if $H \in \mathcal{P}$ and $F_1, \ldots, F_n$ are finite transitive groups with $F_i$ nontrivial for all $1 < i < n$, then the group $(((H \Wr F_n ) \boxtimes F_{n-1} ) \Wr \cdots \boxtimes F_{2}) \Wr F_1$ always lies in $\mathcal{P}$. This follows easily from Theorems~\ref{thm:box_prim}, \ref{prop:wr_prim} and
\ref{thm:box_product_summary}.\\


\subsection{Theorems for topological groups} Here we give our main results, as they apply to topological groups.
The types OAS, PA and BP in the following theorems are the topological interpretations of the types of the same name given in Theorem~\ref{thm:MainTheorem}. 
 
Recall that $\mathscr{S}$ is the class of nondiscrete, compactly generated tldc groups that are topologically simple, and $\mathcal{P}$ is the class of nonregular, closed, primitive permutation groups whose suborbits are all finite. Our results relate aspects of these two classes.

Suppose $G$ is a non-compact tdlc group and $V$ a compact open subgroup.  
There will be a number of topological groups at play, so let us state clearly their relationship to one another. Firstly, let $\mathfrak{t}$ denote the topology of $G$, so $G$ as a topological group can be written $(G, \mathfrak{t})$. We have an action $\lambda$ of $G$ on $G/V$ which gives rise another topology for $G$, namely the permutation topology $\mathfrak{p}$ (described for permutation groups in Section~\ref{subsec:localComp}, and more generally in Section~\ref{subsection:topo_groups}). Finally let $\mathfrak{q}$ denote the quotient topology.
Thus we have topological groups $(G, \mathfrak{t})$ and $(G, \mathfrak{p})$ and $(G \sslash V, \mathfrak{p})$ and $((G, \mathfrak{p})/\ker(\lambda), \mathfrak{q})$. These topological groups enjoy the following relationships, which the reader may easily verify:
\begin{enumerate}
\item
	For $(G, \mathfrak{t})$ and $(G, \mathfrak{p})$ we have $\mathfrak{p} \subseteq \mathfrak{t}$.
\item
	$((G, \mathfrak{p}) / \ker(\lambda), \mathfrak{q})$ and $(G \sslash V, \mathfrak{p})$ are isomorphic as topological groups via the map $\ker(\lambda)g \mapsto \lambda(g)$.
\end{enumerate}

From these observations we have the relationship between $(G, \mathfrak{t})$ and $(G \sslash V, \mathfrak{p})$.

Our permutational theorems apply to the infinite permutation group $G \sslash V$ whenever $G \sslash V \in \mathcal{P}$. This occurs precisely when $G \sslash V$ is primitive, since it is always closed and subdegree-finite. Moreover, $G \sslash V$ is primitive if and only if $V$ is maximal in $G$ as an abstract group (that is, there is no subgroup $W$ such that $V < W < G$). Thus, our results apply to non-compact tdlc groups $G$
 possessing a compact open subgroup $V$ that is maximal. 

Suppose $G$ and $V$ are such that $G \sslash V \in \mathcal{P}$. Since point stabilisers in $G \sslash V$ are compact, open and maximal, $G \sslash V$ is compactly generated. Nontrivial normal subgroups of primitive permutation groups must be transitive, since their orbits form systems of imprimitivity. Thus, all nontrivial normal subgroups of $G \sslash V$ are transitive and therefore cocompact (see Proposition~\ref{Proposition_A}). In other words, all proper quotients of $G \sslash V$ are compact. 

Recall the following. If the intersection $L$ of all non-trivial closed normal subgroups of $G$ is nontrivial, then we call $L$ the {\em monolith} of $G$ and say that $G$ is {\em monolithic}. If $G$ has closed normal subgroups $N_1, \ldots, N_m$ such that the multiplication map $N_1 \times \cdots \times N_m \rightarrow G$ is injective with dense image, then we say that $G$ is a {\em quasi-product} with {\em quasi-factors} $N_1, \ldots, N_m$. 

A compactly generated, non-compact, locally compact group whose proper quotients are all compact is called {\em just non-compact}. Discrete groups that are just non-compact are called {\em just infinite} and were fist examined in detail in the classical article \cite{Wilson:JustInfinite}. Nondiscrete just non-compact groups were studied in \cite{cm}, and the conclusions of this investigation are epitomised by the authors as being:
\begin{quote}
{\em A just non-compact group is either discrete or monolithic.}
\end{quote}

In our less general situation, where $V$ is a maximal subgroup of $G$, we are able to go much further. We determine, in all but the one-ended almost topologically simple case, precise information about the multiplication action of $G$ on $G/V$. We can moreover show that $G \sslash V$ is a cocompact, primitive, subgroup of a group obtained from either a finite nonregular primitive permutation group, or a one-ended group whose monolith lies in $\mathscr{S}$, and the repeated application of the wreath and box products, finitely many times. 

\begin{theorem} \label{thm:MainTheoremTDLC} Suppose $G$ non-compact and tdlc, and $V$ is a proper compact open subgroup. If $V$ is maximal in $G$, then $G \sslash V$ is monolithic and there exists a finite nontrivial transitive permutation group $F \leq S_m$ and a tdlc group $H$ that itself has a proper compact open subgroup $W$ that is maximal in $H$, such that $H = H \sslash W$ and precisely one of the following holds for the completion $G \sslash V$.
\begin{enumerate}
\item[(OAS)]
	The group $G \sslash V$ is one-ended and has a nonabelian cocompact monolith $B$ that is one-ended, topologically simple and compactly generated. Moreover, as abstract groups we have $B \leq G \sslash V \leq \aut(B)$. Note that $H$ plays no role in this case, and $G$ can be discrete or nondiscrete. The monolith $B$ is discrete if and only if $G$ is discrete.
\item[(PA)]
	The group $G \sslash V$ has precisely one end and $G \sslash V$ is a fibrelobe-full, primitive (and therefore cocompact) subgroup of $H \Wr F$, where $H \Wr F$ is acting via its product action and its topology is the permutation topology derived from this action. Moreover, $H = H \sslash W$ is infinite and is therefore subject to this classification theorem, with $H$ of type OAS or BP, and $\sd(H) \leq \sd(G \sslash V)$. The group $H$ has a nonabelian cocompact topologically simple and compactly generated monolith $K$ and the monolith of $G \sslash V$ (which is one-ended, compactly generated, nonabelian and cocompact in $G \sslash V$) is a quasiproduct of $m$ copies of $K$. Moreover, point stabiliser in $G \sslash V$ permute the components of this quasiproduct transitively by conjugation, with this action inducing $F$ on the set of components. The group $G \sslash V$ is discrete if and only if $H$ is discrete.
\item[(BP)]
	The group $G \sslash V$ has $2^{\aleph_0}$ ends and is a fibrelobe-full, primitive (and therefore cocompact) subgroup of $H \boxtimes F$, where the topology on $H \boxtimes F$ is the permutation topology. Moreover, $H = H \sslash W$ is either a finite nonregular primitive permutation group, or $H$ is infinite and therefore subject to this classification theorem with $H$ of type  OAS or PA. In all cases, $\sd(H) < \sd(G \sslash V)$. The group $G \sslash V$ has a nonabelian cocompact monolith that is topologically simple and compactly generated with $2^{\aleph_0}$ ends. The group $G \sslash V$ is never discrete.
\end{enumerate}
\end{theorem}

Nondiscrete examples exist for all types (see Section~\ref{section:examples}).

\begin{remark} As noted in Remark~\ref{rem:RepeatedApplicationOfMainTheorem}, once the structure of $G \sslash V$ has been determined using Theorem~\ref{thm:MainTheoremTDLC}, the theorem can then be applied to $H$, and since $H = H \sslash W$ we obtain a structural decomposition for $H$ itself. Similarly, the theorem can then be applied again to the groups that constitute $H$, and so on. The value of $\sd$ is a natural number, which guarantees that this iterated decomposition process must eventually terminate after finitely many steps, and it must terminate in a finite group in $\mathcal{P}$ or a group in $\mathcal{P}$ of type OAS.

On the other hand, as in Remark~\ref{rem:RepeatedApplicationOfMainTheorem} one can use the wreath product and box product to construct ever-larger compactly generated tdlc groups containing compact open subgroups that are maximal.
\end{remark}

The above remark gives us the following topological interpretation of Theorem~\ref{thm:main_product_statement}.

\begin{theorem} \label{thm:topo_main_product_statement}
Suppose $G$ is non-compact and tdlc, and $V$ is a proper compact open subgroup. If $V$ is maximal in $G$, then $G \sslash V \in \mathcal{P}$ is infinite, and
there is a finite sequence $F_1, \ldots, F_n$ of finite transitive groups, with $F_i$ nontrivial for all $1 < i < n$, and a
tdlc group $H$ that itself has a proper compact open subgroup $W$ that is maximal in $H$, such that $H = H \sslash W$  is a finite nonregular primitive permutation group, or it is infinite of type OAS and 
$G \sslash V$ is a primitive (and therefore cocompact) subgroup of 
\[(((H \Wr F_1 ) \boxtimes F_{2} ) \Wr \cdots \boxtimes F_{n-1}) \Wr F_n.\]
The group $(((H \Wr F_1 ) \boxtimes F_{2} ) \Wr \cdots \boxtimes F_{n-1}) \Wr F_n$ is tdlc, compactly generated and has a compact open subgroup that is maximal.
\end{theorem}

\begin{remark} In Theorems~\ref{thm:main_product_statement} and \ref{thm:topo_main_product_statement}, the group $H$ is either finite, or its monolith is a one-ended group in $\mathscr{S}$. Such groups are the subject of much interest (see \cite{smith:product}, for example).
Thus the class $\mathcal{P}$, and the class of tdlc groups admitting a maximal subgroup that is compact and open, are rich sources for examples of one-ended, compactly generated, topologically simple tdlc groups.
\end{remark}

If we combine Theorem~\ref{thm:MainTheoremTDLC} with known techniques for detecting box products and wreath products, then we obtain Theorem~\ref{thm:detecting_simple_tdlc_groups} below.

Before stating Theorem~\ref{thm:detecting_simple_tdlc_groups}, let us first outline a technique for detecting wreath products.
Following \cite{PraegerSchneider}, a {\em nontrivial homogeneous cartesian decomposition}, $\mathcal{E}$, of a set $\Omega$, is a finite set of partitions $\{\Gamma_1, \ldots, \Gamma_m\}$ of $\Omega$ such that $m > 1$, each partition has at least two parts, all partitions have the same cardinality and
\[|\gamma_1 \cap \cdots \cap \gamma_m| = 1 \quad \text{ for each } \quad \gamma_1 \in \Gamma_1, \ldots, \gamma_m \in \Gamma_m.\]
For $G \leq \sym(\Omega)$ we say that $\mathcal{E}$ is {\em $G$-invariant} if the partitions in $\mathcal{E}$ are permuted by $G$. It is easy to see that if $G$ is permutationally isomorphic to a subgroup of $\sym(X) \Wr S_m$, with $|X|, m > 1$, then there is a nontrivial homogeneous cartesian decomposition on $\Omega$ that is $G$-invariant (see, for example, the proof of Theorem~\ref{thm:detecting_simple_tdlc_groups}). On the other hand, \cite[Theorem 5.13]{PraegerSchneider} shows that if $G \leq \sym(\Omega)$ leaves invariant a nontrivial homogeneous cartesian decomposition $\{\Gamma_1, \ldots, \Gamma_m\}$ on $\Omega$, then $G$ is permutationally isomorphic to a subgroup of $\sym(\Gamma_1) \Wr S_m$.

\begin{theorem} \label{thm:detecting_simple_tdlc_groups}
Suppose that $G$ is non-compact tdlc and $V$ is a  compact open subgroup. Suppose further that $V$ is maximal in $G$ and $G \sslash V$ is nondiscrete. (For example, one could take $G \in \mathcal{P}$ with all point stabilisers infinite.)

If $G \sslash V$ preserves no nontrivial homogeneous cartesian decomposition on $G / V$, and $G \sslash V$ does not split nontrivially as an amalgamated free product over a compact open subgroup, then the monolith of $G \sslash V$ is nondiscrete, one-ended, topologically simple, compactly generated and tdlc.
\end{theorem}

This theorem can be used to find compactly generated tdlc groups that are one-ended and topologically simple.

\subsection{Comparisons with O'Nan--Scott style results}
\label{into:comparisons}

The finite primitive permutation groups are classified by the so-called O'Nan--Scott theorem. The genesis of this theorem was due, independently, to O'Nan and Scott (see \cite{Scott:Onan_Scott}). Over time, O'Nan and Scott's original result was extended and refined (see \cite{aschbacher_scott}, \cite{kovacs1989} and \cite{liebeck&praeger&saxl:finite_onan_scott}) into its modern formulation. The theorem divides primitive groups into classes. In Peter Cameron's book (\cite{cameron:permutation_groups}) some of these classes are called {\em basic}, and the theorem can then be neatly summarised as: 
\begin{quote}
{\em Every finite primitive permutation group is either basic or a primitive subgroup of a wreath product $H \Wr S_m$ in product action on $X^m$, where $H$ is a basic primitive group on $X$.} 
\end{quote}
In Cameron's nomenclature, a basic group is one that cannot be embedded in a wreath product in its product action (there is a slight inaccuracy caused by adopting this definition, as noted in \cite[Section 1.4]{PraegerSchneider}). The idea is that basic groups are somehow intrinsically primitive, whereas all non-basic groups are ``built'' from basic groups using wreath products and embeddings.

The O'Nan--Scott Theorem for finite primitive permutation groups has been extended to infinite groups in some cases. In \cite{gelander_and_glasner:onan_scott}, an O'Nan--Scott style classification is developed for countable linear groups.
In \cite{infinitary_versions}, Cheryl Praeger and Dugald Macpherson showed that most of the arguments from the finite case can be carried over to infinite primitive permutation groups that posses a minimal closed normal subgroup that itself contains a minimal closed normal subgroup. At the time of the paper's publication, there were no natural classes of infinite permutation group that satisfied this condition on minimal normal subgroups.
In \cite{smith:discrete_prim} it was shown that this condition holds in a very natural situation: infinite primitive permutation groups with finite point stabilisers. From this observation, an O'Nan--Scott type theorem for infinite primitive permutation groups with finite point stabilisers was developed, and it inspired an investigation into other classes of permutation groups whose stabilisers are finite (\cite{PraegerNeumannSmith}). We show in Theorem~\ref{thm:min_normal_subgroups} that all groups in $\mathcal{P}$ also satisfy this condition on minimal normal subgroups. However, as noted below, the discovery of the box product means that applying results in \cite{infinitary_versions} to groups in $\mathcal{P}$ does not give a satisfactory classification, since much more can be said.

In the recent book \cite{PraegerSchneider}, a single O'Nan--Scott theory is developed for all quasiprimitive (and therefore primitive) permutation groups $G$ with a {\it socle that is FCR} (that is, the subgroup of $G$ generated by all the minimal normal subgroups of $G$ is a direct product of abstractly simple groups). Henceforth we shall say that $G$ {\em has the FCR property} if the socle of $G$ is FCR. The groups considered in \cite{smith:discrete_prim} have the FCR property, but groups in $\mathcal{P}$ may not have this property. In fact, whether or not groups in $\mathcal{P}$ have the FCR property is directly related to an important and well-known open problem from topological group theory: {\em Is there a compactly generated, locally compact group that is topologically simple, but not abstractly simple?} This is Problem 19.73 in the Kourovka Notebook, \cite{kourovka19}.

Although they might not use Cameron's terminology, all of the O'Nan--Scott type results we described above have the same broad structure: they divide primitive groups into classes, with some classes akin to basic groups, and a group is either basic or a subgroup of a wreath product $H \Wr S_m$ in product action on $X^m$, where $H$ is a basic primitive group on $X$.

In \cite{smith:product}, the box product was introduced, described as a new product for permutation groups. In the paper, it was shown that
 the wreath product in its product action is not the only product to preserve primitivity. As noted previously, for a finite permutation group $F$ and another permutation group $H$,
both $H \Wr S_m$ and $H \boxtimes S_m$ are primitive if and only if $H$ is nonregular and primitive, and $F$ is transitive.

In light of this discovery,  the intuitive idea of a basic group can be revised so that it again agrees with the spirit of Cameron's original idea: basic primitive group are somehow intrinsically primitive and all non-basic primitive groups are ``built'' from basic groups using wreath or box products, and embeddings. This revision doesn't affect finite primitive permutation groups, nor does it affect the groups considered in \cite{smith:discrete_prim}, because $H \boxtimes S_m$ is never discrete when it is primitive. Of course the problem with this revision is that for a non-basic primitive group $G$ lying in $H \Wr S_m$ or $H \boxtimes S_m$, we may not be able to conclude that $H$ is basic. 

In our main results, we see that for the class $\mathcal{P}$, we can 
incorporate this revised notion of a basic primitive permutation group, and still retain an elegant decomposition theorem. By taking  the basic groups in $\mathcal{P}$ to be those that are either finite, or one-ended almost topologically simple, we can neatly summarise our results as:
\begin{quote}
{\em Every group in $\mathcal{P}$ is either basic or a primitive subgroup of a product $((H \Wr S_{m_1}) \boxtimes S_{m_2}) \Wr \cdots \boxtimes S_{m_{n-1}}) \Wr S_{m_n}$, where $H \in \mathcal{P}$ is basic and $n, m_1, \ldots, m_n \in \N$.}
\end{quote}

As mentioned above, groups in $\mathcal{P}$ might not have the FCR property. Nevertheless, it might be useful to some readers for us to relate the types (OAS, BP and PA) for groups in $\mathcal{P}$ that occur in Theorem~\ref{thm:MainTheorem}, with the primitive types (HA, AS, AS$_{\rm{reg}}$, HS, HC, TW, SD, CD and PA) used in \cite{PraegerSchneider}. Table 7.1 and Figure 7.1 from  \cite{PraegerSchneider} are useful guides here.

Firstly, the results in this paper show that a group $G \in \mathcal{P}$ is never of type HA, HS, HC, TW, SD or CD. Type PA in this paper is the same as it is in \cite{PraegerSchneider}.

A group $G$ of type AS$_{\rm{reg}}$ has a minimum normal subgroup that is regular and therefore countable. So by Lemma~\ref{lemma:G_discrete_implies_M_a_infinite} point stabilisers in $G$ are all finite. Theorem 2.5 in \cite{smith:prim_digraphs} implies that a subdegree-finite primitive permutation group with more than one end always has infinite point stabilisers (see \cite[Appendix B]{moller_vonk} for a detailed proof). Hence a group $G \in \mathcal{P}$ of type AS$_{\rm{reg}}$ is discrete and must lie in OAS. We noted already that we could subdivide OAS further into OAS$_d$ and OAS$_{n}$; if we were to follow the conventions of \cite{PraegerSchneider} we should subdivide OAS$_d$ further into two cases, one in which the minimal normal subgroup is regular and the other in which it is not.

The class AS consists of groups with a minimal normal subgroup that is transitive, nonabelian, simple and non-regular. If a group $G \in \mathcal{P}$ has a minimal normal subgroup with these properties, then either $G$ has precisely one end (in which case $G$ lies in the class OAS) or $G$ has infinitely many ends and is contained in a box product structure (and is thus of type BP).

\section{Preliminaries}
\label{section:prelims}

Many of our preliminaries were covered in the introduction. We give the remainder here.

\subsection{Graphs and digraphs} \label{subsection:connectivity_of_graphs}

Let $V$ be a set, let $E$ be a set of two-element subsets of $V$, and let $A$ be a set of ordered pairs of distinct elements of $V$. We call the pair $(V, E)$ a {\em graph} on $V$, and the elements of $V$ and $E$ are respectively its {\em vertices} and {\em edges}. We call the pair $(V, A)$ a {\em digraph} on $V$, and the elements of $A$ are its {\em arcs}. If $\Gamma$ is a graph or digraph, we will write $V\Gamma$, $E\Gamma$ or $A\Gamma$ to refer to its vertex, edge or arc set.
If $a \in A$ then we denote by $o(a)$ (called the {\em origin vertex}) and $t(a)$ (called the {\em terminal vertex}) the vertices such that $a = (o(a), t(a))$. The arc $(t(a), o(a))$, called the {\em reversal} of $a$, is denoted by $\overline{a}$. Associated with a digraph $(V, A)$ is the graph $(V, \{\{o(a), t(a)\} : a \in A\})$, and associated with a graph $(V, E)$ is the digraph $(V, \{(u, v), (v, u) :  \{u,v\} \in E\})$. Thus, we can think of a graph as being a digraph in which there are two arcs (one in each direction) between any two adjacent vertices. 

If $\Gamma$ is a graph (resp. digraph) we say that two vertices $u,v$ are {\em adjacent} if there is an edge  (resp. an arc in any direction) between them. Notice that our graphs and digraphs are simple, without loops or multiple edges. A {\em path} in $\Gamma$ is a series of distinct vertices $v_0, v_1, \ldots, v_n$ such that $v_i, v_{i+1}$ are adjacent for all $0 \leq i < n$; the {\em length} of this path is $n$. A {\em cycle} of length $n+1$ is a series of vertices $v_0, v_1, \ldots, v_{n}, v_{n+1}$ such that $v_{n+1}=v_0$ and $v_0, v_1, \ldots, v_{n}$ is a path and $v_{n}, v_{n+1}$ are adjacent. A {\em directed cycle} in a digraph $(V,A)$ is a cycle $v_0, v_1, \ldots, v_{n}, v_{n+1}$ such that $(v_i, v_{i+1}) \in A$ for all $0 \leq i < n+1$. A {\em ray} is a one-way infinite path $v_0, v_1, v_2, \ldots$, and a {\em double ray} is a two-way infinite path $\ldots, v_{-2}, v_{-1}, v_0, v_1, v_2, \ldots$.

Two vertices are {\em connected} in $\Gamma$ if there exists a path between them. The {\em distance} between two connected vertices $u$ and $v$, denoted $d(u,v)$, is the length of the shortest path between them; the distance between two vertices that are not connected is infinite. We say $\Gamma$ is {\em connected} if any two vertices in $\Gamma$ are connected. 
A {\em tree} is a connected graph without cycles. As mentioned previously, we will typically think of a tree $T$ as being a digraph, with two arcs between any two adjacent vertices, one in each direction. Between any two distinct vertices $v, w$ in a tree, there is a unique path which we denote by $[v,w]_T$

The {\em valency} of a vertex $v \in V\Gamma$ is equal to the cardinality of the set of vertices adjacent to $v$. If all vertices in $\Gamma$ have the same valency $m$, we say that $\Gamma$ is {\em regular} or sometimes {\em $m$-regular}. If all vertices in $\Gamma$ have finite valency, we say that $\Gamma$ is locally finite.

If $T$ is a tree, there is a natural bipartition of $VT$ into sets $V_1$ and $V_2$, in which any two distinct vertices in $V_i$ are an even distance from each other, for $i = 1,2$. If all vertices in $V_1$ have valency $m_1$ and all vertices in $V_2$ have valency $m_2$, then we say that $T$ is an $(m_1, m_2)$-biregular tree.

The automorphism group of a graph or digraph $\Gamma$ is denoted $\aut(\Gamma)$. We say that $\Gamma$ is {\em vertex-transitive} (resp. {\em vertex-primitive}) if $\aut(\Gamma)$ acts transitively (resp. primitively) on the vertices of $\Gamma$. We make similar definitions for the edges or arcs of $\Gamma$. Sometimes we will say $\Gamma$ is transitive; when we do, we always mean vertex transitive.

In an infinite graph or digraph $\Gamma$, one can form an equivalence relation on the set of rays: two rays are related if there is a third ray that intersects them both infinitely often. For a locally finite graph or digraph $\Gamma$ the equivalence classes of this relation are called the {\em ends} of $\Gamma$. It is well-known that all connected, locally finite, vertex-transitive graphs have $0, 1, 2$ or $2^{\aleph_0}$ many ends. Representative examples for each type are as follows: the $3$-regular tree has infinitely many ends, a two-way infinite ``ladder'' has two ends, the graph $\Z \times \Z$ has one end and finite graphs have no ends.

A (possibly empty) set $W$ of vertices of a graph or digraph $\Gamma$ is called a {\em cut set} if
the induced (di)graph $\Gamma \setminus W$ is not connected. If $W$ is a cut set and consists of a single vertex, that vertex is called a {\em cut vertex}. We say that $\Gamma$ is {\em $k$-connected} if 
$|V\Gamma| > k$ and every vertex set $W$ satisfying $|W| < k$ is not a cut set. If there exists some integer $k$ such that $\Gamma$ is $k$-connected but not $(k+1)$-connected, then $\Gamma$ has {\em connectivity $k$}. Thus $\Gamma$ has connectivity zero if and only if $\Gamma$ is disconnected or consists of a single vertex. An infinite, connected and vertex transitive graph or digraph has connectivity one if and only every vertex is a cut vertex.\\

Suppose $\Gamma$ is a graph (resp.~digraph) with connectivity one. The {\em lobes} of $\Gamma$ are connected subgraphs (resp.~subdigraphs) that are maximal subject to the condition that they have connectivity at least two. Let $\mathbb{L}$ be the set of lobes of $\Gamma$, and for a vertex $\gamma \in V\Gamma$, let $\mathbb{L}(\gamma)$ be the set of lobes that contain $\gamma$.

The {\em block-cut-vertex tree} $T$ of $\Gamma$ has vertex set $VT := V\Gamma \cup \mathbb{L}$ and edge set $ET := \{\{\gamma, \Lambda\} : \gamma \in V\Gamma \text{ \& } \Lambda \in \mathbb{L}(\gamma)\}$. It is important to notice that each element in $\mathbb{L}$ is represented by a single vertex in $T$, and that $V\Gamma$ and $\mathbb{L}$ form the two parts of the natural bipartition of $T$. 

If every vertex in $V\Gamma$ lies in at least two lobes, and every lobe in $\Gamma$ contains at least two vertices, then any group $G \leq \aut \Gamma$ induces a faithful action on $T$. This observation will play an important role in our analysis.

Suppose $m \geq 2$ is a cardinal number and $\Lambda$ is a vertex-transitive digraph with at least three vertices and connectivity at least two. Recall from Section~\ref{intro:box_prods} the graph $\Gamma(\Lambda, m)$. This graph has connectivity one, and every vertex lies in precisely $m$ lobes and every lobe is isomorphic (as a digraph) to $\Lambda$. The block-cut-vertex tree of $\Gamma(\Lambda, m)$ is the $(|V\Lambda|, m)$-biregular tree.

\begin{example} In Figure~\ref{fig:Gamma_D8_S3}, an infinite digraph with connectivity one is pictured (left) with its block-cut-vertex tree (right). The infinite digraph is $\Gamma(K_4, 3)$, where $K_4$ denotes the complete digraph on $4$ vertices.
\end{example}

A locally finite, connected graph $\Theta$ has more than one end if and only if there exists a finite set $F$ of vertices such that $\Theta$ has two distinct components that both contain rays.

Recall that the Cartesian product of two graphs, $\Gamma$ and $\Sigma$, is denoted $\Gamma \square \Sigma$. The vertex set of $\Gamma \square \Sigma$ is the Cartesian product $V\Gamma \times V\Sigma$, with vertices $(\alpha_1, \alpha_2)$ and $(\beta_1, \beta_2)$ adjacent in $\Gamma \square \Sigma$ if and only if either (i) $\alpha_1 = \beta_1$ and $\alpha_2$ and  $\beta_2$ are adjacent in $\Sigma$; or (ii) $\alpha_2 = \beta_2$ and $\alpha_1$ and $\beta_1$ are adjacent in $\Gamma$. An equivalent definition of the edge relation is the following:  vertices $(\alpha_1, \alpha_2)$ and $(\beta_1, \beta_2)$ are adjacent in $\Gamma \square \Sigma$ if and only if $d_\Gamma(\alpha_1, \beta_1) + d_\Sigma(\alpha_2, \beta_2) = 1$.
It is well-known (and easy to prove) that the Cartesian product of two infinite, connected, locally finite graphs has precisely one end.\\

Now we consider the relationship between the ends of two connected, locally finite graphs that are quasi-isometric. Here we follow \cite[Section 3.1.1]{KronMoller}. Let $\Gamma$ be a connected and locally finite graph.
If $E$ is the set of ends of $\Gamma$, then one can turn $\Gamma \cup E$ into a compact topological space as follows. Call a subset $C \subseteq V\Gamma$ a {\em cut} if the number of vertices in $\Gamma$ that are adjacent to some vertex in $C$ is finite. If a cut $C$ contains a ray belonging to some end $\epsilon$, then we say that $\epsilon$ is {\em contained} in $C$. The set of ends contained in $C$ is denoted $E(C)$.
A basis for the topology on $\Gamma \cup E$ is then $\{C \cup E(C) : C \text{ is a cut}\}$. Under this topology we can think of $E$ as a topological space, called the {\em end space} of $\Gamma$.

Recall that two metric spaces $(X, d_X)$ and $(Y, d_Y)$ are {\em quasi-isometric} if there exists a function $f: X \rightarrow Y$ and constants $a \geq 1$ and $b \geq 0$ such that for all points $x, x' \in X$ and all points $y \in Y$, the following holds:
\begin{enumerate}
\item
	$\frac{1}{a}d_X(x, x') - \frac{b}{a} \leq d_Y(f(x), f(x')) \leq ad_X(x, x') + ab$
\item
	$d_Y(y, f(X)) \leq b$.
\end{enumerate}

We can consider $\Gamma$ to be a metric space under its distance metric.
Two locally finite connected graphs that are quasi-isometric have homeomorphic end spaces.

\subsection{Permutation groups} \label{prelims:PermGroups}

Let $\Omega$ be a set. We write $\sym(\Omega)$ to denote the group of all permutations of $\Omega$. If $G \leq \sym(\Omega)$ we say that $G$ is a {\em permutation group}. At times, we will write a permutation group as $(G, \Omega)$.

Let $G \leq \sym(\Omega)$ be nontrivial and transitive. The {\em degree} of $G$ is the cardinal $|\Omega|$. For $\alpha \in \Omega$ and $g \in G$ we denote the image of $\alpha$ under $g$ by $\alpha^g$, following the convention that our permutations act from the right. This notation extends to sets, so for a subset $\Phi \subseteq \Omega$ we have $\Phi^g = \{\alpha^g : \alpha \in \Phi\}$. The set $\alpha^G = \{\alpha^g : g \in G\}$ denotes an {\em orbit} of $G$, and if $\Omega$ consists of a single orbit we say that $G$ is {\em transitive}.
We define the {\em pointwise stabiliser} 
of $\Phi$ in $G$ by
\[G_{(\Phi)} := \{g \in G : \alpha^g = \alpha, \text{ for all } \alpha \in \Phi\},\]
and the {\em setwise stabiliser} of $\Phi$ in $G$ by
\[G_{\{\Phi\}} := \{g \in G : \alpha^g \in \Phi \iff \alpha \in \Phi\}.\]
In the special case where $\Phi$ contains only a single element $\alpha$, we write $G_\alpha$ instead of $G_{(\{\alpha\})}$ and call $G_\alpha$ the {\em stabiliser of $\alpha$ in $G$.} A permutation group whose stabilisers are all trivial is said to be {\em semi-regular}; a transitive and semi-regular group is said to be {\em regular}.
The orbits of any point stabiliser $G_\alpha$ are called {\em suborbits} of $G$, and the cardinality of any suborbit is called a {\em subdegree} of $G$. A suborbit of cardinality one is called {\em trivial}. If $\beta^{G_{\alpha}}$ is any suborbit, and $\alpha^g = \beta$, then the {\em pair} of $\beta^{G_{\alpha}}$ is $(\alpha^{g^{-1}})^{G_\alpha}$.

If $H \leq \sym(\Lambda)$ is another permutation group, a {\em permutation isomorphism} from $G$ to $H$ is a pair $(\phi, \theta)$ such that $\phi: G \rightarrow H$ is an isomorphism and $\theta : \Omega \rightarrow \Lambda$ is a bijection with $\theta(\alpha^g) = \theta(\alpha)^{\phi(g)}$ for all $\alpha \in \Omega$ and all $g \in G$.
If we wish to restrict the domain of a function $\phi$ to some subset $Y$, then we write $\phi \big|_{Y}$. We extend this notation to permutation groups in the following way: if $G$ leaves invariant some subset $\Lambda \subseteq \Omega$ then the action of $G$ on $\Lambda$ induces a subgroup of $\sym(\Lambda)$, and we denote this group by $G \big|_{\Lambda}$.

Recall that $G$ is primitive if it preserves no equivalence relation on $\Omega$ except for the two obvious relations: the {\em trivial relation} (in which everything in $\Omega$ is related only to itself), and the {\em universal relation} (in which everything in $\Omega$ is related to everything else). It is well-known that the following are equivalent:
\begin{enumerate}
\item
	$G$ is primitive on $\Omega$;
\item
	For all $\alpha \in \Omega$ the stabiliser $G_\alpha$ is a maximal subgroup of $G$;
\item
	$\Omega$ is the only nontrivial block of $G$ (where a {\em nontrivial block} of $G$ is a subset $\Delta \subseteq \Omega$ containing at least two elements that satisfies the condition that for all $g \in G$ either $\Delta \cap \Delta^g = \Delta$ or $\emptyset$).
\item \label{item:higman}
	For all distinct $\alpha, \beta \in \Omega$ the digraph $(\Omega, \{\alpha, \beta\}^G)$ is connected.
\end{enumerate}

The orbits of a normal subgroup $N \unlhd G$ form a $G$-invariant equivalence relation. Therefore, all nontrivial normal subgroups of primitive permutation groups are transitive.

Condition~\ref{item:higman} above is called Higman's criterion (see \cite[1.12]{d_g_higman}). From it, one can easily deduce the following: if $G$ is primitive and any suborbit and its pair are finite, then $\Omega$ is countable and $G$ is subdegree-finite.\\

Let us now turn our attention to wreath products. In this paper, we will always be taking the wreath product of an arbitrary permutation group $G \leq \sym(\Omega)$ and a finite permutation group $H \leq S_m$, for some $m \geq 2$, and our notation will reflect this.
Recall that the wreath product of $G$ and $H$ has two natural actions: one on $\Omega \times \{1, \ldots, m\}$ that we call the {\em imprimitive action} and the other on $\Omega^{m}$ that we call the {\em product action}. We denote the permutation group resulting from the imprimitive action by $G \wr H$, and the permutation group resulting from the product action by $G \Wr H$. The two actions are given explicitly below.

Fix $(g_1, \ldots, g_m; h)$ in the wreath product of $G$ and $H$, fix $(\alpha, i) \in \Omega \times \{1, \ldots, m\}$ and fix $(\alpha_1, \ldots, \alpha_m) \in \Omega^m$. Then:
\begin{itemize}
\item
	in $G \wr H$ we have $(\alpha, i)^{(g_1, \ldots, g_m; h)} = (\alpha^{g_i}, i^h)$
\item
	in $G \Wr H$ we have $(\alpha_1, \ldots, \alpha_m)^{(g_1, \ldots, g_m; h)} = \left (\alpha^{g_{1^{h^{-1}}}}_{1^{h^{-1}}}, \ldots, \alpha^{g_{m^{h^{-1}}}}_{m^{h^{-1}}} \right )$
\end{itemize}\\

\subsection{Topological groups}
\label{subsection:topo_groups}

Let $G$ be a topological group. It is {\em totally disconnected} and {\em locally compact} if it is both of these things as a topological space. In this paper we don't explicitly assume that our topological groups are Hausdorff, but totally disconnected groups are always Hausdorff.

A subgroup $H \leq G$ is {\em cocompact} if the quotient space $G/H$ is compact under the quotient topology. ßA group is {\em topologically simple} if it contains no proper nontrivial closed normal subgroups. Similarly, a subgroup $H \leq G$ is {\em topologically characteristic} if 
$H$ is a closed subgroup which is preserved by every topological group automorphism of $G$.
If there exists a compact subset $S \subseteq G$ such that $G$ can be generated (as an abstract group) by $S$, then we say that $G$ is {\em compactly generated}.

In the introduction we defined the permutation topology for $\sym(\Omega)$, but usually it is defined more generally. We do this here. 
Let $G$ be an abstract group acting on a set $\Omega$. This action gives rise to a topology on $G$. The {\em permutation topology} takes as a neighbourhood basis of the identity the pointwise stabilisers in $G$ of finite subsets of $\Omega$. That is, a neighbourhood basis of the identity is,
\[\{G_{(\Phi)} : \Phi \subseteq \Omega, \, |\Phi| \in \mathbb{N}\}.\]

It is easy to see that the following are true for $G$ as a topological group under this topology: (i) a subgroup $H \leq G$ is open if and only if $H$ contains $G_{(\Phi)}$ for some finite $\Phi \subseteq \Omega$; (ii) $G$ is Hausdorff if and only if the action of $G$ is faithful; and (iii) $G$ is totally disconnected if and only if $G$ is faithful.
We say that $G$ is {\em closed} if its image in $\sym(\Omega)$ is closed under the permutation topology.

Cocompactness has the following natural interpretation.

\begin{proposition}[{\cite[Lemma 7.5]{moller:tdlc_via_graphs_and_perms}}] \label{Proposition_A}
Let $G$ be a topological group. If $G$ is acting transitively on a set $\Omega$, and the stabiliser of a point $G_\alpha$ for $\alpha \in \Omega$ is compact and open (for example $G$ could be a closed subdegree-finite permutation group), then $H \leq G$ is cocompact if and only if it has finitely many orbits on $\Omega$.
\end{proposition}

Let $G$ be a topological group. Recall that a connected graph $\Gamma$ of bounded valency is said to be a {\em Cayley--Abels graph} (also called a {\em rough Cayley graph}) for $G$ if $G$ acts vertex transitively on $\Gamma$ and the stabilisers of vertices are compact and open subgroups of $G$ (see \cite[Definition 19.6]{NewDirectionsBook} and \cite[Definition 2.1]{KronMoller}). 

\begin{theorem}[{\cite[Theorem 2.2]{KronMoller}}] Let $G$ be a tdlc group. Then $G$ has a Cayley--Abels graph if and only if $G$ is compactly generated.
\end{theorem}

Recall from Section~\ref{subsection:connectivity_of_graphs} that two connected, locally finite graphs that are quasi-isometric have the same end space (up to homeomorphism). 

\begin{theorem}[{\cite[Theorem 2.7]{KronMoller}}] Let $G$ be a compactly generated tdlc group. Any two Cayley--Abels graphs of $G$ are quasi-isometric.
\end{theorem}

Thus the {\em space of ends of a compactly generated tdlc group} is defined to be the space of ends of one of its Cayley--Abels graphs.

\subsection{Permutation groups as topological groups}
\label{section:PermGpsAsTopoGroups}

Readers interested in the interplay between permutation groups and topological groups are referred to \cite{moller:tdlc_via_graphs_and_perms}.

Recall that a permutation group $G \leq \sym(\Omega)$ is a topological group under the permutation topology. One can easily show that $G$ is closed if and only if some point stabiliser $G_\alpha$ is closed, and
a subgroup $H$ of $G$ is open in $G$ if and only if it contains the pointwise stabiliser $G_{(\Phi)}$ for some finite subset $\Phi \subseteq \Omega$. Thus $G$ is discrete if and only if the pointwise stabiliser in $G$ of some finite set is trivial and
 if $G$ is subdegree-finite, then $G$ is discrete if and only if all point stabilisers in $G$ are finite.

If $G$ is subdegree-finite and closed, then it is not difficult to show that any point stabiliser $G_\alpha$ is compact (see \cite[Proof of Lemma 1]{Woess} for example). Since $G_\alpha$ is open by definition, it follows then that $G$ is tdlc with compact open point stabilisers. Hence if $G$ is closed, transitive and has a connected and locally finite orbital graph $\Gamma$, then  $\Gamma$ is a Cayley--Abels graph for $G$ and $G$ is compactly generated.

\begin{remark} \label{remark:CountableInterior}
Suppose $\Omega$ is a countable set. For any group $G \leq \sym(\Omega)$ it is well-known that $G$ contains a countable subgroup $H$ such that, in the permutation topology on $\sym(\Omega)$, the closure of $H$ is equal to the closure of $G$. Indeed, note that the set of all finite tuples of elements in $\Omega$ is countable, and 
so in particular the set of all pairs of tuples that lie in the same $G$ orbit is countable. For each pair $(\underline{\alpha}, \underline{\beta})$ of finite tuples that lie in the same $G$ orbit, choose an element in $G$ that maps $\underline{\alpha}$ to $\underline{\beta}$, and let $H$ be the group generated by these elements. Now $H$ is countable because it is countably generated, and $\overline{G} = \overline{H}$ because for all finite tuples $\underline{\alpha}$ of $\Omega$, the orbits $\underline{\alpha}^G$ and $\underline{\alpha}^H$ are equal.
\end{remark}

The following is widely known, but we include the proof for completeness.

\begin{lemma} \label{lemma:NormalizersAreClosed} Suppose $\Omega$ is countable and $H \leq G \leq \sym(\Omega)$, with $G$ closed. Then for all $g \in G$ the centralizer $C_G(g)$ is closed, $C_G(H)$  is closed, and if $H$ is closed then $N_{G}(H)$ is closed.
\end{lemma}
\begin{proof} 
Enumerate the elements of $\Omega$ as $\{\alpha_1, \alpha_2, \ldots\}$. Suppose $(g_n)$ is some convergent sequence of permutations, with each $g_i \in G$. Since $G$ is closed, the limit $\ell$ of the sequence $(g_n)$ lies in $G$. 
If $g \in G$ and each $g_i$ is in $C_G(g)$, then $g_i^{-1}g^{-1} g_i g= 1$ and therefore $\ell^{-1} g^{-1} \ell g = 1$. Hence $C_G(g)$ is closed. Similarly, if each $g_i$ is in $C_G(H)$ and $h \in H$, then $\ell^{-1}h^{-1}\ell h = 1$. Hence $C_G(H)$ is closed. Finally, suppose that $H$ is closed, each $g_i$ lies in $N_G(H)$, and $h \in H$. Then $g_i^{-1}h^{-1} g_i h \in H$ for all $i$, and therefore $\ell^{-1}h^{-1} \ell h \in H$. Hence $N_G(H)$ is closed.
\end{proof}

The following useful theorem of David M.~Evans does not require the continuum hypothesis.

\begin{theorem}[{\cite[Theorem 1.1]{Evans87}}] \label{theorem:Evans} Suppose $\Omega$ is countable. If $G$ and $H$ are closed subgroups of $\sym(\Omega)$ with $H \leq G$, then either $|G:H| = 2^{\aleph_0}$ or $H$ contains the pointwise stabiliser in $G$ of some finite subset of $\Omega$. \qed
\end{theorem}

\subsection{The quasi-center of a permutation group}

The {\em quasi-centre} plays an important role in determining the structure of tdlc groups. Here we give an interpretation for permutation groups, and prove some known properties using short permutational arguments. The advantage for us of this approach is that we will obtain a permutational statement that doesn't require topological conditions. This is something we will need later in the paper.

In \cite{BurgerMozes}, the {\em quasi-center} of a topological group $H$ is defined as $QZ(H) := \{ h \in H : \text{$C_H(h)$ is open}  \}$, where $C_H(h)$ denotes the centralizer in $H$ of $h$.
There is a natural interpretation of this definition for permutation groups: for a countable set $\Omega$ and a permutation group $G \leq \sym(\Omega)$, the quasi-center of $G$ is,
\[QZ(G) := \{g \in G : G_{(\Phi)} \leq C_G(g) \text{ for some finite } \Phi \subset \Omega \}.\]

The following is known for tdlc groups.

\begin{proposition} \label{prop:QZContainsAllCountableNormalSubgroups} Suppose $\Omega$ is countable and $G \leq \sym(\Omega)$. If $G$ is closed, then $QZ(G)$ is the union of all countable normal subgroups of $G$.
\end{proposition}

\begin{proof} Suppose $G$ is closed. By Lemma~\ref{lemma:NormalizersAreClosed}, $C_G(g)$ is closed for all $g \in G$. Hence, by Theorem~\ref{theorem:Evans},
\[g \in QZ(G) \iff |G : C_G(g)| \leq \aleph_0.\]

Now $G$ acts on itself by conjugation. The stabiliser of $g \in G$ under this action is $C_G(g)$, and the orbit $g^G$ has cardinality $|G : C_G(g)|$ and is contained in the normal subgroup $\langle g^G \rangle \unlhd G$.

If $g \in G$ lies in some countable normal subgroup $N$ of $G$, then we have $|G : C_G(g)| = |g^G| \leq |N| \leq \aleph_0$, and hence $g \in QZ(G)$. On the other hand, if $g \in QZ(G)$ then 
by definition there is a finite subset $\Phi$ of the countable set $\Omega$ such that $G_{(\Phi)} \leq C_G(g)$. Hence for $g \in QZ(G)$ we have $\aleph_0 \geq |G:G_{(\Phi)}| = |G : C_G(g)| |C_G(g) : G_{(\Phi)}| \geq |G :C_G(g)| = |g^G|$.
Since $g^G$ is countable we know $\langle g^G \rangle$ is countable.
Hence $g$ is contained in the countable normal subgroup $\langle g^G \rangle \unlhd G$.
\end{proof}

For tdlc groups $H$, it is well known that the quasi-center $QZ(H)$ is a topologically characteristic subgroup of $H$. For permutation groups, we have the following.

\begin{proposition} \label{prop:QZ_topo_characteristic} Suppose $\Omega$ is countable and $G \leq \sym(\Omega)$. Then $QZ(G)$ is a normal subgroup of $N_{\sym(\Omega)}(G)$.
\end{proposition} 
\begin{proof}
Suppose $g \in QZ(G)$ and $h \in N_{\sym(\Omega)}(G)$. By definition, there exists a finite set $\Phi \subseteq \Omega$ such that $G_{(\Phi)} \leq C_G(g)$.
Now
$g^h \in G$ and $G_{(\Phi^h)} = h^{-1} G_{(\Phi)} h \leq h^{-1} C_G(g) h = C_G(g^h)$. Hence $g^h \in QZ(G)$.
\end{proof}

The following results can be found in \cite{BurgerMozes}. The proofs given in \cite{BurgerMozes} are topological. For completeness here we give short permutational arguments.

\begin{lemma}[{\cite[1.3.2--1.3.6]{BurgerMozes}}] \label{lemma:BurgerMozesInterpretation} Let $\Gamma$ be a connected locally finite graph and $H, G \leq \aut \Gamma$. If $H$ has only finitely many orbits on the vertex set $V\Gamma$ then,
\begin{enumerate}
\item \label{BurgerMozesInterpretation:case1}
	$H$ contains a finitely generated subgroup which also has only finitely many orbits on $V\Gamma$;
\item \label{BurgerMozesInterpretation:case2}
	the centralizer $C_{\aut \Gamma}(H)$ has no infinite point stabiliser;
\item \label{BurgerMozesInterpretation:caseN}
	if $H$ has no infinite point stabiliser, then  $N_{\aut \Gamma}(H)$ has no infinite point stabiliser; and
\item \label{BurgerMozesInterpretation:case3}
	if $[H, G]$ has no infinite point stabiliser, then $G$ has no infinite point stabiliser;
\item \label{BurgerMozesInterpretation:caseQZ}
	if $G$ has an infinite point stabiliser, then $QZ(G)$ has infinitely many orbits on $V\Gamma$.
\end{enumerate} 
\end{lemma}

\begin{proof} Suppose $H$ has only finitely many orbits on the vertices of $\Gamma$.
Part (\ref{BurgerMozesInterpretation:case1}) of the lemma is obviously true. 

Let $\Phi$ be a set of orbit representatives for $H$ on the vertices of $\Gamma$. Notice that $\Phi$ is finite.
Let $C:=C_{\aut \Gamma}(H)$.
The pointwise stabiliser $C_{(\Phi)}$ commutes with $H$ and so it fixes $V = \Phi^H$ pointwise; hence $C_{(\Phi)}$ is trivial.
Since $\Phi$ is finite and $\Gamma$ is locally finite and connected, we know that for all $\alpha \in V\Gamma$ the orbit $\Phi^{C_\alpha}$ must be finite. Hence $|C_\alpha : C_\alpha \cap C_{(\Phi)}|$ is finite. Since $C_{(\Phi)}$ is trivial, $C_\alpha$ is finite.
Hence part (\ref{BurgerMozesInterpretation:case2}) is true.

Note that if $K \leq \aut(\Gamma)$ has only finitely many orbits on $V\Gamma$, then (\ref{BurgerMozesInterpretation:case2}) implies that $G$ has no infinite point stabilisers whenever the following condition holds: $G_{(\Delta)} \leq C_{\aut \Gamma}(K)$ for some finite subset $\Delta \subseteq V\Gamma$. 

For now we ignore part (\ref{BurgerMozesInterpretation:caseN}) 
and instead prove part~(\ref{BurgerMozesInterpretation:case3}).
Let $S$ be some finite subset of $H$ such that $\langle S \rangle$ has only finitely many orbits, and if necessary extend $S$ so that it is symmetric (i.e. for every $s\in S, s^{-1} \in S$) and contains the identity. Suppose that all point stabilisers in $[H, G]$ are finite.
There exists a finite set $\Delta' \subseteq V\Gamma$ such that $[H, G]_{(\Delta')}$ is trivial. Let $\Delta := \{ \delta^s : s \in S, \delta \in \Delta'\}$.
For any $s \in S$, elements in $[s, G_{(\Delta)}]$ fix $\Delta'$ pointwise and therefore lie in $[H,G]_{(\Delta')}$, which is trivial.
Hence $G_{(\Delta)} \leq C_{\aut \Gamma}(\langle S \rangle)$. As noted earlier in the proof, this implies that $G$ has no infinite point stabilisers.
Hence part (\ref{BurgerMozesInterpretation:case3}) is true.

We now prove part~(\ref{BurgerMozesInterpretation:caseN}). Suppose $H$ has no infinite point stabiliser. If $N := N_{\aut \Gamma}(H)$, then $[H,N] \leq H$ has no infinite point stabiliser, and so by (\ref{BurgerMozesInterpretation:case3}) it follows that $N$ has no infinite point stabiliser. Hence (\ref{BurgerMozesInterpretation:caseN}) is true.

Finally we prove part~(\ref{BurgerMozesInterpretation:caseQZ}). Suppose $QZ(G)$ has only finitely many orbits on $V\Gamma$. By (\ref{BurgerMozesInterpretation:case1}), there is a finite subset $S := \{s_1, \ldots, s_m\} \subseteq QZ(G)$ such that $\langle S \rangle$ has only finitely many orbits on $V\Gamma$. Each $s_i \in S$ lies in $QZ(G)$, so there exist finite sets of vertices $\Delta_i \subseteq V\Gamma$ such that $G_{(\Delta_i)} \leq C_{\aut \Gamma}(s_i)$ for all $i$. The set $\Delta := \bigcup_{i=1}^m \Delta_i$ is finite and $G_{(\Delta)} \leq C_{\aut \Gamma}(\langle S \rangle)$. As noted earlier in the proof, this implies that $G$ has no infinite point stabiliser. Hence (\ref{BurgerMozesInterpretation:caseQZ}) is true.
\end{proof}

We conclude with a small but useful lemma.

\begin{lemma} \label{lemma:TrivialQuasiCenter}
 Let $\Gamma$ be a connected locally finite graph and suppose $G \leq \aut \Gamma$ has an infinite point stabiliser. If every nontrivial normal subgroup of $G$ has finitely many orbits on $V\Gamma$, then the quasi-center of any $N \unlhd G$ is trivial. Moreover, if $N$ is closed in $\aut \Gamma$, then $N$ has no nontrivial countable normal subgroups.
\end{lemma}

\begin{proof} Suppose every nontrivial normal subgroup of $G$ has finitely many orbits on $V\Gamma$, and let $N$ be a nontrivial normal subgroup of $G$. By Lemma~\ref{lemma:BurgerMozesInterpretation}(\ref{BurgerMozesInterpretation:caseN}) it follows that $N$ has an infinite point stabiliser. Hence by Lemma~\ref{lemma:BurgerMozesInterpretation}(\ref{BurgerMozesInterpretation:caseQZ}), $QZ(N)$ has infinitely many orbits on $V\Gamma$.

By Proposition~\ref{prop:QZ_topo_characteristic} (which applies because $V\Gamma$ is countable), $QZ(N) \unlhd N_{\sym(V\Gamma)}(N)$. Since $QZ(N) \leq N \leq G \leq N_{\sym(V\Gamma)}(N)$, it follows that $QZ(N) \unlhd G$. The quasi-center $QZ(N)$ is thus a normal subgroup of $G$ with infinitely many orbits and consequently $QZ(N)$ must be trivial.

If $N$ is closed in $\aut \Gamma$ then it is closed in $\sym(V\Gamma)$, and therefore by Proposition~\ref{prop:QZContainsAllCountableNormalSubgroups}, $QZ(N)$ (which we have shown to be trivial) is the union of all countable normal subgroups of $N$.
\end{proof}

\subsection{Box products and fibrelobes}
\label{section:boxproduct}

In \cite{smith:product} a new product for permutation groups was introduced, called the {\em box product}, denoted by the symbol $\boxtimes$. This new product will play an important role in our classification, primarily because the box product and the unrestricted wreath product in its product action preserve primitivity under strikingly similar conditions.

\begin{theorem}[{\cite[Theorem 26]{smith:product}}] \label{thm:box_prim}
Given permutation groups $G_1$ and $G_2$ of nontrivial degree, the permutation group $G_1 \boxtimes G_2$ is primitive if and only if $G_1$ is primitive but not regular, and $G_2$ is transitive.
\end{theorem}

Compare this with the well-known primitivity conditions for the unrestricted wreath product in its product action (see \cite[Lemma 2.7A]{dixon&mortimer} for example).

\begin{proposition}
\label{prop:wr_prim}
Given permutation groups $G_1$ and $G_2$ of nontrivial degree, the permutation group $G_1 \Wr G_2$ is primitive in its product action if and only if $G_1$ is primitive but not regular, and $G_2$ is transitive and finite.
\end{proposition}

Although $\boxtimes$ and $\Wr$ preserve primitivity under similar conditions, they are thoroughly dissimilar. For example, $S_3 \boxtimes S_2$ has cardinality $2^{\aleph_0}$, while $S_3 \Wr S_2$ is finite.\\

We give a brief description of the construction of the box product here and summarise some of its relevant properties; for more details see \cite{smith:product}.

\begin{definition}
Suppose $G_1 \leq \sym(X_1)$ and $G_2 \leq \sym(X_2)$ are  permutation groups of nontrivial degree, where $X_1$ and $X_2$ are disjoint and may be finite or infinite. The box product $G_1 \boxtimes G_2$ is constructed by colouring the arcs of a biregular tree. 

Let $T$ denote the $(|X_1|, |X_2|)$-biregular tree, with $V_{1}$ and $V_2$ denoting the two parts of the natural bipartition of the vertices of $T$, such that all vertices in $V_i$ have valency $|X_i|$, for $i=1,2$. Recall that between any two adjacent vertices of $T$ there are two arcs, one in each direction.
Let $A(v)$ (resp. $\overline{A}(v)$) denote the set of arcs of $T$ whose origin (resp. terminal) vertex is $v$, and let $B(v)$ denote the set of vertices at distance one from $v$ in $T$.

An {\em arc colouring} is a function $\mathcal{L} : AT \rightarrow X_1 \cup X_2$ such that for all $i \in \{1,2\}$ and for all $v \in V_i$, the restriction $\mathcal{L} \big|_{A(v)} : A(v) \rightarrow X_i$ is a bijection. An arc colouring is {\em legal} if the map $\mathcal{L} \big|_{\overline{A}(v)}$ is constant for all $v \in VT$. We will call a legal arc colouring a {\em legal colouring}.
It is a simple matter to check that one may always find a legal colouring $\mathcal{L}$ for the biregular tree $T$.
For an arc colouring $\mathcal{L}$, a vertex $v \in VT$ and an automorphism $g \in \aut T$ that setwise stabilises $V_1$ and $V_2$, we define,
\[\theta(g, v, \mathcal{L}):= \mathcal{L} \big|^{-1}_{A(v)} \circ g \big|_{A(v)} \circ \mathcal{L} \big|_{A(gv)}.\]
Our convention in this paper is that groups act on the right, and (to simplify our notation) this applies also to arc colourings and to $\theta(g, v, \mathcal{L})$. Note that in \cite{smith:product} it is the convention that groups and colourings act on the left, and so the definition of $\theta$ given in \cite{smith:product} is the reverse of that given here.

Choose a legal colouring $\mathcal{L}$.
For $i = 1,2$ define,
\[
U_i := \{g \in (\aut T)_{\{V_i\}} : \theta(g, v, \mathcal{L}) \in G_i, \quad \forall v \in V_i
\},
\]
and set $\mathcal{U}_{\mathcal{L}}(G_1, G_2) := U_1 \cap U_2$. The definition of $\mathcal{U}_{\mathcal{L}}(G_1, G_2)$ is a generalisation of the Burger--Mozes universal group (see \cite{BurgerMozes}).
\end{definition}

It is shown in \cite[Proposition 11]{smith:product} that a different choice of legal colouring gives rise to a permutation group that is conjugate  to $\mathcal{U}_{\mathcal{L}}(G_1, G_2)$ in $\aut T$. In fact something slightly stronger is established in the proof, which we state here.

\begin{proposition}[{\cite[Proposition 11 \& proof]{smith:product}}] \label{prop:col_doesnt_matter}
If $\mathcal{L}$ and $\mathcal{L'}$ are legal colourings, then $\mathcal{U}_{\mathcal{L}}(G_1, G_2)$ and $\mathcal{U}_{\mathcal{L'}}(G_1, G_2)$ are conjugate in the setwise stabiliser $(\aut T)_{\{V_1\}}$.
\end{proposition}

We shall say that a group $H \leq \aut T$ is {\em locally-$(G_1, G_2)$} if $H$ fixes setwise the parts $V_1$ and $V_2$, and for all vertices $v$ of $T$ the stabiliser $H_v$ induces a permutation group on $B(v)$ that is permutationally isomorphic to $G_1$ if $v \in V_1$, and $G_2$ if $v \in V_2$.

\begin{proposition}[{\cite[Proposition 12]{smith:product}}] \label{prop:locally_MN}
If $G_1$ and $G_2$ are transitive and $H \leq \aut T$ is locally-$(G_1, G_2)$ then $H \leq \mathcal{U}_{\mathcal{L'}}(G_1, G_2)$ for some legal colouring $\mathcal{L'}$.
\end{proposition}

\begin{definition}
The {\em box product} $G_1 \boxtimes_{\mathcal{L}} G_2$ is defined to be the subgroup of $\sym(V_2)$ induced by the action of $\mathcal{U}_{\mathcal{L}}(G_1, G_2)$ on $V_2$.
\end{definition}

In light of Proposition~\ref{prop:locally_MN}, we will usually write  $\mathcal{U}(G_1, G_2)$ and $G_1 \boxtimes G_2$ instead of $\mathcal{U}_{\mathcal{L}}(G_1, G_2)$ and $G_1 \boxtimes_{\mathcal{L}} G_2$. Note that $\mathcal{U}(G_1, G_2)$ acts faithfully on $V_2$, and the action of $G_1 \boxtimes G_2$ on $V_2$ induces a faithful action of $G_1 \boxtimes G_2$ on $T$. As noted in \cite[Remark 4]{smith:product}, it is easy to see that  $\mathcal{U}(G_1, G_2)$ and $G_1 \boxtimes G_2$ are isomorphic as topological groups, although they of course are not isomorphic as permutation groups.

\begin{theorem}[{\cite[Theorem 1, Remark 5, Theorem 26 \& Proposition 27]{smith:product}}]
\label{thm:box_product_summary}
Suppose that $G_1 \leq \sym(X_1)$ and $G_2 \leq \sym(X_2)$ are permutation groups of nontrivial degree, with at least one of $G_1$ or $G_2$ nontrivial. Then the following hold:
\begin{enumerate}
\item
	$\mathcal{U}(G_1, G_2) \leq \aut T$ is locally-$(G_1, G_2)$;
\item
	If $G_1$ and $G_2$ are closed, then $\mathcal{U}(G_1, G_2)$ is closed (and hence $G_1 \boxtimes G_2$ is also closed);
\item
	If $G_1$ and $G_2$ are generated by point stabilisers, then $\mathcal{U}(G_1, G_2)$ is simple if and only if $G_1$ or $G_2$ is transitive;
\item
	$\mathcal{U}(G_1, G_2)$ is discrete if and only if $G_1$ and $G_2$ are semiregular;
\item
	If $H_1 \leq G_1$ and $H_2 \leq G_2$, then $\mathcal{U}(H_1, H_2) \leq \mathcal{U}(G_1, G_2)$;
\item
	$G_1 \boxtimes G_2$ is transitive if and only if $G_1$ is transitive; and
\item
	$G_1 \boxtimes G_2$ is subdegree-finite if and only if $G_1$ is subdegree-finite and all orbits of $G_2$ are finite.
\end{enumerate}
\end{theorem}

\begin{definition} The tree $T$ is called the {\em structure tree} of $G_1 \boxtimes G_2$ and for each vertex $v \in V_1$, the set $B(v)$ is called a {\em lobe} of $G_1 \boxtimes G_2$. The elements of $V_2$ are called the {\em points} of $G_1 \boxtimes G_2$.
The lobes of a box product and the fibres of a wreath product in its product action play similar roles.
\end{definition}

\begin{remark} \label{rem:intuitive_box}
There is a natural and intuitive way of picturing the box product of two nontrivial transitive permutation groups, $G_1$ and $G_2$.

Let $n_i$ be the (not necessarily finite) degree of $G_i$, for $i = 1, 2$. 
Recall the definition of the tree-like graph $\Gamma(K_{n_1}, n_2)$ from Section~\ref{intro:box_prods}.
Write $\Theta := \Gamma(K_{n_1}, n_2)$, so $\Theta$ is a connectivity one graph whose lobes are complete graphs. For $\gamma \in V\Theta$ define
 $\mathbb{L}(\gamma)$ to be the set of lobes of $\Theta$ containing $\gamma$.
 Let $\mathcal{H}$ be the set of subgroups $H$ of $\aut \Theta$ that satisfy both of the following conditions:
\begin{enumerate}
\item \label{case2:rem:intuitive_box}
	for all lobes $\Lambda$ of $\Theta$ the induced permutation group $H_{\{V\Lambda\}} \big|_{V\Lambda}$ is permutation isomorphic to $G_1$; and
\item \label{case1:rem:intuitive_box}
	for all vertices $\gamma$ in $\Theta$ we have that $H_\gamma \big|_{\mathbb{L(\gamma)}}$ is permutation isomorphic to $G_2$.
\end{enumerate}
Up to permutation isomorphism, $G_1 \boxtimes G_2$ is the unique maximal element in $\mathcal{H}$. Moreover, each lobe of $G_1 \boxtimes G_2$ is the vertex set of a lobe of $\Gamma(K_{n_1}, n_2)$, and the structure tree of $G_1 \boxtimes G_2$ is the block-cut-vertex tree of $\Gamma(K_{n_1}, n_2)$.
\end{remark}

The above remark is a formalisation of the intuitive sense of the box product given in the introduction, Section~\ref{intro:box_prods}.

\begin{example} Let $D_{8}$ be the dihedral group acting on four points and consider $G := D_8 \boxtimes S_3$. The group $G$ acts as a group of automorphisms of the graph $\Gamma(K_4, 3)$ pictured in Figure~\ref{fig:Gamma_D8_S3} (left) alongside the structure tree of $G$ (right). Notice that $G$ is the largest subgroup of $\aut \Gamma(K_4, 3)$ that satisfies conditions (\ref{case2:rem:intuitive_box}) and (\ref{case1:rem:intuitive_box}) of Remark~\ref{rem:intuitive_box}.

One could use the same graph, $\Gamma(K_4, 3)$, to understand $D_8 \boxtimes C_3$ and $S_8 \boxtimes S_3$, for example.
\end{example} 

\begin{figure}[ht!]
    \centering
    \subfloat{
	\includegraphics[width=.4\linewidth]{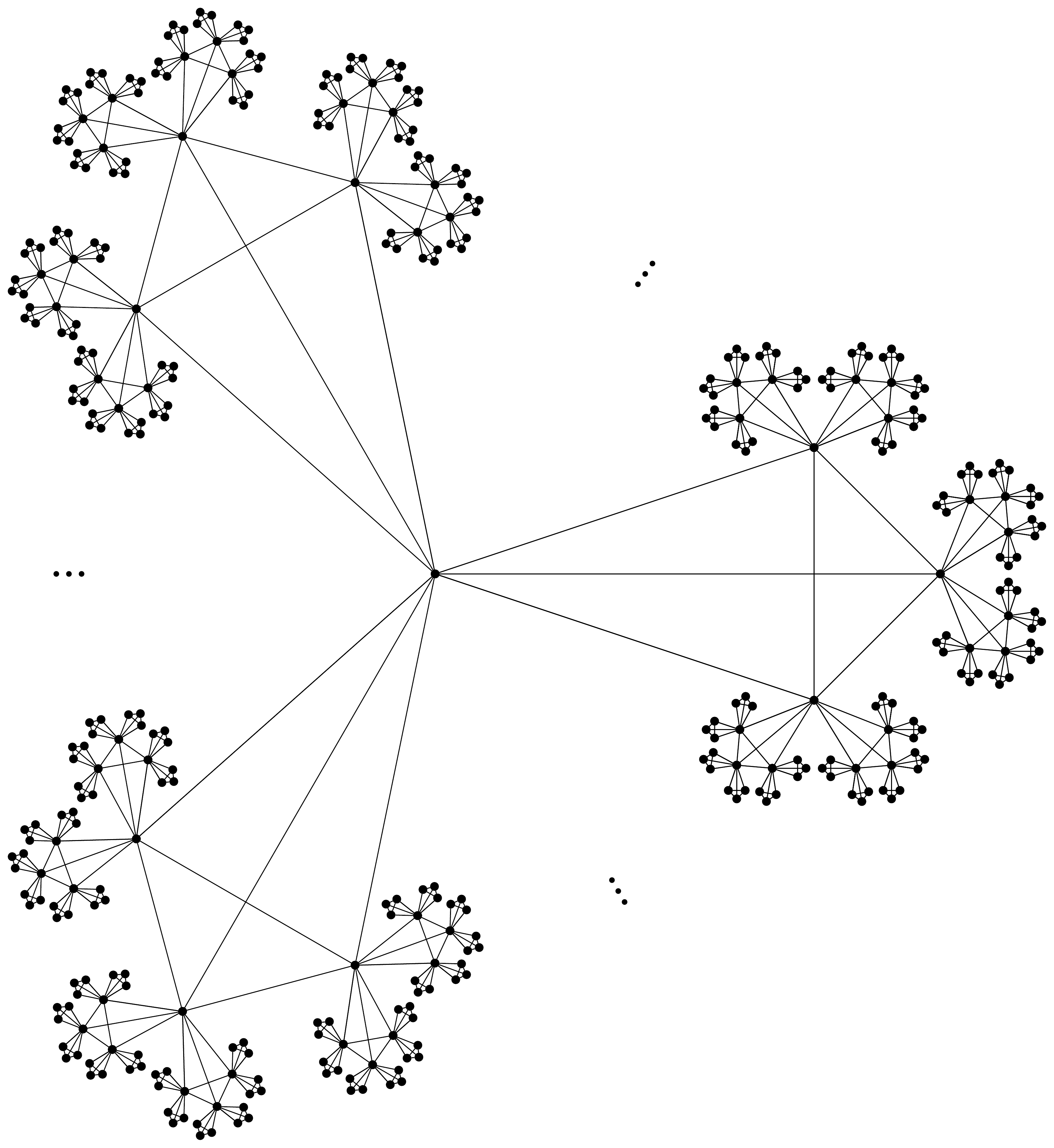}
    }
    \hspace{3em}
    \subfloat{
	\includegraphics[width=.4\linewidth]{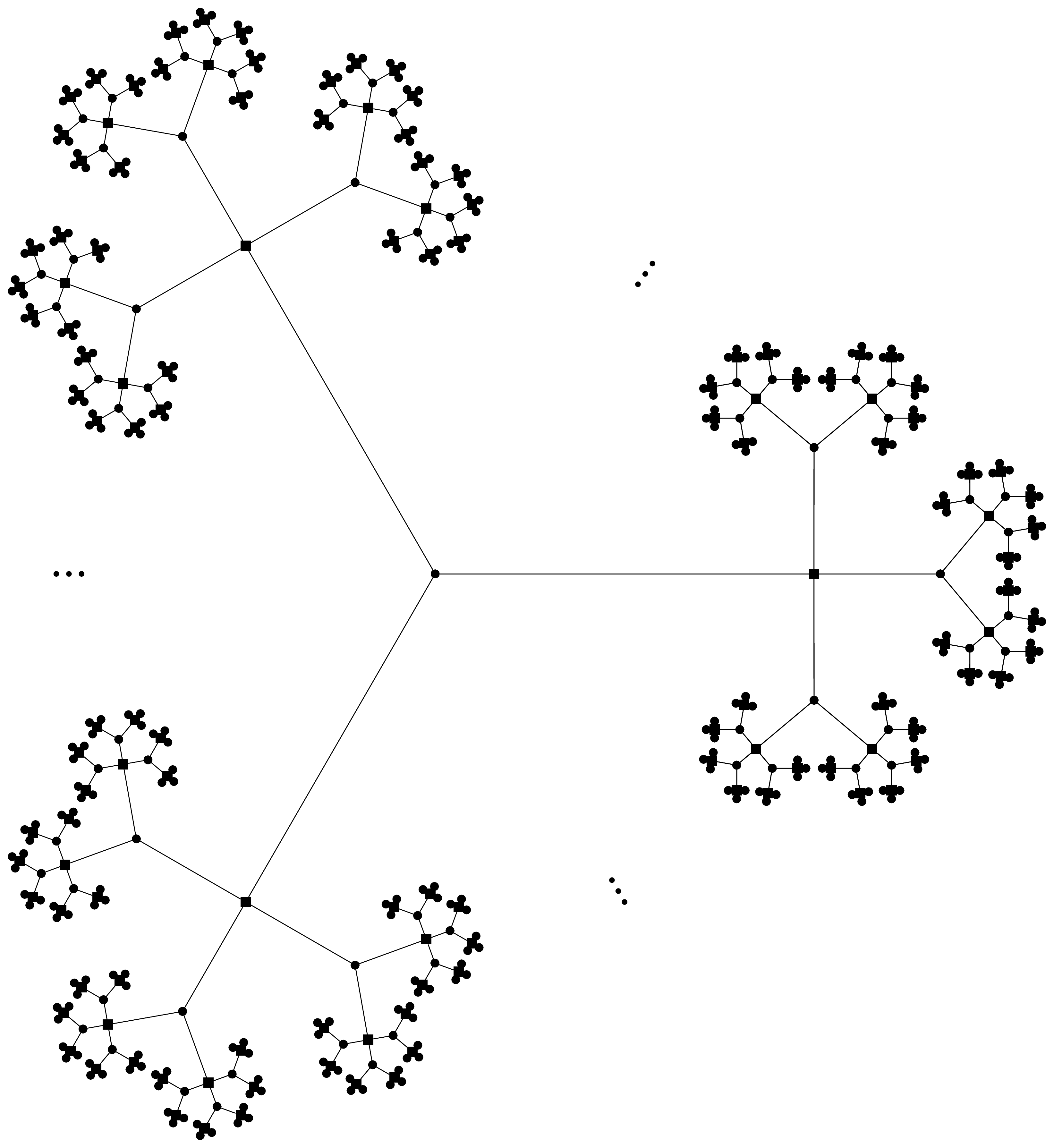}
    }
    \caption{$\Gamma(K_4, 3)$ (left) and its block-cut-vertex tree (right)}
    \label{fig:Gamma_D8_S3}
\end{figure}

\begin{example} Let $\Lambda$ be a vertex- and edge-transitive connected graph with connectivity at least two. Then $(\aut \Lambda) \boxtimes S_n = \aut \Gamma(\Lambda, n)$.
\end{example}

\section{The proof part I: Normal subgroups}
\label{section:proof_part_1}

In this section we investigate the normal subgroup structure of subdegree-finite primitive permutation groups, culminating in the proof of Theorem~\ref{thm:min_normal_subgroups}.

Henceforth, fix an infinite set $\Omega$ and a primitive permutation group $G \leq \sym(\Omega)$ that is closed in the permutation topology. Let $\alpha, \beta \in \Omega$ be distinct and suppose $\beta^{G_\alpha}$ is a finite nontrivial suborbit whose pair is also finite. 
Let $\Gamma$ be the orbital graph with vertex set $\Omega$ and edge-set $\{\alpha, \beta\}^G$.

\begin{remark} \label{rem:set_is_countable} It is well-known that under these conditions $\Omega$ must be countably infinite, and all suborbits of $G$ must be finite, so $G \in \mathcal{P}$.
Indeed, it is not difficult to see that the valency of $\alpha$ in $\Gamma$ is finite, and because $\Gamma$ is vertex transitive it follows that $\Gamma$ is locally finite. 
There is a natural $G$-invariant equivalence relation on the vertices of $\Gamma$, whereby two vertices are related if and only if they are connected. Since $G$ is primitive, this relation must be universal and $\Gamma$ is therefore connected.
The observation that $\Omega$ is countably infinite and all suborbits of $G$ are finite now follows immediately. Since all suborbits of $G$ are finite, for any pair of distinct vertices $\alpha', \beta' \in \Omega$, the orbital graph $(\Omega, \{\alpha', \beta'\}^G)$ is connected and locally finite.
\end{remark}

The normal subgroup structure of infinite primitive permutation groups with finite point stabilisers was determined in \cite{smith:discrete_prim}.

\begin{proposition}[{\cite[Theorems 2.2 \& 2.3]{smith:discrete_prim}}] \label{prop:discrete_normal_structure}
An infinite primitive permutation group with a finite point stabiliser has a unique (nontrivial) minimal normal subgroup that is the direct product of finitely many infinite nonabelian simple groups. These simple groups are pairwise isomorphic and are permuted transitively under conjugation by the finite stabiliser.
\end{proposition}

Since all subgroups in a discrete group are closed, Proposition~\ref{prop:discrete_normal_structure} implies that Theorem~\ref{thm:min_normal_subgroups} holds when $G$ is discrete. Theorem~\ref{thm:min_normal_subgroups} for nondiscrete groups can be established in a number of ways. The proof we give below is permutational and mirrors, as far as possible, the standard argument for finite groups; we exploit the topology of $G$ only where strictly necessary.

There is an alternative approach that is topological, in which most of Theorem~\ref{thm:min_normal_subgroups} (for nondiscrete groups) can be deduced from a theorem of P.~E.~Caprace and N.~Monod (\cite[Theorem E]{cm}). We briefly sketch this topological approach here. Caprace and Monod's theorem states that if $H$ is a compactly generated non-compact locally compact group in which every nontrivial closed normal subgroup is cocompact, then $H$ is either discrete or monolithic with monolith $\MinClosedNormal$, and in the latter case $\MinClosedNormal$ is isomorphic to ${\bf R}^n$ or $\MinClosedNormal$ is a quasi-product with finitely many isomorphic topologically simple groups as quasi-factors.
As a topological group, our permutation group $G$ is totally disconnected and locally compact. The point stabilisers $G_\alpha$ and $G_\beta$ are compact (they have only orbits of finite length) and maximal (because $G$ is primitive) and so $G$ is generated by the compact set $G_\alpha \cup G_\beta$. Because $G$ is primitive, all its closed nontrivial normal subgroups are transitive and therefore cocompact. Because $G_\alpha$ is open and $|G:G_\alpha|$ is infinite, $G$ cannot be compact. Hence Caprace and Monod's theorem applies to $G$, and one can see that if $G$ is not discrete then it is monolithic and the monolith of $G$ is a quasi-product with finitely many isomorphic topologically simple groups as quasi-factors.

\begin{proof}[Proof of Theorem~\ref{thm:min_normal_subgroups}]
Recall that $G \leq \aut \Gamma$, where $\Gamma$ is infinite, connected and locally finite with $V\Gamma = \Omega$. Let $X$ denote the set of neighbours of $\alpha$ in $\Gamma$. 

First, suppose $N \unlhd G$ is nontrivial. Then $N$ is transitive. If $G$ has no infinite point stabiliser, then $N$ has no infinite point stabiliser. On the other hand, if $N$ has no infinite point stabiliser, then $G$ has no infinite point stabiliser by Lemma~\ref{lemma:BurgerMozesInterpretation}(\ref{BurgerMozesInterpretation:caseN}). Hence $G$ has an infinite point stabiliser if and only if $N$ has an infinite point stabiliser.

If $G$ is discrete, then we have already seen that the result holds, so  assume $G$ is nondiscrete. We claim that if $N \unlhd G$ is closed and nontrivial, then $QZ(N) = \langle 1 \rangle$ and $N$ has a nontrivial minimal closed normal subgroup. Our proof of this claim is similar to that of \cite[Lemma 1.4.1]{BurgerMozes}. Indeed, suppose $N \unlhd G$ is closed and nontrivial. Because $G$ is primitive, $N$ is transitive. By assumption, $G$ is not discrete and therefore has no finite point stabilisers. Hence, by Lemma~\ref{lemma:BurgerMozesInterpretation}(\ref{BurgerMozesInterpretation:caseN}) it follows that $N$ also has no finite point stabilisers and by Lemma~\ref{lemma:TrivialQuasiCenter} we have $QZ(N) = \langle 1 \rangle$. Let $\mathcal{N}$ denote the set of all nontrivial closed normal subgroups of $N$. Suppose $N_1 \geq N_2 \geq \cdots$ is a chain in $\mathcal{N}$. For each $N_i$ the stabiliser $(N_i)_\alpha$ induces a permutation group $F_i$ on the finite set $X$. If some $F_i$ is trivial then (because $N$ is transitive, $\Gamma$ is connected, and $N_i \unlhd N$) the stabiliser $(N_i)_\alpha$ must fix every vertex of $\Gamma$, so it must be trivial;
hence $N_i$ is countable and is thus contained in $QZ(N)$ by Proposition~\ref{prop:QZContainsAllCountableNormalSubgroups}, but $QZ(N)$ is trivial. Hence each $F_i$ is finite and nontrivial, so there exists some nontrivial $f \in F$ which is contained in all of them. For each $i \geq 1$, let $P_i$ be the preimage of $f$ in $(N_i)_\alpha$. It is easy to see that each $P_i$ is closed and contained in the compact group $N_\alpha$;
since $\{P_i\}_{i \in \N}$ enjoys the finite intersection property, it follows that $P := \bigcap_{i \geq 1} P_i$ is nonempty. Because $f$ is nontrivial, $P$ is nontrivial. One can now easily verify that $\bigcap_{i \geq 1} N_i$ is closed, nontrivial and normal in $N$, and so it lies in $\mathcal{N}$. Hence, by Zorn's Lemma, the set $\mathcal{N}$ contains a minimal element, and our claim is established.

By our claim, we have that $G$ has a nontrivial closed minimal normal subgroup $\MinClosedNormal$ and $QZ(G) = \langle 1 \rangle$; applying the claim again, this time to $\MinClosedNormal$, we have that $\MinClosedNormal$ also has a nontrivial closed minimal normal subgroup $K$ and $QZ(\MinClosedNormal) = \langle 1 \rangle$. By Proposition~\ref{prop:QZContainsAllCountableNormalSubgroups} we have that $G$ and $B$ have no nontrivial countable normal subgroups, and so in particular point stabilisers in $K$ are infinite.

Note that $\MinClosedNormal$ is the unique minimal closed normal subgroup of $G$. Indeed, if $H \not = \MinClosedNormal$ is a minimal closed normal subgroup of $G$ then $\MinClosedNormal \cap H$ is closed and normal in $G$ and therefore $\MinClosedNormal \cap H$ is trivial. This implies that $H \leq C_{\sym(\Omega)}(\MinClosedNormal)$.
Because $\MinClosedNormal$ is transitive, it is well-known (see \cite[Theorem 4.2A]{dixon&mortimer} for example) that $C_{\sym(\Omega)}(\MinClosedNormal)$ must be semiregular. Hence $H$ is a countable normal subgroup of $G$, and so by Proposition~\ref{prop:QZContainsAllCountableNormalSubgroups} we have $H \leq QZ(G)$ and thus $H$ is trivial.

Now choose a countable subgroup $V \leq G$ such that the closure $\overline{V}$ of $V$ is equal to $G$ (see Remark~\ref{remark:CountableInterior}). Let $K^V := \{K^v : v \in V\}$ and $U := \langle K^h : h \in V \rangle$. Note that $\overline{U} \leq \MinClosedNormal$. Clearly $U$ is normalised
by $V$ and so
$\overline{U}$ is normalised by $\overline{V} = G$. Therefore
$\overline{U}$ is a closed nontrivial normal subgroup of $G$ and (because $\MinClosedNormal$ is unique) we have $\MinClosedNormal \leq \overline{U}$. Hence $\overline{U} = \MinClosedNormal$. 

We claim that $K^{G_{\alpha}}$ is finite and contains every nontrivial minimal closed normal subgroup of $\MinClosedNormal$. Our proof of this claim was inspired by the proof of \cite[Proposition 1.5.1]{BurgerMozes}. Suppose that $J$ is a minimal closed normal subgroup of $\MinClosedNormal$ and $J \not \in K^V$. For any $v \in V$ the group $J \cap K^v$ is a closed normal subgroup of $\MinClosedNormal$ and so it is trivial. Hence $J \leq C_{\sym(\Omega)}(U)$. Now $U$ is transitive because $\overline{U} = \MinClosedNormal$ is transitive and so again we have that $C_{\sym(\Omega)}(U)$ is semiregular. Therefore $J$ is a countable normal subgroup of $\MinClosedNormal$ and so, by Proposition~\ref{prop:QZContainsAllCountableNormalSubgroups}, $J \leq QZ(\MinClosedNormal)$. Since $QZ(\MinClosedNormal)$ is trivial, $J$ is trivial, and thus $K^V$ contains every nontrivial minimal closed normal subgroup of $\MinClosedNormal$. 
Now $\MinClosedNormal$ is transitive on $\Omega$, so $G = BG_\alpha$. Hence $K^V = K^G = K^{B G_{\alpha}} = K^{G_\alpha}$. Therefore $K^{G_\alpha}$ contains every nontrivial minimal closed normal subgroup of $\MinClosedNormal$. Since $K^{G_\alpha}$, being equal to $K^V$, is countable, $|G_{\alpha}: N_{G_{\alpha}}(K)| \leq \aleph_0$. Since $G_{\alpha}$ is closed, $N_{G_{\alpha}(K)}$ is closed (by Lemma~\ref{lemma:NormalizersAreClosed}) and we may apply Theorem~\ref{theorem:Evans} to deduce that there exists some finite set $\Phi \subseteq \Omega$ whose pointwise stabiliser in $G_\alpha$ is contained in $N_{G_{\alpha}}(K)$. Since $|G_\alpha : G_\alpha \cap G_{(\Phi)}|$ is finite, $|G_\alpha : N_{G_\alpha}(K)|$ is finite, and therefore $K^{G_\alpha}$ is finite. Hence our claim is true.

Write $K^{G_\alpha}$ as $\{K_1, \ldots, K_m\}$, with $m \in \N$ and $K = K_1$, and let $M := \langle K_1, \ldots, K_m \rangle$. Note that $M \leq B$ and $M \unlhd BG_\alpha = G$. Hence $\overline{M} = B$. Note also that $K_i \cap K_j$ is trivial for $i \not = j$.

It is now easy to see that $K$ is topologically simple. Indeed, if $K$ has a nontrivial closed normal subgroup $N$, then $N$ is normalised by $K_1$ and centralised by $K_2, \ldots, K_m$, and so $N \unlhd \overline{M} = B$. Since $K$ is a minimal closed normal subgroup of $B$ it follows that $N = K$.

If $K$ is abelian then $M$ is an abelian normal subgroup of $G$, and so in particular $M$ is abelian and transitive. Hence any element in $\Omega$ can be written $\alpha^h$ for some $h \in M$, and we have $(\alpha^h)^{M_\alpha} = \{\alpha^h\}$. Therefore $M_\alpha$ is trivial, which is impossible because $G$ has no countable nontrivial normal subgroups. Hence $K$ must be nonabelian. Moreover, by Lemma~\ref{lemma:NormalizersAreClosed}, $Z(K)$ is a closed normal subgroup of $K$ and so it must be trivial.

Finally, we claim that $M =  K_1 \times \cdots \times K_m$. Indeed, suppose the claim is false. Let $j$ be minimal such that $\langle K_1, \ldots, K_j \rangle \not = K_1 \times \cdots \times K_j$, and note that $j > 1$. Then there exists
some nontrivial $x \in K_j \cap \left ( K_1 \times \cdots \times K_{j-1} \right )$. 
For $1 \leq i \leq j-1$ let $k_i$ be the projection of $x$ into $K_i$. Since $x \in K_j$ we have that $x$ centralises each $K_i$, and thus $k_i \in Z(K_i)$ for each $i$. However, this is absurd because each $Z(K_i)$ is trivial. Our claim stands.
\end{proof}

Henceforth, $\MinClosedNormal$ is the unique minimal closed normal subgroup of $G$, where $\MinClosedNormal$ is the closure of $\DirectProd = K_1 \times \cdots \times K_m$, for some pairwise isomorphic closed infinite nonabelian topologically simple groups $K_1, \ldots, K_m$, and some finite $m \geq 1$. Since $\MinClosedNormal$ is a nontrivial normal subgroup of the primitive permutation group $G$, it follows that $\MinClosedNormal$, and hence $\DirectProd$, must act transitively on $\Omega$. 

Let $K := K_1$. Because $G_\alpha$ acts transitively on $\{K_1, \ldots, K_m\}$, we can choose $T := \{g_1, \ldots, g_m\} \subseteq G_\alpha$ such that $g_1 = 1$ and $K_i = K^{g_i}$ for $1 \leq i \leq m$. Note that $T$ is a right transversal of $N_{G}(K)$ in $G$.\\

We conclude this section by noting that $G$ acts faithfully as a group of automorphisms of $\DirectProd$ and $\MinClosedNormal$.

\begin{lemma} \label{lemma:G_discrete_implies_M_a_infinite} 
If $G_\alpha$ is not finite, then $QZ(G)$ is trivial, $\DirectProd$ is not countable and $\DirectProd_\alpha$ and $\MinClosedNormal_\alpha$ are infinite, with $|\MinClosedNormal_\alpha| = 2^{\aleph_0}$. 
\end{lemma}

\begin{proof} Suppose $G_\alpha$ is not finite. Since $G$ is primitive, every nontrivial normal subgroup of $G$ is transitive on $\Omega = V\Gamma$. Hence $QZ(G)$ is trivial and $G$ has no nontrivial countable normal subgroups by Lemma~\ref{lemma:TrivialQuasiCenter}. Hence $\DirectProd$ is not countable. Now $M$ is transitive on $\Omega$, so $|M : M_\alpha| = \aleph_0$. Thus $\DirectProd_\alpha$ and $\MinClosedNormal_\alpha$ are infinite. Since $\MinClosedNormal$ is closed, it follows from Theorem~\ref{theorem:Evans} that $|\MinClosedNormal_\alpha| = 2^{\aleph_0}$.
\end{proof}

\begin{proposition} $\DirectProd \leq G \leq \aut \DirectProd$ and $\MinClosedNormal \leq G \leq \aut \MinClosedNormal$. \label{prop:ActsAsAutL}
\end{proposition}
\begin{proof} Since $\DirectProd$ is nontrivial and normal in $G$, it is transitive on $\Omega$. Now $C:=C_{G}(\DirectProd) \unlhd N_{G}(\DirectProd) = G$, and so $C$ is either trivial or transitive.

Let us first suppose that $C$ is transitive. Then $C_{\sym(\Omega)}(\DirectProd)$ is transitive, which implies $\DirectProd$ is regular (see \cite[Theorem 4.2A]{dixon&mortimer} for example). Since point stabilisers in $\DirectProd$ are trivial, it follows from Lemma~\ref{lemma:G_discrete_implies_M_a_infinite} that $G_\alpha$ is finite and thus that $G$ is discrete. Hence $\DirectProd$ is closed, and therefore $\DirectProd = \MinClosedNormal$ is the unique minimal normal subgroup of $G$. Since $C$ is a nontrivial normal subgroup of $G$, it must therefore be the case that $\DirectProd \leq C$. This is absurd, however, because $\DirectProd$ is nonabelian.

Hence $C$ is trivial. Since $G$ acts on $\DirectProd$ by conjugation, and the kernel of this action is $C$, the result follows. Because $C_G(\MinClosedNormal) \leq C_G(\DirectProd)$, it also follows that $G \leq \aut \MinClosedNormal$.
\end{proof}

\begin{remark} \label{remark:Pg13.5} Note that $m = 1$ if and only if $G$ is almost topologically simple. Indeed, if $m = 1$ then $B = M = K_1$ and $G$ is almost topologically simple. Conversely, if $G$ is almost topologically simple then there is a closed normal nonabelian topologically simple group $N \unlhd G$ such that $N \leq G \leq \aut N$. As noted above, $B$ is the unique minimal closed normal subgroup of $G$, so in this case $B = N$ is topologically simple and $m = 1$.
\end{remark}

\section{The proof part II: $G$ is not almost topologically simple}
\label{section:proof_part_Not_AS}

In this section we investigate the structure of $G$ when $m > 1$. We prove the following result, which establishes most attributes of the PA case of Theorem~\ref{thm:MainTheorem}.

\begin{proposition} \label{prop:PA_case} \label{prop:Alternative_to_prop:PA_case}
Suppose $G \in \mathcal{P}$ is infinite. If $G$ is not almost topologically simple then $G$ has precisely one end and there exists a nontrivial finite transitive group $F \leq S_m$ and an infinite, almost topologically simple group $H \in \mathcal{P}$ such that (up to permutation isomorphism) $G$ is a primitive, fibrelobe-full subgroup of $H \Wr F$ acting via its product action, with $\sd(H) \leq \sd(G)$, and $K_1$ is the unique closed minimal normal subgroup of $H$. Furthermore, $G$ is discrete if and only if $H$ is discrete.
\end{proposition} 

The above result should look familiar to readers, since (ignoring ends and topological considerations) it mirrors the $m>1$ case of the finite O'Nan--Scott Theorem (\cite{liebeck&praeger&saxl:finite_onan_scott}) and the same case in the classification of infinite primitive permutation groups with finite point stabilisers (\cite{smith:orbital_digraphs}) and indeed all primitive groups with the FCR property (see \cite{PraegerSchneider}).

For Proposition~\ref{prop:PA_case}, we do not have the FCR property and only have that the minimal {\em closed} normal subgroups of $M$ are {\em topologically} simple. Nevertheless, it is interesting that we obtain a result here that is in line with what one would expect. It seems possible that the theory of cartesian decompositions developed by Praeger and Schneider can be modified slightly to apply to topological groups and then exploited to analyse transitive subdegree-finite permutation groups that are 
contained in a wreath product in product action.

One may wonder at this point why we do not appeal to the results of \cite{infinitary_versions}, since we have established that $G$ has a minimal closed normal subgroup that itself has a minimal closed normal subgroup. The reason is that the result we would obtain with such an approach would be weaker than Proposition~\ref{prop:PA_case}. We would not, for example, be able to conclude that $H$ is subdegree-finite --- a property that is fundamental to our later arguments.
\\

For the remainder of this section, suppose $m > 1$. 
We begin our proof of Proposition~\ref{prop:PA_case} by proving 
that the projection of $M_\alpha$ onto $K_1$ is a proper subgroup of $K_1$.

Let $\pi_i : \DirectProd \rightarrow K_i$ be the natural projection map, and let $R_i := \pi_i(\DirectProd_\alpha)$ and $R := R_1 \times \cdots \times R_m$. Clearly $\DirectProd_\alpha \leq R \leq \DirectProd$ and so $\DirectProd_\alpha = R_\alpha$. As in \cite{smith:discrete_prim}, we have that for all $x \in \DirectProd$ and all $g \in G_\alpha$,
\[\pi_i(x)^g = \pi_{i^{\sigma(g)}}(x^g),\]
where $\sigma(g) \in S_m$ is the permutation induced by the conjugation action of $g$ on $\{K_1, \ldots, K_m\}$. 

Hence $R$ is normalised by $G_\alpha$, so $G_\alpha R$ is a group. Therefore $G_\alpha \leq G_\alpha R \leq G$. Since $G$ is primitive, $G_\alpha$ is maximal, and so $G_\alpha R = G_\alpha$ or $G_\alpha R = G$.

\begin{lemma} \label{lemma:DoesNotProjectOntoLa} The projection of $\DirectProd_\alpha$ onto $K_1$ is a proper subgroup of $K_1$.
\end{lemma}
\begin{proof}
If $G_\alpha$ is finite, then $\DirectProd_\alpha$ is finite, and therefore the projection of $\DirectProd_\alpha$ onto the infinite groups $K_1$ must be a proper subgroup of $K_1$.

Now assume that $G_\alpha$ is infinite. Suppose, for a contradiction, that $\pi_1(\DirectProd_\alpha) = K_1$. It follows from this assumption that for any $N \unlhd \DirectProd_\alpha$ we have $\pi_1(N) \unlhd K_1$.
Since $K_1 \unlhd \DirectProd$, we have $K_1 \cap \DirectProd_\alpha \unlhd \DirectProd_\alpha$, and hence $K_1 \cap \DirectProd_\alpha = \pi_1(K_1 \cap \DirectProd_\alpha) \unlhd K_1$.

The direct product $\prod_{i=2}^m K_i$ centralizes $K_1$, so $\DirectProd = K_1 C_\DirectProd(K_1)$. Hence $K_1 \cap \DirectProd_\alpha \unlhd \DirectProd$. But $\DirectProd$ is transitive on $\Omega$ and $K_1 \unlhd \DirectProd$, so $K_1 \cap \DirectProd_\alpha$ must fix $\Omega$ pointwise. Hence $K_1 \cap \DirectProd_\alpha = \langle 1 \rangle$.

Now the point stabiliser of $\alpha$ in $K_1$ is $K_1 \cap \DirectProd_\alpha$, so $|K_1| = |K_1 : (K_1)_\alpha| \leq |\Omega| = \aleph_0$. It follows that $\DirectProd$ is countable and (because $\DirectProd$ is transitive) nontrivial. This is a contradiction, because Lemma~\ref{lemma:TrivialQuasiCenter} guarantees that $G$ contains no nontrivial countable normal subgroups.
\end{proof}

If  $G_\alpha R = G$, then $R$ is transitive. We must therefore have that $R = \DirectProd$ (this is easy to see: given any $\ell \in \DirectProd$ there exists $r \in R$ such that $\alpha^r = \alpha^\ell$, and so $\ell r^{-1} \in \DirectProd_\alpha = R_\alpha$ and thus $\ell \in R$). But this implies that $\pi_1(\DirectProd_\alpha) = K_1$, which contradicts Lemma~\ref{lemma:DoesNotProjectOntoLa}.

Hence, it must be the case that $G_\alpha R = G_\alpha$. Then $R \leq G_\alpha$ and so $R = \DirectProd_\alpha$.\\

For $i \leq m$, define $P_i := \core_{K_i}(R_i) = \bigcap_{k \in K_i} k^{-1} R_i k$ and note that $P_i \leq R_i$ and $P_i \unlhd K_i$. Thus $P := P_1 \times \cdots \times P_m$ is a normal subgroup of $\DirectProd$ and a subgroup of $G_\alpha$. Since $\DirectProd$ is transitive, this implies that $P$ is trivial. In particular, we have that $P_1$ is trivial.

Let $Y$ be the right cosets of $R_1$ in $K_1$. Then $K_1$ acts on $Y$ transitively with kernel $P_1$, so this action is faithful. Writing $\gamma := R_1 \in Y$, we have that the stabiliser $(K_1)_\gamma$ of $\gamma$ in $K_1$ is $R_1 = \pi_1(\DirectProd_\alpha)$.  

We can now apply the following lemma from \cite{smith:discrete_prim} to describe a permutation embedding of $G$ into a wreath product. There are aspects of the proof of this lemma that play a role in our later arguments, so we present the lemma together with an outline of its proof. For the complete argument see  \cite[Proof of Lemma 3.1]{smith:discrete_prim}. Our aim is to obtain a precise description of the permutational embedding of $G$ into our preferred wreath product.

\begin{lemma}[{\cite[Lemma 3.1]{smith:discrete_prim}}] \label{lemma:ProductAction} 
Suppose $\DirectProd = K_1 \times \cdots \times K_m \unlhd G \leq \sym(\Omega)$ for some $m > 1$, such that $\DirectProd$ is transitive and some point stabiliser $G_\alpha$ transitively permutes the components $\{K_1, \ldots, K_m\}$ of $\DirectProd$ by conjugation. Suppose further that $\DirectProd_\alpha = \pi_1(\DirectProd_\alpha) \times \cdots \times \pi_m(\DirectProd_\alpha)$, where each $\pi_i$ is the projection of $\DirectProd$ onto $K_i$. If $K:=K_1$ acts faithfully and transitively on a set $Y$ such that $\pi_1(\DirectProd_\alpha) = K_\gamma$ for some $\gamma \in Y$, then there exists a homomorphism $\psi : N_{G_\alpha}(K) \rightarrow N_{\sym(Y)_\gamma}(K)$ and a permutational embedding $(\hat{\phi}, \theta)$ of $G$ into $\sym(Y^m)$, where $\theta: \Omega \rightarrow Y^m$ is a bijection such that $\theta(\alpha) = (\gamma, \ldots, \gamma)$, and $\hat{\phi}: G \rightarrow \sym(Y^m)$ is a monomorphism such that $\hat{\phi}(\DirectProd) = K^m$ and $\hat{\phi}(G_\alpha) \leq \psi(N_{G_\alpha}(K)) \Wr S_m$ acting with its product action on $Y^m$.
\end{lemma}

\begin{proof}[Outline of proof] 
Recall that we have $T=\{g_1, \ldots, g_m\} \subseteq G_\alpha$ such that $K_i = K^{g_i}$ for $i = 1, \ldots, m$, with $g_1 = 1$. The action of any $g \in G_\alpha$ on $\{K_1, \ldots, K_m\}$ induces a permutation $\sigma(g) \in S_m$ so that $K_i^g = K_{i^{\sigma(g)}}$ for all $i$. For all $i \in \{1, \ldots, m\}$ and $g \in G_\alpha$ define $h_i := g_i gg_{i^{\sigma(g)}}^{-1}$ and observe that each $h_i$ lies in $N:=N_{G_\alpha}(K)$ (note that the proof given in  \cite{smith:discrete_prim} contains a misprint in the first paragraph: $N_{G_\alpha}(T)$ should read $N_{G_\alpha}(K)$).

We consider $K^m \leq \sym(Y^m)$ via the product action. Let $\underline{\gamma}:= (\gamma, \ldots, \gamma) \in Y^m$ and let $\psi : N \rightarrow N_{\sym(Y)_\gamma}(K)$ be the homomorphism which takes $h \in N$ to the permutation of $Y$ sending $\gamma^k$ to $\gamma^{h^{-1}kh}$ for all $k \in K$. Conjugation by $h \in N$ and by $\psi(h)$ induce the same automorphism of $K$.

Define $\phi : M \rightarrow K^m$ by $\phi(x) := (\pi_1(x)^{g^{-1}_1}, \ldots, \pi_m(x)^{g^{-1}_m})$ for each $x \in M$, where the maps $\pi_i$ are the projection maps onto the components of $K_1 \times \cdots \times K_m$. The map $\phi$ is an isomorphism between $M$ and $K^m$
such that $\phi(M_\alpha) = (K_\gamma)^m$. Let $\theta: \Omega \rightarrow Y^m$ be the map with $\theta(\alpha^x) := \underline{\gamma}^{\phi(x)}$ for all $x \in M$. The pair $(\phi, \theta)$ form a permutation isomorphism.

We now extend $(\phi, \theta)$ to a permutation isomorphism $(\hat{\phi}, \theta)$ between $\sym(\Omega)$ and $\sym(Y^m)$ in the following way. Given $f \in \sym(\Omega)$, let $\hat{\phi}(f)$ be the permutation of $Y^m$ which maps each $\underline{\delta} \in Y^m$ to $\theta(\theta^{-1}(\underline{\delta})^f)$. With a little work one can prove that $(\hat{\phi}, \theta)$ is a permutation isomorphism from $\sym(\Omega)$ to $\sym(Y^m)$ such that the restriction of $\hat{\phi}$ to $M$ is $\phi$, and moreover for any $g \in G_\alpha$ we have $\hat{\phi}(g) = (\psi(h_1), \ldots, \psi(h_m))\sigma(g) \in \psi(N) \Wr S_m$, where $h_{i} := g_i g g_{i^\sigma}^{-1}$ for $i = 1, \ldots, m$.

Thus the action of $\hat{\phi}(G)$ on $Y^m$ is as follows. Since $G = G_\alpha M = M G_\alpha$, any element in $g' \in G$ can be written as $g x$ for some $g \in G_\alpha$ and $x \in M$, with $\hat{\phi}(g') = \hat{\phi}(g) \hat{\phi}(x)$, so that,
\begin{equation}
	\label{eq:ActionOfPhiGa}
\underline{\delta}^{\hat{\phi}(g)} =
\left (
	\delta_{1^{\sigma^{-1}}}^{\psi(h_{1^{\sigma^{-1}}})},
	\ldots ,
	\delta_{m^{\sigma^{-1}}}^{\psi(h_{m^{\sigma^{-1}}})}
\right )
\end{equation}
and
\begin{equation}
	\label{eq:ActionOfPhiX}
\underline{\delta}^{\hat{\phi}(x)} =
\left (
	\delta_{1}^{g_1 \pi_1(x) g_1^{-1}},
	\ldots ,
	\delta_{m}^{g_m \pi_m(x) g_m^{-1}}
\right ),
\end{equation}
where where $\sigma:=\sigma(g)$ is the permutation in $S_m$ induced by the action of $g$ on $\{K_1, \ldots, K_m\}$, and $h_{i} := g_i g g_{i^\sigma}^{-1}$ for $i = 1, \ldots, m$.
\end{proof}

We continue to adopt the notation developed in the outline proof throughout this section. Examining the above proof, we have some additional information about the embedding $\hat{\phi}$ that we describe in Lemma~\ref{lem:G_a_acts_nicely_on_gamma_fibres} and Proposition~\ref{prop:H_is_G_on_fibres}. We will use these two results to establish the main result of this section, Proposition~\ref{prop:Alternative_to_prop:PA_case}.

Recall that $g_1 = 1$, and note that if $g \in N_{G_\alpha}(K)$ then $\sigma(g)$ stabilises $1 \in \{1,\ldots,m\}$.

\begin{lemma} 
\label{lem:G_a_acts_nicely_on_gamma_fibres}
\label{lem:KN_induces_H}
\label{lemma:KAndPsiNHasFiniteOrbits} 
Let $k \in K$ and $g \in N_{G_\alpha}(K)$, and write $x := (k, 1, \ldots, 1) \in M$. Then for all $(\delta, \gamma, \ldots, \gamma) \in Y_{\underline{\gamma}}[1]$ and for all $j \in \{1, \ldots, m\}$,
\begin{enumerate}
\item \label{Item:lem:KN_induces_H:One}
	$(\delta, \gamma, \ldots, \gamma)^{\hat{\phi}(g_j)} 
= (\gamma, \ldots, \gamma, \delta, \gamma, \ldots, \gamma) \in Y_{\underline{\gamma}}[j]$
\item \label{Item:lem:KN_induces_H:Two}
	$(\delta, \gamma, \ldots, \gamma)^{\hat{\phi}(xg)} = (\delta^{k \psi(g)}, \gamma, \ldots, \gamma)$
\item \label{Item:lemma:KAndPsiNHasFiniteOrbits}
	$K_\gamma$ and $\psi(N_{G_\alpha}(K))$ have only finite orbits on $Y$
\item \label{Item:lem:KIsClosed} 
	The group $K$ is a closed subgroup of $\sym(Y)$.
\end{enumerate}
\end{lemma}

\begin{proof} One can easily verify (\ref{Item:lem:KN_induces_H:One}) and (\ref{Item:lem:KN_induces_H:Two}) by using equations (\ref{eq:ActionOfPhiGa}) and (\ref{eq:ActionOfPhiX}) above.

To prove (\ref{Item:lemma:KAndPsiNHasFiniteOrbits}), write $\underline{\delta} := (\delta, \gamma, \ldots, \gamma) \in Y^m$ and $N := N_{G_\alpha}(K)$. 
Now $\underline{\delta}^{\hat{\phi}(N)}$ is finite because $N \leq G_\alpha$ and $G_\alpha$ has only finite orbits. Hence by (\ref{Item:lem:KN_induces_H:Two}) the projection $\pi_1(\underline{\delta}^{\hat{\phi}(N)}) = \delta^{\psi(N)}$ is also finite. 
Similarly, $\underline{\delta}^{\hat{\phi}(M_\alpha)}$ is finite because $M_\alpha \leq G_\alpha$. Recall from the outline of the proof of Lemma~\ref{lemma:ProductAction} that
 $\hat{\phi}(M_\alpha) = \phi(M_\alpha) = (K_\gamma)^m$. Therefore $|\underline{\delta}^{\hat{\phi}(M_\alpha)}| = |\underline{\delta}^{K_\gamma^m}| = |\delta^{K_\gamma}|$, and hence $|\delta^{K_\gamma}|$ is finite.
 
Finally, we prove (\ref{Item:lem:KIsClosed}).
As noted in the outline of the proof of Lemma~\ref{lemma:ProductAction}, the pair $(\hat{\phi}, \theta)$ is a permutation isomorphism from $\sym(\Omega)$ to $\sym(Y^m)$; hence $\hat{\phi}$ is an isomorphism of topological groups. By Theorem~\ref{thm:min_normal_subgroups}, $K$ is closed in $\MinClosedNormal$ and $\MinClosedNormal$ is closed in $G$. Since $G$ is closed in $\sym(\Omega)$, it follows that $K$ is closed in $\sym(\Omega)$. Thus $\hat{\phi}(K)$ is closed in $\sym(Y^m)$. Now $\hat{\phi}(K) = K \times \langle 1 \rangle \times \cdots \times \langle 1 \rangle \leq \sym(Y) \times \langle 1 \rangle \times \cdots \times \langle 1 \rangle$, and the latter group is clearly closed in $\sym(Y^m)$. Hence $K$ is closed in $\sym(Y)$.
\end{proof}

\begin{proposition} \label{prop:H_is_G_on_fibres} \label{prop:H_cannot_be_PA}
Let $H$ be the closure of $K \psi(N_{G_\alpha}(K))$ in $\sym(Y)$, and let $F \leq S_m$ be the permutation group induced by the conjugation action of $G_\alpha$ on $\{K_1, \ldots, K_m\}$. Then the following hold.
\begin{enumerate}
\item \label{Item:stabiliserIsFibreLobe} 
	$\hat{\phi}(G_\alpha) = (\hat{\phi}(G))_{\underline{\gamma}} \leq H_\gamma \Wr F$ and $F$ is transitive.
\item \label{Item:fibrelobeFull} 
	$\hat{\phi}(G)$ is a fibrelobe-full subgroup of $H \Wr F$.
\item \label{Item:stabiliser_of_H} 
	$H_\gamma$ is equal to the closure of $K_\gamma \psi(N_{G_\alpha}(K))$ in $\sym(Y)$.
\item \label{item:prop:H_cannot_be_PA}
	$H \in \mathcal{P}$ and $Y$ is countably infinite.
\item \label{item:prop:H_closed_min_normal}
	$H$ is almost topologically simple, with unique minimal closed normal subgroup $K$.
\item \label{item:prop:H_cannot_be_PA_pt2}
	$G_\alpha$ is finite if and only if $H_\gamma$ is finite.
\end{enumerate}
\end{proposition}
\begin{proof}
Recall from the outline proof of Lemma~\ref{lemma:ProductAction} that  $\sigma: G_\alpha \rightarrow S_m$ is a homomorphism. We have $F := \sigma(G_\alpha) \leq S_m$.
Because $G_\alpha$ acts transitively on $\{K_1, \ldots, K_m\}$, this group $F$ is transitive.
The permutational embedding $(\hat{\phi}, \theta)$ given in Lemma~\ref{lemma:ProductAction} takes the permutation group $(G, \Omega) = (MG_\alpha, \Omega)$ into $(H \Wr F, Y^m)$. 
Notice that $\hat{\phi}(G_\alpha) = (\hat{\phi}(G))_{\underline{\gamma}}$ and $\hat{\phi}(G_\alpha) \leq H_\gamma \Wr F$. Thus, (\ref{Item:stabiliserIsFibreLobe}) holds.

We now prove (\ref{Item:fibrelobeFull}).
We have already established that $\hat{\phi}(G) \leq H \Wr F$.
Now $\hat{\phi}(M) = K^m$ acts on $Y^m$ via the product action. Since $K$ is transitive on $Y$, we see that $K^m$ is transitive on the points of $H \Wr F$ and, for each $i$, it is transitive  on the set $\{Y_{\underline{\delta}}[i] : \underline{\delta} \in Y^m\}$. Furthermore,  Lemma~\ref{lem:G_a_acts_nicely_on_gamma_fibres}(\ref{Item:lem:KN_induces_H:One}) implies that $Y_{\underline{\gamma}}[1]^{\hat{\phi}(g_i)} = Y_{\underline{\gamma}}[i]$. Hence $\hat{\phi}(G)$ is transitive on the fibres of $H \Wr F$.

Let $J$ denote the setwise stabiliser in $\hat{\phi}(G)$ of the $\underline{\gamma}$-fibre $Y_{\underline{\gamma}}[1]$ and let $L$ be the closure of the subgroup of $\sym(Y)$ induced by $J$ acting on the first component of $Y_{\underline{\gamma}}[1]$.
Note that, since $\hat{\phi}(G) \leq H \Wr S_m$ and $H$ is closed, we have $L \leq H$.
By Lemma~\ref{lem:KN_induces_H}(\ref{Item:lem:KN_induces_H:Two}), we have $K \leq L$ and $\psi(N_{G_\alpha}(K)) \leq L$, and so $\langle \psi(N_{G_\alpha}(K)), K \rangle \leq L$. Since $L$ is closed and $H$ is the closure in $\sym(Y)$ of $K\psi(N_{G_\alpha}(K))$, we have $H \leq L$ and hence $H = L$.

Now $(\hat{\phi}(G))_{\underline{\gamma}} = \hat{\phi}(G_\alpha)$ permutes the set of $\underline{\gamma}$-fibres of $H \Wr F$. From equation (\ref{eq:ActionOfPhiGa}) above we see that $g \in \hat{\phi}(G_\alpha)$ permutes this set according to $\sigma(g)$, and so the permutation group induced by this action is permutationally isomorphic to $F$. Hence (\ref{Item:fibrelobeFull}) is true.

To see why (\ref{Item:stabiliser_of_H}) must be true,
note that $\overline{(K\psi(N_{G_\alpha}(K)))_\gamma} \leq H_\gamma$ and
any element $h \in H_\gamma$ must be the limit of a convergent sequence whose elements lie in $K \psi(N_{G_\alpha}(K))$. All but finitely many elements in the convergent sequence must fix $\gamma$ and so $h$ is in fact the limit of a convergent sequence in $(K\psi(N_{G_\alpha}(K)))_\gamma$. Hence $H_\gamma \leq \overline{(K\psi(N_{G_\alpha}(K)))_\gamma}$. Since $\psi(N_{G_\alpha}(K))$ fixes $\gamma$, it follows that $H_\gamma = \overline{K_\gamma \psi(N_{G_\alpha}(K))}$.

Now we prove (\ref{item:prop:H_cannot_be_PA})
By (\ref{Item:stabiliser_of_H})
the groups $H_\gamma$ and $K_\gamma \psi(N_{G_\alpha}(K))$ have the same orbits on finite tuples of elements in $Y$. By Lemma~\ref{lemma:KAndPsiNHasFiniteOrbits}, $K_\gamma$ and $\psi(N_{G_\alpha}(K))$ have only finite orbits on $Y$. Hence $K_\gamma \psi(N_{G_\alpha}(K))$ has only finite orbits. Since $H$ is transitive on $Y$, it follows that $H$ is subdegree-finite.
Now $G$ is primitive on $\Omega$, so $\hat{\phi}(G)$ is primitive on $Y^m$. Since $\hat{\phi}(G) \leq H \Wr F$, it follows that $H \Wr F$ is also primitive on $Y^m$. By Proposition~\ref{prop:wr_prim}, $H \Wr F$ is primitive in its product action on $Y^m$ so the group $H$ must be primitive but not regular on $Y$. 
The group $H$ is thus an infinite, closed, nonregular, primitive and subdegree-finite permutation group and so (by Remark~\ref{rem:set_is_countable}) the set $Y$ is countably infinite. 

To establish  (\ref{item:prop:H_closed_min_normal}), observe that $H$ satisfies the hypothesis of Theorem~\ref{thm:min_normal_subgroups}. It follows then that $H$ contains a unique (nontrivial) minimal closed normal subgroup. Since $K$ is a topologically simple, closed normal subgroup of $H$, it is the unique minimal closed normal subgroup of $H$. Thus $H$ is almost topologically simple by Proposition~\ref{prop:ActsAsAutL}.

Finally, we prove (\ref{item:prop:H_cannot_be_PA_pt2}).  
Indeed, if $H_\gamma$ is finite, then $H$ must be countable because
$\aleph_0 = |Y| = |H : H_\gamma|$. Hence $\hat{\phi}(G) \leq H \Wr S_m$ must also be countable. As noted previously, $G$ is countable if and only if $G_\alpha$ is finite.
On the other hand, if $G_\alpha$ is finite then $\psi(N_{G_\alpha}(K))$ and $\pi_1(M_\alpha)$ are finite. Now $K_\gamma = \pi_1(M_\alpha)$, therefore $K_\gamma \psi(N_{G_\alpha}(K))$ is finite. The closure of a finite group is  finite, so $H_\gamma$ is finite by (\ref{Item:stabiliser_of_H}) of this proposition.
\end{proof}

\begin{proposition} \label{prop:sdH_leq_sdG}
$\sd(H) \leq \sd(G)$.
\end{proposition}
\begin{proof} 
Now $G$ and $H$ are transitive, so any point stabiliser in $G$ has an orbit of length $\sd(G)$, and similarly for $H$. Hence $\hat{\phi}(G_\alpha)$ has an orbit of length $\sd(G)$.
Let $\underline{\delta} \in Y^m$ be such that the nontrivial suborbit $\underline{\delta}^{\hat{\phi}(G_\alpha)}$ has cardinality $\sd(G)$. Write $\underline{\delta} = (\delta_1, \ldots, \delta_m)$.
Since $|\underline{\delta}^{\hat{\phi}(G_\alpha)}| > 1$, we have that some $\delta_i \not = \gamma$. Writing $\sigma := \sigma(g_i^{-1})$ we have $1^{\sigma^{-1}} = i$. Thus for $\underline{\delta}' := \underline{\delta}^{\hat{\phi}(g_i^{-1})} \in \underline{\delta}^{\hat{\phi}(G_\alpha)}$ we have that the first component of $\underline{\delta}'$ is $\delta_i^{\psi(h_i)} = \delta_i^{\psi(g_i g_i^{-1}g_1)}  = \delta_i \not = \gamma$. Hence, without loss of generality, we can assume that $\delta_1 \not = \gamma$.

Recall that $g_1$ is trivial. Notice that if $\underline{k} = (k_1, \ldots, k_m) \in (K_\gamma)^m$ and $g \in N_{G_\alpha}(K)$, then $\sigma:= \sigma(g) \in S_m$ fixes $1 \in \{1, \ldots, m\}$ and so
the first component of $\underline{\delta}^{\underline{k}\hat{\phi}(g)}$ is $\delta_1^{k_1 \psi(h_1)} = \delta_1^{k_1 \psi(g_1 g g_1^{-1})} = \delta_1^{k_1 \psi(g)}$.

As noted in the outline of the proof of Lemma~\ref{lemma:ProductAction}, $\hat{\phi}(M_\alpha) = \phi(M_\alpha) = (K_\gamma)^m$. Therefore, 
$\sd(G)
= 
	|\underline{\delta}^{\hat{\phi}(G_\alpha)}| 
= 
	|\underline{\delta}^{\hat{\phi}(M_\alpha G_\alpha)}| 
\geq
	|\{\underline{\delta}^{\underline{k} \hat{\phi}(g)} : \underline{k} \in (K_\gamma)^m, \, g \in N_{G_\alpha(K)}\}|
\geq
	|\{\delta_1^{k \psi(g)} : k \in K_\gamma, \, g \in N_{G_\alpha}(K)\}|
$. Recall that a permutation group and its closure have the same orbits on finite tuples.
By Proposition~\ref{prop:H_is_G_on_fibres}(\ref{Item:stabiliser_of_H}) we therefore have that $|\{\delta_1^{k \psi(g)} : k \in K_\gamma, \, g \in N_{G_\alpha}(K)\}| = |\delta_1^{H_\gamma}|$. Hence $\sd(G) \geq |\delta_1^{H_\gamma}|$

By Proposition~\ref{prop:H_cannot_be_PA}(\ref{item:prop:H_cannot_be_PA}), $H$ is primitive but not regular. Hence, since $\delta_1 \not = \gamma$, we have $|\delta_1^{H_\gamma}| > 1$ and therefore $|\delta_1^{H_\gamma}| \geq \sd(H)$.
\end{proof}

Now that we understand the structure of $G$, we can precisely determine the number of ends it has. Our arguments follow the proof of \cite[Theorem 2.4]{smith:prim_subdegrees} with minor changes. 

\begin{lemma} \label{lemma:Wr_cartesian_product_is_graph_cartesian_product} 
Let $\Gamma = (Y, \{\gamma, \delta\}^H)$ be a connected orbital graph of $H$, define $\underline{\delta}:=(\delta, \gamma, \ldots, \gamma) \in Y^m$ and let $\Sigma := (Y^m, \{\underline{\gamma}, \underline{\delta}\}^{\hat{\phi}(G)})$. Then for vertices $\underline{\beta} := (\beta_1, \ldots, \beta_m)$ and $\underline{\beta'} := (\beta_1', \ldots, \beta_m')$ in $\Sigma$ we have $d_{\Sigma}(\underline{\beta}, \underline{\beta}') = 1$ if and only if $\sum_{i = 1}^m d_{\Gamma}(\beta_i, \beta_i') = 1$.
\end{lemma}

\begin{proof} 
If $d_{\Sigma}(\underline{\beta}, \underline{\gamma}) = 1$, then there exists some $g \in G$ such that $(\underline{\gamma}, \underline{\delta})^{\hat{\phi}(g)}$ is equal to $(\underline{\gamma}, \underline{\beta})$ or $(\underline{\beta}, \underline{\gamma})$. Write $\hat{g} := \hat{\phi}(g) = (h_1', \ldots, h_m')\sigma \in H \Wr S_m$.
If $(\underline{\gamma}, \underline{\delta})^{\hat{g}} = (\underline{\gamma}, \underline{\beta})$, then $\hat{g}$ fixes $\underline{\gamma}$, so each $h_i'$ lies in $H_\gamma$. Furthermore, $(\delta, \gamma, \ldots, \gamma)^{\hat{g}} = \underline{\beta}$, so precisely one component of $\underline{\beta}$ is not equal to $\gamma$. If $\beta \in Y \setminus \{\gamma\}$ is this component, then $\beta \in \delta^{H_\gamma}$. Hence, $\sum_{i=1}^m d_{\Gamma}(\gamma, \beta_i) = 1$.
On the other hand, if $(\underline{\gamma}, \underline{\delta})^{\hat{g}} = (\underline{\beta}, \underline{\gamma})$ then $\underline{\delta}^{\hat{g}} = \underline{\gamma}$. Therefore $\delta^{h'_1} = \gamma$ and for all $i > 1$ we have $h_1' \in H_\gamma$.  
Hence precisely one component of $\underline{\beta} = \underline{\delta}^{\hat{g}}$ is not equal to $\gamma$. If $\beta \in Y \setminus \{\gamma\}$ is this component, then $\beta = \gamma^{h_1'}$. Now $(\gamma, \delta)^{h_1'} = (\gamma^{h'_1}, \gamma) = (\beta, \gamma)$, and so $d_{\Gamma}(\gamma, \beta) = 1$. Hence, $\sum_{i=1}^m d_{\Gamma}(\gamma, \beta_i) = 1$.

We now prove the converse. Assume that $\sum_{i=1}^m d_{\Gamma}(\gamma, \beta_i) = 1$. Thus, there is precisely one component of $\underline{\beta}$ that is not equal to $\gamma$,
and if $\beta \in Y \setminus \{\gamma\}$ is this component, then $d_{\Gamma}(\gamma, \beta) = 1$.
In particular, $\underline{\beta}$ lies in $Y_{\underline{\gamma}}[j]$ for some $j \in \{1, \ldots, m\}$. Let $\hat{g}_j:=\hat{\phi}(g_j^{-1})$. 
By Lemma~\ref{lem:G_a_acts_nicely_on_gamma_fibres} we have $(\beta, \gamma, \ldots, \gamma) = \underline{\beta}^{\hat{g}_j}$, and since $\hat{g}_j \in \hat{\phi}(G_\alpha) \leq H_\gamma \Wr S_m$ we have $\underline{\gamma}^{\hat{g}_j} = \underline{\gamma}$.
Now $d_{\Gamma}(\gamma, \beta) = 1$, so there exists $h' \in H$ such that $\{\gamma, \beta\} = \{\gamma, \delta\}^{h'}$. Since $H$ is the closure of $K \psi(N_{G_\alpha}(K))$ in $\sym(Y)$, there exists $k \in K$ and $g \in N_{G_\alpha}(K)$ such that $(\gamma, \delta)^{h'} = (\gamma, \delta)^{k \psi(g)}$. Let $x := (k, 1, \dots, 1) \in M$. By Lemma~\ref{lem:KN_induces_H}, 
$\{\underline{\gamma}, \underline{\delta}\}^{\hat{\phi}(xg)} = \{(\gamma^{k\psi(g)}, \gamma, \ldots, \gamma), (\delta^{k\psi(g)}, \gamma, \ldots, \gamma)\} = \{\underline{\gamma}, \underline{\beta}\}^{\hat{g}_j}$.
Hence $\{\underline{\gamma}, \underline{\beta}\} \in \{\underline{\gamma}, \underline{\delta}\}^{\hat{\phi}(G)}$, so $d_{\Sigma}(\underline{\beta}, \underline{\gamma}) = 1$.
\end{proof}

Suppose $\Phi$ is a finite set of vertices in $\Sigma$, and $\underline{\alpha}$ is a vertex in $\Sigma \setminus \Phi$. Notice that one can create a path in $\Sigma$ from $\underline{\alpha} \in Y^m$ by only varying a single component. Once a desired distance from $\Phi$ has been reached, other components can be altered, and a path back constructed.
From this observation we see that for any finite set $\Phi$ of vertices, one can always find a path in $\Sigma \setminus \Phi$ between any two vertices in $\Sigma \setminus \Phi$. Hence $\Sigma$ has precisely one end.

\begin{corollary}
\label{cor:m_bigger_than_1_implies_one_ended} 
$G$ has precisely one end.
\end{corollary}

The main result of this section, Proposition~\ref{prop:PA_case}, now follows immediately from 
Propositions~\ref{prop:H_is_G_on_fibres} and \ref{prop:sdH_leq_sdG}, and Corollary~\ref{cor:m_bigger_than_1_implies_one_ended}.

\section{The proof part III: $G$ is almost topologically simple}
\label{section:proof_part_AS}

The following is a restatement of Remark~\ref{remark:Pg13.5}.

\begin{proposition} \label{prop:m_is_1_case}
If $m = 1$ then $G$ is almost topologically simple, with $\MinClosedNormal \leq G \leq \aut \MinClosedNormal$ for some infinite, nonabelian, closed, topologically simple, normal subgroup $\MinClosedNormal$ of $G$. 
\end{proposition}

Thus, if $m=1$ and $G$ has precisely one end, then our group $G$ is of type OAS in Theorem~\ref{thm:MainTheorem}. 

On the other hand, if $G$ has more than one end we can combine Corollary~\ref{cor:m_bigger_than_1_implies_one_ended} and Proposition~\ref{prop:m_is_1_case} to obtain the following.

\begin{theorem} If $G$ is a closed and subdegree-finite primitive permutation group with more than one end, then $G$ is almost topologically simple.
\end{theorem}

Compared this with the following result of M\"{o}ller and Vonk.

\begin{theorem}[{\cite[Theorem 4.1]{moller_vonk}}] \label{thm:moller_vonk}
Let $\Sigma$ be a locally finite connected primitive graph with infinitely many ends. Then $\aut \Sigma$ has a topologically simple subgroup, and this subgroup is open of finite index.
\end{theorem}

Underpinning M\"{o}ller and Vonk's result is a type of independence called {\em property H}. Recall that if $T$ is a tree and $e$ is an edge in $T$ such that the two components of $T\setminus e$ are infinite, then these components are called {\em half-trees} of $T$. A group acting on $T$ is said to have {\em property H} if the pointwise stabiliser of every half-tree is nontrivial. If $Q \leq \aut T$ is closed, then $Q^{++}$ is the closure of the subgroup generated by all pointwise stabilisers in $Q$ of half-trees. If $Q$ is closed, has property H and leaves no proper nonempty subtree or end of $T$ invariant, then $Q^{++}$ is contained in every closed subgroup of $Q$ that is normalised by $Q^{++}$; in particular, $Q^{++}$ is topologically simple (\cite[Theorem 2.4]{moller_vonk}).

In their proof of Theorem~\ref{thm:moller_vonk}, M\"{o}ller and Vonk note that there is a faithful action of $\aut \Sigma$ on an infinite tree $T$ that leaves no proper nonempty subtree or end of $T$ invariant, and moreover with this action $\aut \Sigma$
has property H, and the pointwise stabiliser in $\aut \Sigma$ of any edge $e$ of $T$ acts independently on the two components of $T \setminus e$. Property H guarantees that $\aut \Sigma^{++}$ is topologically simple, and the independence of the edge stabilisers is used to show that $\aut \Sigma^{++}$ is open. Now $\aut \Sigma$ is primitive in its action on $V\Sigma$ and $\aut \Sigma^{++}$ is nontrivial and normal in $\aut \Sigma$; therefore it acts transitively on $V\Sigma$. Since $\aut \Sigma^{++}$ is closed and transitive, it is cocompact in $\aut \Sigma$. Since $\aut \Sigma^{++}$ is also open, it follows that $\aut \Sigma^{++}$ is finite index in $\aut \Sigma$.

In our situation, where $G$ is a subdegree-finite primitive permutation group with more than one end, Proposition~\ref{prop:m_is_1_case_and_inf_ended} (below) tells us that $G$ can be written as a box product $G = G_1 \boxtimes G_2$, and therefore $G$ has a faithful action on its structure tree $T$. Moreover, since $G_1$ and $G_2$ are transitive, it follows that $G$ leaves no proper nonempty subtree or end of $T$ invariant. Hence if $G$ has property H, then $G^{++}$ is topologically simple by \cite[Theorem 2.4]{moller_vonk}. Since we know already that $B$ is the unique minimal closed normal subgroup of $G$, it follows that $B = G^{++}$. This is a very pleasant description of $B$. With this in mind, we make the following conjecture.

\begin{conjecture} If $G$ is a subdegree-finite primitive permutation group with more than one end, then $G$ has property H.
\end{conjecture}

When $G$ has more than one end, a great deal can be said about the structure of $G$. Our aim in this section is to prove the following.

\begin{proposition} \label{prop:m_is_1_case_and_inf_ended}\label{prop:Alternative_to_prop:m_is_1_case_and_inf_ended} 
Suppose $G \in \mathcal{P}$ is infinite. If $G$ has more than one end, then $G$ has $2^{\aleph_0}$ ends and there exists a nontrivial finite transitive group $F$ and a group $H \in \mathcal{P}$ with at most one end, such that $G$ is a primitive, fibrelobe-full subgroup of $H \boxtimes F$, with $\sd(H) < \sd(G)$. Moreover, if $\alpha \in \Omega$, then $F$ is the group induced by $G_\alpha$ acting on the set $\mathbb{L}(\alpha)$ of lobes of $H \boxtimes F$ that contain $\alpha$.
\end{proposition}

In order to prove this proposition, we must first introduce some algebraic graph and digraph theory.

In \cite{moller:prim_and_ends_of_graphs}, R\"{o}gnvaldur G.~M\"{o}ller uses the powerful theory of structure trees developed by Warren Dicks and M.~J.~Dunwoody in \cite{dicks_dunwoody} to show that primitive graphs with more than one end are closely related to graphs with connectivity one; M\"{o}ller's results are neatly summarised in his survey paper \cite{moller:groups_on_graphs}.

\begin{theorem}[{\cite[Theorem 15]{moller:groups_on_graphs}}]
\label{thm:gamma_related_to_conn1_graph}
Suppose that $\Sigma$ is a connected, locally finite, primitive graph with more than one end. Then there exist distinct vertices $\alpha', \beta'$ in $\Sigma$ such that the graph $\Sigma' := (V\Sigma, \{\alpha', \beta'\}^{\aut \Sigma})$ has connectivity one and each lobe $\Sigma'$ of has at most one end.
\end{theorem}

For our purposes, it will be more convenient to work with orbital digraphs rather than orbital graphs. One can use M\"{o}ller's result to prove the following (see \cite[Theorem 3.3]{smith:orbital_digraphs}).

\begin{theorem} 
\label{thm:me:orbitalgraphs}
If $H$ is a subdegree-finite, infinite, primitive permutation group with more than one end, then $H$ has a locally finite orbital digraph with connectivity one in which each lobe has at most one end.
\end{theorem}

The structure of primitive graphs with connectivity one was determined by H.~A.~Jung M.~E.~Watkins in \cite{jung_watkins}.

\begin{theorem}[{\cite[Theorem 4.2]{jung_watkins}}] \label{thm:jung_watkins}
A vertex transitive graph $\Sigma$ with connectivity one is primitive if and only if the lobes of $\Sigma$ are primitive, have at least three vertices, and are pairwise isomorphic.
\end{theorem}

Recall from Section~\ref{intro:box_prods} the graph $\Gamma(\Lambda, m)$: this graph has connectivity one, and every vertex lies in precisely $m$ lobes and every lobe is isomorphic to $\Lambda$. Together, Theorems~\ref{thm:me:orbitalgraphs} and \ref{thm:jung_watkins} imply that any subdegree-finite, infinite, primitive permutation group $H$ with more than one end has an orbital graph $\Gamma(\Lambda, m)$ for some $m \geq 2$ and some primitive, connected graph $\Lambda$, where $\Lambda$ has at least three vertices and at most one end.

If $G, H \in \mathcal{P}$ are primitive on $\Omega$, with $H \leq G$, and the orbital graph $(\Omega, \{\alpha, \beta\}^G)$ is of the form $\Gamma(m, \Lambda)$, then $(\Omega, \{\alpha, \beta\}^H)$ must be of the form $\Gamma(m, \Lambda')$. Hence we have the following well-known fact.

\begin{lemma} \label{lem:ends_of_primitive_subgroups_preserved} Suppose $H, G \in \mathcal{P}$ are primitive on $\Omega$ and $H \leq G$. If $G$ has more than one end, then $H$ has more than one end.
\end{lemma}

One consequence of this is the following: if $H \in \mathcal{P}$ and $F$ is nontrivial, transitive and finite, then any primitive subgroup of $H \boxtimes F$ has infinitely many ends.\\

Theorem~\ref{thm:jung_watkins} was extended to digraphs in \cite{smith:prim_digraphs}.

\begin{theorem}[{\cite[Theorem 3.3]{smith:prim_digraphs}}]
\label{thm:structure_of_con_one_digraphs}
If $\hat{\Sigma}$ is a vertex-transitive digraph with connectivity one, then it is primitive if and only if the lobes of $\hat{\Sigma}$ are primitive but not directed cycles, are pairwise isomorphic and each has at least three vertices.
\end{theorem}

The same paper contains the following characterisation of groups acting on primitive digraphs.

\begin{theorem}[{\cite[Theorem 2.6]{smith:prim_digraphs}}] \label{thm:smith_locally_prim}
Let $H$ be a vertex-transitive group of automorphisms of a connectivity-one digraph $\hat{\Sigma}$, whose lobes have at least three vertices. If $H$ acts primitively on $V\hat{\Sigma}$ and $\Lambda$ is some lobe of $\hat{\Sigma}$, then $H_{\{V\Lambda\}}$ is primitive and not regular on $V\Lambda$.
\end{theorem}

Using the above results, we can prove the main result of this section.

\begin{proof}[Proof of Proposition~\ref{prop:m_is_1_case_and_inf_ended}]
Suppose that $G$ has more than one end. By Theorem~\ref{thm:me:orbitalgraphs}, there is a locally finite orbital digraph $\Gamma$ with connectivity one in which each lobe has at most one end. This digraph is primitive because $G$ acts primitively on its vertex set, and so by Theorem~\ref{thm:structure_of_con_one_digraphs}, the lobes of $\Gamma$ are primitive but not directed cycles, are pairwise isomorphic, and each has at least three vertices. This tree-like graph clearly has $2^{\aleph_0}$ ends. Hence $G$ has $2^{\aleph_0}$ many ends. Let $\Lambda$ be a lobe in $\Gamma$ containing the vertex $\alpha$. Since $\Gamma$ is locally finite, vertex transitive and has connectivity one, there exists some integer $n \geq 2$ such that every vertex in $\Gamma$ lies in precisely $n$ distinct lobes.  Using our notation from Section~\ref{subsection:connectivity_of_graphs}, we can write $\Gamma = \Gamma(\Lambda, n)$.

Let $\mathbb{L}$ be the set of lobes of $\Gamma$, and for a vertex $\gamma$ let $\mathbb{L}(\gamma)$ denote the set of lobes in $\mathbb{L}$ that contain $\gamma$; this is a finite set with cardinality $n \geq 2$. Recall that the block-cut-vertex tree $T$ of $\Gamma = \Gamma(\Lambda, n)$ has vertex set $VT := V\Gamma \cup \mathbb{L}$ and edge set $ET := \{\{\gamma, \Lambda'\} : \gamma \in V\Gamma \text{ and } \Lambda' \in \mathbb{L}(\gamma)\}$. We again denote the neighbours in $T$ of a vertex $\gamma \in VT$ by $B(\gamma)$.

Let us pause momentarily, to make our notation easier to read. Since elements in $\mathbb{L}$ are on one hand lobes in $\Gamma$ and on the other hand vertices in $T$, we write $V_1 := \mathbb{L}$ and think of $V_1$ as vertices in $T$ and $\mathbb{L}$ as lobes in $\Gamma$. For our distinguished lobe $\Lambda$ of $\Gamma$, we write $\ell:=\Lambda$ and think of $\Lambda \in \mathbb{L}$ and $\ell \in V_1$. We do the same now for $V\Gamma$, writing $V_2 := V\Gamma$ and thinking of $V\Gamma$ as vertices in $\Gamma$ and $V_2$ as vertices in $T$.

Now we have $VT = V_1 \cup V_2$, and $V_1$ and $V_2$ are the two parts of the natural bipartition of $T$. As noted in Section~\ref{subsection:connectivity_of_graphs}, our group $G \leq \aut \Gamma$ induces a faithful action on $T$. Let $\hat{G}$ be the subgroup of $\aut T$ induced by this action of $G$. 

Now $G$ is transitive on the arcs of $\Gamma$. Since every lobe in $\Gamma$ contains an arc and no two lobes share any arc, it follows that $G$ is transitive on $\mathbb{L}$ and therefore $\hat{G}$ is transitive on $V_1$. Of course $\hat{G}$ is transitive on $V_2$ because $G$ is transitive $V\Gamma$.

Every arc in $\Gamma$ lies in a unique lobe in $\mathbb{L}$. Since $G$ is transitive on the arcs of $\Gamma$, it follows that $H_1 := G_{\{V\Lambda\}} \big|_{V\Lambda} $ must be arc transitive on $\Lambda$. Now, $B(\ell) = V\Lambda$, and so $\hat{G}_\ell = \hat{G}_{\{B(\ell)\}} =  \hat{G}_{\{V\Lambda\}}$.
Since $G_{\{V\Lambda\}}$ and $\hat{G}_{\{V\Lambda\}}$ induce the same permutation group on $V\Lambda$, we have that $\hat{G}_\ell$ induces $H_1$ on $B(\ell)$; that is, $H_1 = \hat{G}_\ell \big|_{B(\ell)}$. Hence, for all $\ell' \in V_1$, the permutation group induced by $\hat{G}_{\ell'}$ acting on $B(\ell')$ is permutation isomorphic to $H_1$. By Theorem~\ref{thm:smith_locally_prim}, $H_1$ is primitive but not regular on $V\Lambda$.

Let $F := G_\alpha \big|_{\mathbb{L}(\alpha)}$ so $F$ is a finite group induced by the action of $G_\alpha$ on the set of lobes that contain $\alpha$. Note that $G_\alpha \big|_{\mathbb{L}(\alpha)} = \hat{G}_\alpha \big|_{B(\alpha)}$.
 Since $G_{\{V\Lambda\}}$ is vertex- and arc-transitive on $V\Lambda$, we have that $A(\alpha) \cap A\Lambda$ is nonempty --- that is, there is an arc in $\Lambda$ whose origin is $\alpha$. 
 Hence the action of $G_\alpha$ on $\mathbb{L}(\alpha)$ is determined by the action of $G_\alpha$ on $A(\alpha)$. Since $G$ is arc transitive, $G_\alpha$ is transitive on $A(\alpha)$ and so (because $n \geq 2$) the action is nontrivial. Thus $F$ is finite, transitive and nontrivial.
 
The group $G$ is transitive on the vertices and lobes of $\Gamma$, therefore for all $\gamma \in V\Gamma$, the action of $G_\gamma$ on $\mathbb{L}(\gamma)$ is determined by $G_\gamma \big|_{A(\gamma)}$, and furthermore $G_\gamma \big|_{A(\gamma)}$ is permutation isomorphic to $G_{\alpha} \big|_{A(\alpha)}$. Hence $G_\gamma \big|_{\mathbb{L}(\gamma)} = \hat{G}_\gamma \big|_{B(\gamma)}$ is permutation isomorphic to $F$ for all $\gamma \in V_2$.
 
We have shown that $H_1$ and $F$ are transitive permutation groups of nontrivial degree and $\hat{G}$ is locally-$(H_1, F)$. Hence $\hat{G} \leq \mathcal{U}_{\mathcal{L}}(H_1, F)$ for some legal colouring $\mathcal{L}$ of $T$, by Proposition~\ref{prop:locally_MN}. Let $H$ be the closure of $H_1$ in $\sym(V\Lambda)$. Then $H_1 \leq H$, and so $\hat{G} \leq \mathcal{U}_{\mathcal{L}}(H_1, F) \leq \mathcal{U}_{\mathcal{L}}(H, F)$ by Theorem~\ref{thm:box_product_summary}. Looking at the permutation groups induced on $V_2$, we obtain $G \leq H \boxtimes_{\mathcal{L}} F$.

Now $H_1$ and $H$ have the same orbits on all finite tuples of $V\Gamma$, therefore $H$ has at most one end, has degree at least three, is closed and subdegree-finite and is primitive but not regular. Thus $H$ is a one-ended group in $\mathcal{P}$.

Since $\hat{G}$ is transitive on $V_1$, it transitively permutes the lobes of $H \boxtimes F$, but the lobes of $H \boxtimes F$ are subsets of $V_2$, so $G$ must also transitively permute the lobes of $H \boxtimes F$. 
Moreover, the set $V\Lambda = B(\ell)$ is a lobe of $H \boxtimes F$, and by definition $H$ is permutationally isomorphic to the closure of $G_{\{\Lambda\}} \big|_{\Lambda}$ in $\sym(\Lambda)$. Hence $G$ is a fibrelobe-full subgroup of $H \boxtimes F$. Of course $\Gamma$ is also an orbital graph of $H \boxtimes F$ and so the lobes of $\Gamma$ are the lobes of $H \boxtimes F$.

It remains to show that $\sd(H) < \sd(G)$. Indeed, because $G$ is transitive and $\hat{G}$ induces $G$ on $V_2$, there is some $\beta \in V_2 \setminus \{\alpha\}$ such that $|\beta^{\hat{G}_\alpha}| = \sd(G)$. Notice that the geodesic $[\alpha, \beta]_T$ between $\alpha$ and $\beta$ in $T$ contains a vertex in $V_2$ that is adjacent to $\alpha$; this vertex corresponds to a lobe containing $\alpha$. Without loss of generality, we can therefore assume that $\ell \in V_1$ is adjacent to $\alpha$ in $[\alpha, \beta]_T$. Let $\gamma$ be adjacent to $\ell$ in the geodesic $[\ell, \beta]_T$, and note that $\alpha, \gamma$ are distinct vertices in $\Lambda$.
 Hence $\hat{G}_{\alpha, \beta} \leq \hat{G}_\ell \cap \hat{G}_\gamma$, and so
\[|\beta^{\hat{G}_\alpha}| = |\hat{G}_\alpha : \hat{G}_{\alpha, \ell}| \cdot |\hat{G}_{\alpha, \ell} : \hat{G}_{\alpha, \ell, \gamma}| \cdot |\hat{G}_{\alpha, \ell, \gamma} : \hat{G}_{\alpha, \beta}|.\] 
Now $\hat{G}$ is locally-$(H_1, F)$ and $F$ is transitive with degree at least two, therefore $|\hat{G}_\alpha : \hat{G}_{\alpha, \ell}| > 2$. Hence $\sd(G) > |\hat{G}_{\alpha, \ell} : \hat{G}_{\alpha, \ell, \gamma}| = |\gamma^{(H_1)_{\alpha}}|$. Now $H$ is the closure of  $H_1$ in $\sym(V\Lambda)$, and so $|\gamma^{(H_1)_{\alpha}}| = |\gamma^{(H)_{\alpha}}| \geq \sd(H)$.
\end{proof}

\section{The proof part IV: Conclusion}
\label{section:proof_part_4}

In this section we draw together our arguments and prove the theorems stated in the introduction.

\begin{proof}[Proof of Theorem~\ref{thm:MainTheorem}]
Suppose the theorem is false, and $G \in \mathcal{P}$ 
is a minimal counterexample to the theorem with respect to $\sd(G)$. Since it is a counterexample, $G$ is infinite and is not of type OAS, so therefore either $G$ has more than one end, or $G$ is not almost topologically simple. We show that in both cases, $G$ cannot be a counterexample to the theorem.

First, let us suppose that $G$ has more than one end. By Proposition~\ref{prop:Alternative_to_prop:m_is_1_case_and_inf_ended}, $G$ has $2^{\aleph_0}$ many ends and there exists a nontrivial finite transitive group $F$ and a group $H \in \mathcal{P}$ with at most one end such that $G$ is a primitive, fibrelobe-full subgroup of $H \boxtimes F$ with $\sd(H) < \sd(G)$. If $H$ is finite, then $G$ satisfies the theorem. On the other hand, if $H$ is infinite then, because $\sd(H) < \sd(G)$, the group $H$ must satisfy the theorem. However, in this case $H$ has precisely one end, so it is of type OAS or PA. Hence if $H$ is infinite, then $G$ also satisfies the theorem. It follows then that our minimal counterexample $G$ must have precisely one end.

Since $G$ is not of type OAS and has precisely one end, it must be the case that $G$ is not almost topologically simple. By Proposition~\ref{prop:Alternative_to_prop:PA_case}, there exists a nontrivial finite transitive group $F$ and an infinite, almost topologically simple group $H \in \mathcal{P}$ such that $G$ is a primitive, fibrelobe-full subgroup of $H \Wr F$ acting via the product action, with $\sd(H) \leq \sd(G)$.

Since $H \in \mathcal{P}$ is infinite, it has precisely one end or it has infinitely many ends. If $H$ has precisely one end, then $H$ is OAS and $G$ satisfies the theorem. Since $G$ is a counterexample, it must be the case then that $H$ is infinitely ended.

By Proposition~\ref{prop:Alternative_to_prop:m_is_1_case_and_inf_ended} again, there exists a nontrivial finite transitive group $F_0$ and a group $H_0 \in \mathcal{P}$ with at most one end such that $H$ is a primitive, fibrelobe-full subgroup of $H_0 \boxtimes F_0$ with $\sd(H_0) < \sd(H) \leq \sd(G)$. If $H_0$ is finite then $H$ is of type BP, which implies that $G$ satisfies the theorem. Therefore $H_0$ must be infinite. However, since $G$ is a minimal counterexample and $\sd(H_0) < \sd(G)$, the group $H_0$ must satisfy the theorem. Since $H_0$ is infinite with at most one end, it must have precisely one end and is therefore of type OAS or PA. Hence $H$ is of type BP, which implies that $G$ is of type PA. In particular, $G$ satisfies the theorem.
\end{proof}

We proved Theorem~\ref{thm:min_normal_subgroups} in Section~\ref{section:proof_part_1}.

\begin{proof}[Proof of Corollary~\ref{cor:min_normal_subgroups_cor}]
In the statement of Theorem~\ref{thm:min_normal_subgroups}, we have that $m=1$ if and only if $G$ is almost topologically simple. If $G$ has precisely one end, then $G$ is of type OAS. If $G$ has more than one end, then $G$ is of type BP by Theorem~\ref{thm:MainTheorem}.

When $m > 1$ $G$ is not almost topologically simple, and so by  Proposition~\ref{prop:Alternative_to_prop:PA_case} we see that $K_1$ is the minimal closed normal subgroup of $H$.
\end{proof}

Our intention now is to prove Theorem~\ref{thm:main_product_statement}. To do this, we require the following two propositions. Throughout we use the notation of Theorem~\ref{thm:MainTheorem}.
 We begin with two lemmas.

\begin{lemma} \label{lem:Points_and_fibrelobes_coincide}
Suppose $G_i \leq \sym(Y_i)$ for $i =1,2,3,4$ are nontrivial transitive permutation groups.
If  $G_2$ is a fibrelobe-full subgroup of $G_4 \Wr G_3$, then the fibrelobes (resp. points) of $(G_4 \Wr G_3) \boxtimes G_1$ and the fibrelobes (resp. points) of $G_2 \boxtimes G_1$ are equal.

Similarly, if  $G_2$ is a fibrelobe-full subgroup of $G_4 \boxtimes G_3$, then the fibrelobes (resp. points) of $(G_4 \boxtimes G_3) \Wr G_1$ and the fibrelobes (resp. points) of $G_2 \Wr G_1$ are equal.
\end{lemma}
\begin{proof}
First suppose $G_2$ is a fibrelobe-full subgroup of $G_4 \Wr G_3$. Then both groups are transitive permutation groups on the same set of points, so this point set must equal $Y_2$. Hence the structure tree of $G_2 \boxtimes G_1$ is equal to the structure tree of $(G_4 \Wr G_3) \boxtimes G_1$. It follows immediately that the fibrelobes and points of 
$(G_4 \Wr G_3) \boxtimes G_1$ coincide with those of $G_2 \boxtimes G_1$.

Now suppose instead that $G_2$ is a fibrelobe-full subgroup of $G_4 \boxtimes G_3$. Then the set of points of $G_4 \boxtimes G_3$ is equal to $Y_2$. Therefore the points of $(G_4 \boxtimes G_3) \Wr G_1$ is equal to $Y_2^{|Y_1|}$, which is equal to the points of $G_2 \Wr G_1$. Of course this implies that the fibres of $(G_4 \boxtimes G_3) \Wr G_1$ and the fibres of $G_2 \Wr G_1$ also coincide.
\end{proof}

\begin{lemma} \label{lem:Decomp_of_leq_1_end}
If $G \in \mathcal{P}$ is of type PA, with $G \leq H \Wr F$, and $H$ is of type BP, then there is a nontrivial finite transitive permutation group $F_1$, and $H_0 \in \mathcal{P}$ with at most one end, such that $\sd(H_0) < \sd(G)$ and $G$ is a primitive subgroup of $(H_0 \boxtimes F_1) \Wr F$.
\end{lemma}
\begin{proof}
Suppose $G$ is of type PA, so $G$ is a primitive and fibrelobe-full subgroup of $H \Wr F$ with $\sd(H) \leq \sd(G)$. Suppose also that $H$ is of  type BP.
Thus $H$ is a fibrelobe-full subgroup of $H_0 \boxtimes F_1$, where $F_1$ is finite nontrivial and transitive, and $H_0 \in \mathcal{P}$ is of type FIN, OAS or PA and $\sd(H_0) < \sd(H) \leq \sd(G)$.

Now $G \leq H \Wr F \leq (H_0 \boxtimes F_1) \Wr F$, and by Lemma~\ref{lem:Points_and_fibrelobes_coincide}, the points of $H \Wr F$ and  $(H_0 \boxtimes F_1) \Wr F$ are equal. Hence $G$ is a primitive subgroup of $(H_0 \boxtimes F_1) \Wr F$.
\end{proof}

\begin{remark}
In the above lemma, because the fibrelobes of $H \Wr F$ and  $(H_0 \boxtimes F_1) \Wr F$ are equal, one might wonder if $G$ is fibrelobe-full in $(H_0 \boxtimes F_1) \Wr F$. In general, it is not, since on each fibrelobe $\Lambda$ the setwise stabiliser $G_{\{\Lambda\}}$ induces a group that is dense in $H$, but $H$ could be significantly smaller than the full box product $H_0 \boxtimes F_1$. 

Any fibrelobe $\hat{\Lambda}$ of $H_0$ is contained in some fibrelobe $\Lambda$ of $(H_0 \boxtimes F_1) \Wr F$. Now $\overline{G_{\{\Lambda\}} \big|_{\Lambda}} = H$, and $\overline{H_{\{\hat{\Lambda}\}} \big|_{\hat{\Lambda}}} = H_0$. In this way, one can trace the action of $G$ through nested sequences of fibrelobes. One can perform a similar analysis (with increasingly cumbersome notation) for the following two propositions, and indeed for Theorem~\ref{thm:main_product_statement}.
\end{remark}

\begin{proposition} \label{lem:Decomp_of_leq_1_end_v2}
If $G \in \mathcal{P}$ has at most one end, then there is a finite sequence $F_1, \ldots, F_n$ of finite transitive groups, with $F_i$ nontrivial for all $i > 1$, and a group $H_0 \in \mathcal{P}$ that is of type FIN or OAS, such that $G$ is a primitive   subgroup of
\[(((H_0 \Wr F_1 ) \boxtimes F_{2} ) \Wr \cdots \boxtimes F_{n-1}) \Wr F_n.\]
\end{proposition}

\begin{proof}
Suppose the result is false, and $G \in \mathcal{P}$ 
is a minimal counterexample to the corollary with respect to $\sd(G)$. Now $G$ has at most one end, and so by Theorem~\ref{thm:MainTheorem} $G$ is of type FIN, OAS or PA. If $G$ is of type FIN or OAS then $G$ satisfies the proposition (with $n = 1$ and $F_1$ trivial). Since this is impossible, $G$ must be of type PA. Thus we can write $G \leq H \Wr F_2$ for some nontrivial transitive and finite group $F_2$ and some $H \in \mathcal{P}$ of type OAS or BP. If $H$ is of type OAS then again $G$ satisfies the proposition which is impossible. Therefore $H$ is of type BP.
By Lemma~\ref{lem:Decomp_of_leq_1_end}, $G$ is a primitive subgroup of $(H _0 \boxtimes F_1) \Wr F_2$, where $F_1$ is nontrivial transitive and finite, and $H_0 \in \mathcal{P}$ has at most one end and $\sd(H_0) < \sd(G)$. 

Since $G$ is a minimal counterexample, $H_0$ must satisfy the corollary. Hence there exist finite transitive groups $J_1, \ldots, J_n$, with $J_i$ nontrivial for all $i > 1$, and there exists $\hat{H} \in \mathcal{P}$ that is of type FIN or OAS such that $H_0$ is a primitive subgroup of $G_0 := (((\hat{H} \Wr J_1 ) \boxtimes J_{2} ) \Wr \cdots \boxtimes J_{n-1}) \Wr J_n$. 

Since $H_0$ and $G_0$ are permutation groups on the same set, by Theorem~\ref{thm:box_product_summary} we have $H_0 \boxtimes F_1 \leq G_0 \boxtimes F_1$. Both of these groups are of course also permutation groups on the same set. Therefore $(H_0 \boxtimes F_1) \Wr F_2 \leq (G_0 \boxtimes F_1) \Wr F_2$. Hence $G$ is a primitive subgroup of $(G_0 \boxtimes F_1) \Wr F_2$, and therefore $G$ is a primitive subgroup of
\[(( (((\hat{H} \Wr J_1 ) \boxtimes J_{2} ) \Wr \cdots \boxtimes J_{n-1}) \Wr J_n) \boxtimes F_1) \Wr F_2.\]
Whence $G$ satisfies the corollary; this is impossible by assumption.
\end{proof}

\begin{proposition} \label{prop:Decomp_of_all_ends_v2}
If $G \in \mathcal{P}$ has more than one end, then there exists a finite sequence $F_1, \ldots, F_{n+1}$ of finite transitive groups, with $F_i$ nontrivial for all $i > 1$, and a group $H_0 \in \mathcal{P}$ that is either FIN or OAS, such that $G$ is a primitive subgroup of
\[((((H_0 \Wr F_1 ) \boxtimes F_{2} ) \Wr \cdots \boxtimes F_{n-1}) \Wr F_n) \boxtimes F_{n+1}.\]
\end{proposition}

\begin{proof}[Proof of Proposition~\ref{prop:Decomp_of_all_ends_v2}]
Suppose $G \in \mathcal{P}$ has more than one end. By Theorem~\ref{thm:MainTheorem}, $G$ is of type BP and so there is a nontrivial finite transitive group $F$ and a group $H \in \mathcal{P}$ of type FIN, OAS or PA such that $G$ is a primitive, fibrelobe-full subgroup of $H \boxtimes F$. In particular, $H$ has at most one end.
By Proposition~\ref{lem:Decomp_of_leq_1_end_v2}, there is a finite sequence $F_1, \ldots, F_n$ of finite transitive groups, with $F_i$ nontrivial for all $i > 1$, and a group $H_0 \in \mathcal{P}$ that is either FIN or OAS, such that $H$ is a  primitive subgroup of
$G_0 := (((H_0 \Wr F_1 ) \boxtimes F_{2} ) \Wr \cdots \boxtimes F_{n-1}) \Wr F_n$. Since $H$ and $G_0$ are permutation groups on the same set, we again have $H \boxtimes F \leq G_0 \boxtimes F$. We have thus shown that $G$ is a primitive subgroup of $((((H_0 \Wr F_1 ) \boxtimes F_{2} ) \Wr \cdots \boxtimes F_{n-1}) \Wr F_n) \boxtimes F$.
\end{proof}

Theorem~\ref{thm:main_product_statement} now follows immediately from Propositions~\ref{lem:Decomp_of_leq_1_end_v2} and \ref{prop:Decomp_of_all_ends_v2}.\\

Now we prove the topological interpretations of our main theorems.

\begin{proof}[Proof of Theorem~\ref{thm:MainTheoremTDLC}] The proof of this theorem relies upon Theorem~\ref{thm:MainTheorem}. Because of this, we must be careful not to confuse the various types in the two theorems since we have not yet proved that they coincide. So, for this proof let us call the types in Theorem~\ref{thm:MainTheorem} pFIN, pOAS, pPA and pBP, and let us call the types in Theorem~\ref{thm:MainTheoremTDLC} tOAS, tPA, tBP.

As noted previously, $G^* := G \sslash V$ is a closed, primitive and subdegree-finite permutation group on the set $\Omega := G/V$, and so $G^* \in \mathcal{P}$. Since $V$ is compact and $G$ is non-compact, it follows that $|G:V|$ is infinite and therefore $G^*$ is infinite. Fix $\alpha \in \Omega$.

By Theorem~\ref{thm:min_normal_subgroups}, $G^*$ has a unique nontrivial minimal closed normal subgroup $B$, so in particular $G^*$ is monolithic, with monolith $B$ that is nonabelian, and there exist finitely many topologically simple, pairwise isomorphic, nonabelian infinite permutation groups $K_1, \ldots, K_m$ such that each $K_i$ is closed and normal in $B$ and $M:= K_1 \times \cdots \times K_m$ is a dense subgroup of $B$. In other words, $B$ is a quasiproduct of $m$ copies of $K := K_1$. The stabiliser $G^*_\alpha$ acts by conjugation on $B$, permuting the components of $M$, and this conjugation action induces a nontrivial transitive group $F$.

All nontrivial normal subgroups of $G^* \in \mathcal{P}$ are transitive, and point stabilisers in $G^*$ are compact and open. Therefore by Proposition~\ref{Proposition_A},  all nontrivial normal subgroups of $G^*$ are cocompact in $G^*$. Hence $B$ is cocompact. 

Point stabilisers in $G^*$ are maximal subgroups so $G^*$ is compactly generated. For some $\beta \in \Omega \setminus \{\alpha\}$ let $\Gamma$ be the orbital graph $(\Omega, \{\alpha, \beta\}^{G^*})$. By Higman's criterion, this graph is connected. It is locally finite with bounded valency because $G^*$ is transitive and subdegree-finite.
Since $G^*$ is closed, tdlc and transitive with compact open vertex stabilisers, $\Gamma$ is a Cayley--Abels graph for $G^*$. Since $B$ is closed, $B$ is tdlc with compact open point stabilisers. Furthermore, $B$ is transitive on the vertices of $\Gamma$. It follows then that $\Gamma$ is also a Cayley--Abels graph for $B$, so $B$ is compactly generated. Since $\Gamma$ is a Cayley--Abels graph for both $G^*$ and $B$, the two groups have precisely the same ends.

We can now apply Theorem~\ref{thm:MainTheorem} to $G^*$, and deduce that $G^*$ (as a permutation group) is of type pOAS, pPA or pBP. 

The circular nature of our decomposition theorem makes a direct proof that reconciles the types pOAS, pPA, pBP, tOAS, tPA and tBP difficult; instead we again use a minimum counterexample argument. Suppose $G^*$ is a counterexample to Theorem~\ref{thm:MainTheoremTDLC} that is minimal with respect to $\sd(G^*)$.

If $G^*$ is of type pOAS, then $G^*$ is one-ended and has a nonabelian, closed normal subgroup that is topologically simple, which of course must equal $B$. Therefore $B$ is topologically simple, and as abstract groups we have $B \leq G^* \leq \aut(B)$. 

If $G^*$ is discrete then $B$ is discrete. If $G^*$ is not discrete, then $G^*$ has an infinite point stabiliser. Every nontrivial normal subgroup of $G^*$ is transitive on $\Gamma$, therefore, by Lemma~\ref{lemma:TrivialQuasiCenter}, $G^*$ has no countable closed normal subgroups. Hence $B$ cannot be countable, and in particular is not discrete. Thus, $G^*$ (as a topological group) matches the description of type tOAS in Theorem~\ref{thm:MainTheoremTDLC}. This means that $G^*$ satisfies the theorem, which is counter to our assumption. Hence $G^*$ is not of type pOAS. 

Now we assume that $G^*$ is of type pBP. Then $B$ is topologically simple and $G^*$ is a fibrelobe-full, primitive subgroup of $H \boxtimes F'$ for some finite nontrivial transitive group $F'$ and $H \in \mathcal{P}$ is either a finite nonregular primitive permutation group, or $H$ is infinite of type pOAS or pPA. Since $H$ is a closed permutation group, we have that $H = H \sslash W$ where $W$ is any point stabiliser in $H$.
Our above argument shows that if $H$ is of type pOAS then $H$ is of type tOAS. In this case, $G^*$ would satisfy the theorem, which is not possible. Thus $H$ is of type pPA. Now $\sd(H) < \sd(G^*)$, and $G^*$ is a minimal counterexample, therefore $H$ satisfies the theorem. Since $H$ is of type pPA its monolith is not topologically simple. The only type in our theorem with this property is tPA. Whence, $H$ is of type tPA. However, this again implies that $G^*$ satisfies the theorem, which is impossible. Thus, we conclude that $G^*$ is not of type pBP.

Finally, we suppose that $G^*$ is of type pPA. Then $G^*$ is one-ended, primitive, fibrelobe-full subgroup of $H \Wr F$, where $F$ is the finite nontrivial transitive permutation group defined above and $H$ is some infinite group in $\mathcal{P}$ of type pOAS or pBP, with $\sd(H) \leq \sd(G^*)$. Since $H$ is a closed permutation group, we have that $H = H \sslash W$ where $W$ is any point stabiliser in $H$. By Proposition~\ref{prop:Alternative_to_prop:PA_case}, we have that $G^*$ is discrete if and only if $H$ is discrete.

The monolith of $H$ is $K$, which is a nonabelian, closed, transitive (and therefore cocompact) topologically simple, normal subgroup of $H$. From Theorem~\ref{thm:min_normal_subgroups} we have that the monolith of $G^*$ is $B$ and $B$ is the closure of $M \cong K^m$. Thus the monolith of $G^*$ is a quasiproduct of $m$ copies of $K$. Again from Theorem~\ref{thm:min_normal_subgroups}, the permutation group induced by the conjugation action of $G^*_\alpha$ on the components of $M$ is $F$.

Our above arguments show that if $H$ is of type pOAS, then it is of type tOAS and $G^*$ then satisfies the theorem. 
Since this cannot happen by assumption, we have that $H$ is in fact of type pBP. In particular $H$ has infinitely many ends. Therefore, if $H$ satisfies the theorem then $H$ must be of type tBP (since this is the only infinitely ended type). However, if $H$ is of type tBP then $G^*$ satisfies the theorem, which is not possible. Therefore $H$ must also be a minimal counterexample to the theorem, with $\sd(H) = \sd(G^*)$. However, our above argument shows that a minimal counterexample to the theorem cannot be of type pBP.

Since we have exhausted all possibilities, it must be the case that no minimum counterexample exists and the theorem stands.

Note that
discrete and nondiscrete examples exist for groups of type tOAS (see Section~\ref{section:examples}).
\end{proof}

Theorem~\ref{thm:topo_main_product_statement} requires no proof, since it is simply a restating of Theorem~\ref{thm:main_product_statement}.

\begin{proof}[Proof of Theorem~\ref{thm:detecting_simple_tdlc_groups}] Applying Theorem~\ref{thm:MainTheoremTDLC} (and adopting its notion) we see that $G \sslash V$ is of type OAS, PA or BP. If $G \sslash V$ is of type PA then 
$G \sslash V$ is permutationally isomorphic to a subgroup of $H \Wr F$, where $H \Wr F$ is acting via its product action on some set $X^m$, with $H \leq \sym(X)$ and $F \leq S_m$.

For each $i$, let $\pi_i : X^m \rightarrow X$ be the projection to the $i$-th coordinate.
Then $\Gamma_i := \{\pi_i^{-1}(y) : y \in X\}$ is a partition of $X^m$. Since $X$ is infinite and $m \geq 2$, one can easily verify that $\mathcal{E} := \{\Gamma_1, \ldots, \Gamma_m\}$ is a nontrivial homogeneous cartesian decomposition of $X^m$.
One can also easily verify that $\mathcal{E}$ is $(H \Wr F)$-invariant, and thus the corresponding nontrivial homogeneous cartesian decomposition on $G / V$ is
$(G \sslash V)$-invariant. Since this is impossible by assumption, $G \sslash V$ is not of type PA.

On the other hand, $G \sslash V$ is of type BP if and only if $G \sslash V$ has infinitely many ends. In this case, it is known that $G \sslash V$ splits non-trivially as an amalgamated free product over a compact open subgroup (see \cite{abels-74} or \cite[Theorem 3.18]{KronMoller}). Since this is impossible by assumption, $G \sslash V$ is not of type BP.

Hence, our nondiscrete group $G \sslash V$ is of type OAS. The monolith of $G \sslash V$ is thus nondiscrete, one-ended, topologically simple, compactly generated and nonabelian.
\end{proof}

\section{Examples}
\label{section:examples}

In this section we show that there are examples for each type occurring in Theorem~\ref{thm:MainTheorem}. 

\subsection{Discrete groups of type OAS and PA}
Discrete examples of infinite groups in $\mathcal{P}$ of type OAS can be found in \cite{smith:discrete_prim}. The paper contains examples in which the minimal closed normal subgroup is regular. These examples are built using Tarski-Ol'Shanski{\u\i} Monsters (see \cite{olshanski}). The paper also contains examples in which the minimal closed normal subgroup is nonregular. These examples are built using a result of V.~N.~Obraztsov (see \cite{obraztsov}).

Now that we have examples of discrete groups in $\mathcal{P}$ of type OAS, it is  trivial to construct discrete examples of groups of type PA: if $H \in \mathcal{P}$ is discrete of type OAS, then $H \Wr S_3$ lies in $\mathcal{P}$ and is discrete of type PA.

\subsection{Nondiscrete groups of type PA and BP}
Theorem~\ref{thm:MainTheorem} tells us that there are no discrete groups of type BP. Using basic properties of the box product, we see that $S_3 \boxtimes S_3$ is nondiscrete, lies in $\mathcal{P}$, and is of type BP.

Constructing a nondiscrete example of a group in $\mathcal{P}$ of type PA is now easy. Using basic properties of the wreath product in its product action, we see that $(S_3 \boxtimes S_3) \Wr S_3$ lies in $\mathcal{P}$, is nondiscrete, and is of type PA.\\

\subsection{Nondiscrete groups of type OAS}

We do not give an introduction to Kac--Moody groups here. Readers wishing for an introduction are referred to \cite{marquis:book}. We follow the notion of \cite{caprace:remy}.
Throughout this section, $(W, S)$ will be a Coxeter system with root system $\Phi$ and generalised Cartan matrix $A$.
Let $(\Lambda, \{U_\alpha\}_{\alpha \in \Phi})$ be the twin root datum of type $(W, S)$. Assume each root group is finite and that $W$ is infinite. The associated twin buildings are denoted $(B_+, B_-)$. The groups $\aut(B_+)$ and $\aut(B_-)$ are given the compact open topology, under which they are tdlc. The group generated by all root groups is $\Lambda^\dagger$ and $\Lambda^\dagger \unlhd \Lambda$. The kernel of the action of $\Lambda$ on $B_+$ is the centraliser $Z_\Lambda(\Lambda^\dagger)$ (see \cite[Section 1.1]{caprace:remy}). Let $\pi_+ : \Lambda \rightarrow \aut(B_+)$ be the natural homomorphism of the action of $\Lambda$ on $B_+$, and let $\Lambda_+^{\rm{eff}} := \pi_+(\Lambda) \leq \aut(B_+)$. The closure of $\Lambda_+^{\rm{eff}}$ in $\aut(B_+)$ is denoted $\overline{\Lambda}_+^{\rm{eff}}$, and is a completion of $\Lambda$. There is another completion of $\Lambda$ given in \cite{caprace:remy}, where $\Lambda$ is completed with respect to the so-called positive building measure; this completion is denoted $\overline{\Lambda}_+$.

\begin{proposition}[{\cite[Proposition 1]{caprace:remy}}] With the above notation we have the following.
\begin{enumerate}
\item
	$\overline{\Lambda}_+$ is tdlc under the (positive) building topology.
\item
	The map $\pi_+$ has a natural extension to a continuous surjective homomorphism $\overline{\pi}_+ : \overline{\Lambda}_+ \rightarrow \overline{\Lambda}_+^{\rm{eff}}$. The kernel of $\overline{\pi}_+$ is $Z_{\Lambda}(\Lambda^\dagger)$, and this is a discrete subgroup of $\overline{\Lambda}_+$.
\item
	The canonical map $\overline{\Lambda}_+ / Z_{\Lambda}(\Lambda^\dagger) \rightarrow \overline{\Lambda}_+^{\rm{eff}}$ is an isomorphism of topological groups
\end{enumerate}
\end{proposition}

\begin{proposition}[{\cite[Proposition 4, Proposition 11 and Lemma 9]{caprace:remy}}] \label{prop:KacMoody-simplicity}
Suppose $W$ is irreducible and non-spherical, with soluble root groups, and $\Lambda$ is a split or almost split Kac--Moody group and is generated by its root subgroups. Then $\overline{\Lambda}_+^{\rm{eff}}$ is topologically simple.
\end{proposition}

In \cite{caprace:marquis}, the authors investigate the number of open subgroups of locally compact groups. They say that a locally compact group $G$ 
{\em has few open subgroups} if every proper open subgroup of $G$ is compact, and $G$ is {\em Noetherian} if $G$ satisfies the ascending chain condition for open subgroups.

\begin{proposition}[{\cite[Corollary C]{caprace:marquis}}] \label{prop:Kac-Moody-Few-Open}
If $W$ is irreducible and $\Lambda$ is a Kac--Moody group over a finite field, then $\overline{\Lambda}_+$ has few open subgroups if and only if $W$ is of affine or of compact hyperbolic type.
\end{proposition}

Suppose $\overline{\Lambda}_+$ has few open subgroups. Given any open subgroup $U$ of $\overline{\Lambda}_+^{\rm{eff}}$, its preimage in $\overline{\Lambda}_+$ must also be open and therefore compact. The continuous image of a compact set is compact, so $U$ must also be compact in $\overline{\Lambda}_+^{\rm{eff}}$. Hence $\overline{\Lambda}_+^{\rm{eff}}$ must also have few open subgroups. 

\begin{proposition} \label{prop:Kac-Moody}
Suppose $W$ is irreducible and non-spherical, of affine or compact hyperbolic type,
with finite soluble root groups. Assume the generalised Cartan matrix $A$ is not of finite type.
 Suppose also that $\Lambda$ is a Kac--Moody group over a finite field, with $\Lambda$ split or almost split and generated by its root subgroups. Then the group $\overline{\Lambda}_+^{\rm{eff}}$ is a topologically simple, tdlc and nondiscrete subgroup of $\aut(B_+)$ that has a maximal subgroup $U$ such that $U$ is compact and open.
\end{proposition}
\begin{proof}
It is known that $\overline{\Lambda}_+^{\rm{eff}} \leq \aut(B_+)$ is tdlc and nondiscrete (see \cite[Section 1, pp 713]{marquis:paper} for example). It is topologically simple by Proposition~\ref{prop:KacMoody-simplicity}, and has few open subgroups by Proposition~\ref{prop:Kac-Moody-Few-Open}. Since it has few open subgroups, $\overline{\Lambda}_+^{\rm{eff}}$ is clearly Noetherian. Hence there is a subgroup $U < \overline{\Lambda}_+^{\rm{eff}}$ that is maximal subject to being open. By the few open subgroups condition $U$ is also compact. Of course if $U \leq H < \overline{\Lambda}_+^{\rm{eff}}$, then $H$ is also open. Since $U$ is maximal subject to being open, we have $U = H$. Thus, $U$ is abstractly a maximal subgroup of $\overline{\Lambda}_+^{\rm{eff}}$. Moreover, $U$ is compact and open.
\end{proof}

\begin{corollary} Under the conditions of Proposition~\ref{prop:Kac-Moody}, the group $\overline{\Lambda}_+^{\rm{eff}} \sslash U$ is a nondiscrete topologically simple group in $\mathcal{P}$.
\end{corollary}

Since $\overline{\Lambda}_+^{\rm{eff}} \sslash U$ lies in $\mathcal{P}$, it has a locally finite, connected orbital graph $\Gamma$ on which it acts vertex primitively. Hence $\overline{\Lambda}_+^{\rm{eff}}$ also acts vertex primitively on $\Gamma$, with compact open point stabilisers. It follows then that $\Gamma$ is a Cayley--Abels graph for  $\overline{\Lambda}_+^{\rm{eff}} \sslash U$ and for  $\overline{\Lambda}_+^{\rm{eff}}$. The ends of these two groups are thus the same. Now $\overline{\Lambda}_+^{\rm{eff}}$ acts transitively on the chambers of $B_+$, with compact open chamber stabilisers. Therefore the ends of $\overline{\Lambda}_+^{\rm{eff}}$ and the ends of $B_+$ coincide. The ends of right-angled buildings are investigated in \cite{caprace:only}.

\begin{proposition}[{\cite[Theorem 9.2, abridged]{caprace:only}}] \label{prop:CapraceEnds}
Let $B$ be a thick, semi-regular, locally finite, right-angled building of type $(W,S)$. Assume that $(W,S)$ is irreducible and non-spherical. Then the following are equivalent.
\begin{enumerate}
\item
	$B$ is one-ended;
\item \label{item:caprace:one}
	W is one-ended;
\item \label{item:caprace:two}
	W does not split as a free amalgamated product over a finite subgroup;
\item \label{item:caprace:three}
	There is no partition $S = S_0 \cup S_1 \cup S_2$ with $S_1$, $S_2$ non-empty, $m_{i,j} = 2$ for all $i,j \in S_0$ and $m_{i,j} = \infty$ for all $i \in S_1$ and $j \in S_2$.
\end{enumerate}
\end{proposition}

Under the conditions of Proposition~\ref{prop:Kac-Moody} we have that $B_+$ is a thick, regular, locally finite building of type $(W,S)$, where $(W,S)$ is irreducible and non-spherical (see \cite[Example 3.1]{Baumgartner_et_al}). 
If we also require that $(W,S)$ is a right-angled Coxeter system (i.e. $m_{i,j} = 2$ or $m_{i,j} = \infty$ for all $i \not = j$) then $B_+$ is also right-angled. Hence Proposition~\ref{prop:CapraceEnds} holds for $B_+$, so $B_+$ is one-ended if and only if $W$ is one-ended.

As noted in \cite{caprace:only}, the equivalence of (\ref{item:caprace:one})--(\ref{item:caprace:three}) is well-known, and this one-endedness condition can easily be read from the Coxeter diagram. Many such one-ended buildings exist, for example any Bourdon building (i.e. if $S = \{1,\ldots,r\}$ and $m_{i,j} = 2$ if and only if $|i-j|=1$ or $-1$). 

We summarise our analysis in this subsection with the following.

\begin{proposition} 
Let $(W, S)$ be be a right-angled Coxeter system with root system $\Phi$ and generalised Cartan matrix $A$.
Suppose $W$ is one-ended,
irreducible and non-spherical, of affine or compact hyperbolic type,
with finite soluble root groups. Assume the generalised Cartan matrix $A$ is not of finite type.
Suppose also that $\Lambda$ is a Kac--Moody group of type $(W,S)$
 over a finite field, with $\Lambda$ split or almost split and generated by its root subgroups. 
 
If the associated twin buildings are denoted $(B_+, B_-)$, then the completion
$\overline{\Lambda}_+^{\rm{eff}}$ is a topologically simple, tdlc and nondiscrete subgroup of $\aut(B_+)$ that has a maximal subgroup $U$ such that $U$ is compact and open. Moreover, $\overline{\Lambda}_+^{\rm{eff}} \sslash U \in \mathcal{P}$ is nondiscrete of type OAS.
\end{proposition}

\section{Acknowledgements}

The author would like to thank Colin Reid for helpful conversations, and Pierre--Emmanuel Caprace for pointing out that certain complete Kac--Moody groups would yield examples of nondiscrete groups in $\mathcal{P}$ of type OAS.

%
%

\vspace{5mm}


\begin{thebibliography}{99}
\markright{}
\bibitem {abels-74} H.~Abels,
	`Specker-Kompaktifizierungen von lokal kompakten
topologische Gruppen',
	{\em Math.~Z.}
	135 (1973/74) 325--362.
%
\bibitem {aschbacher_scott} M.~Aschbacher and L.~Scott,
    `Maximal subgroups of finite groups',
    {\em J.~Algebra}
    92 (1985) 44--80.
%
\bibitem {Baumgartner_et_al} U.~Baumgartner, J.~Parkinson, J.~Ramagge
	`Scale and tidy subgroups for Weyl-transitive automorphism groups of buildings',
	{\em J.~Algebra}
	520 (2019) 460--478.
%
\bibitem {BurgerMozes} M.~Burger and S.~Mozes,
	`Groups acting on trees: from local to global structure',
	{Publications math\'{e}matiques de l'I.H.\'{E}.S.}
	tome 92 (2000) 113--150.
%
\bibitem {cameron:permutation_groups} P.~J.~Cameron,
    {\em Permutation groups},
    London Mathematical Society Student Texts 45
    (Cambridge University Press, Cambridge, 1999).
%
\bibitem {caprace:remy} P.-E.~Caprace and B.~R\'{e}my,
	`Simplicity and superrigidity of twin building lattices',
	{\em Invent.~math.}
	176 (2009) 169--221.
%
\bibitem {cm} P.-E.~Caprace and N.~Monod,
	`Decomposing locally compact groups into simple pieces',
	{Math.~Proc.~Camb.~Phil.~Soc.}
	150 (2011) 97--128.
%
\bibitem {caprace:marquis} P.-E.~Caprace and T.~Marquis,
	`Open subgroups of locally compact Kac--Moody groups',
	{\em Math.~Z.}
	(2013) 291--313.
%
\bibitem {caprace:only} P.-E.~Caprace,
	`Automorphism groups of right-angled buildings: simplicity and local splittings',
	{\em Fundam.~Math.},
	224 (2014).
%
\bibitem {NewDirectionsBook} P.~E.~Caprace and N.~Monod (Eds.),
	{\em New Directions in Locally Compact Groups},
	London Mathematical Society Lecture Note Series 447 
	(Cambridge University Press, Cambridge, 2019)
\bibitem {vanDantzig} D.~van Dantzig,
	`Zur topologischen Algebra III',
	{\em Brouwersche Und Cantorsche Gruppen.~Compos.~Math.}
	3 (1936) 408--426.
%
\bibitem {dicks_dunwoody} W.~Dicks and M.~J.~Dunwoody,
	{\em Groups acting on graphs},
	Cambridge Studies in Advanced Math.
	17 (Cambridge University Press, 1989).
%
\bibitem {dixon&mortimer} J.~Dixon and B.~Mortimer,
    {\em Permutation groups},
    Graduate Texts in Mathematics 163
    (Springer-Verlag, New York, 1996).
%
\bibitem{Evans87} D.~M.~Evans,
	`A note on automorphism groups of countably infinite structures',
	{\em Arch.~Math.}
	49 (1987) 479--483.
%
\bibitem {gelander_and_glasner:onan_scott} T.~Gelander and Y.~Glasner,
    `An Aschbacher--O'Nan--Scott theorem for countable linear groups',
    {\em J.~Algebra}
    378 (2013) 58--63. 
%
\bibitem {hewitt_ross} E.~Hewitt \& K.~A.~Ross,
	{\em Abstract Harmonic Analysis},
	Vol. 1.
	Springer, Berlin (1963).
%
\bibitem {d_g_higman} D.~G.~Higman,
	`Intersection matrices for finite permutation groups',
	{\em J.~Algebra}
	6 (1967) 22--42.
%
\bibitem {jung_watkins} H.~A.~Jung and M.~E.~Watkins,
	`On the structure of infinite vertex-transitive graphs',
	{\em Discrete Math.}
	18 (1977) 45--53.
%
\bibitem {kourovka19} E.~I.~Khukhro and V.~D.~Mazurov (Eds.)
	{\em The Kourovka Notebook},
	No.~19,
	\url{https://kourovka-notebook.org}.
%
\bibitem {kovacs1989} L.~G.~Kov\'{a}cs,
	`Wreath decompositions of  finite permutation groups', 
	{\em Bull. Austral. Math. Soc.}
	40 (1989) 255--279.
%
\bibitem {KronMoller} B.~Kr\"{o}n \& R.~G.~M\"{o}ller,
	`Analogues of Cayley graphs for topological groups',
	{\em Math.~Z.}
	(2008) 637--675.
%
\bibitem {liebeck&praeger&saxl:finite_onan_scott} M.~W.~Liebeck, C.~E.~Praeger, and J.~Saxl,
    `On the O'Nan--Scott Theorem for finite primitive permutation groups',
    {\em J. Austral. Math. Soc.}
    44 (1988) 389--396.
%
\bibitem {infinitary_versions} D.~Macpherson and C.~E.~Praeger,
    `Infinitary versions of the O'Nan--Scott Theorem',
    {\em Proc. London Math. Soc.}
    (3) 68 (1994) 518--540.
%
\bibitem {marquis:paper} T.~Marquis,
	`Abstract simplicity of locally compact Kac--Moody groups',
	{\em Compositio Math.,}
	150 (2014), 713--728.
%
\bibitem {marquis:book} T.~Marquis,
	{\em An Introduction to Kac--Moody Groups over Fields},
	EMS Textbooks in Mathematics,
	European Mathematical Society (2018).
%
\bibitem {moller:aut_gps_reg_trees} R.~G.~M\"{o}ller,
	`The automorphism groups of regular trees',
	{\em J. London Math. Soc}
	(2) 43 (1991) 236--252.
%
\bibitem {moller:prim_and_ends_of_graphs} R.~G.~M\"{o}ller,
	`Primitivity and ends of graphs',
	{\em Combinatorica},
	14 (1994) 477--484.
%
\bibitem {moller:groups_on_graphs} R.~G.~M\"{o}ller,
	`Groups acting on locally finite graphs---a survey of the infinitely ended case',
	In {\em Groups `93, Galway/St. Andrews,}
	vol. 2, London Math. Soc. Lecture Note Ser. 212
	(Cambridge University Press, 1995), 426--456.
%
\bibitem {moller:tdlc_via_graphs_and_perms} R.~G.~M\"{o}ller,
	`Structure theory of totally disconnected locally compact groups via graphs and permutations',
	{\em Canad. J. Math.}
	54 (2002) 795--827.
%
\bibitem {moller_vonk} R.~G.~M\"{o}ller and Jan Vonk
	`Normal subgroups of groups acting on trees and automorphism groups of graphs',
	{\em J.~Group Theory}
	15 (2012) 831--850.
%
\bibitem {obraztsov} V.~N.~Obraztsov,
	`Embedding into groups with well-described lattices of subgroups',
	{\em Bull. Austral. Math. Soc.}
	54 (1996) 221--240.
%
\bibitem {olshanski} A.~Yu.~Ol'Shanski{\u\i},
    {\em Geometry of defining relations in groups},
    Mathematics and its Applications (Soviet Series) 70
    (Kluwer Acad. Publ., Dordrecht, 1991).
%
\bibitem {PraegerNeumannSmith} Cheryl E.~Praeger, Peter M.~Neumann and Simon M.~Smith,
	`Some infinite permutation groups and related finite linear groups',
	{\em J. Aust. Math. Soc.}
	102 (2017) 136--149.
%
\bibitem {PraegerSchneider} Cheryl E.~Praeger and Csaba Schneider,
	{\em Permutation groups and cartesian decompositions},
	London Mathematical Society Lecture Note Series: 449,
	(Cambridge University Press, 2018).
%
\bibitem {reid_wesolek:homomorphisms} C.~D.~Reid \& P.~R.~Wesolek,
	`Homomorphisms into totally disconnected, locally compact groups with dense image',
	{\em Forum Mathematicum}
	31 (2019).
%
\bibitem {Scott:Onan_Scott} L.~Scott,
    `Representations in characteristic $p$',
    {\em Santa Cruz conference on finite groups},
    Proceedings of Symposia in Pure Mathematics 37
    (American Mathematical Society, Providence, R.I., 1980), 318--331.
%
\bibitem {smith:orbital_digraphs} S.~M.~Smith,
	`Orbital digraphs of infinite primitive permutation groups',
	{\em J. Group Th.}
	10 (2007).
%
\bibitem {smith:prim_digraphs} S.~M.~Smith,
	`Infinite primitive directed graphs',
	{\em J. Algebr. Comb.}
	31 (2010) 131--141.
%
\bibitem {smith:prim_subdegrees} S.~M.~Smith,
	`Subdegree growth rates of infinite primitive permutation groups',
	{\em J. London Math. Soc.}
	82 (2010) 526--548.
%
\bibitem {smith:discrete_prim} S.~M.~Smith,
	`A classification of primitive permutation groups with finite stabilizers',
	{\em J. Algebra}
	432 (2015) 12--21.
%
\bibitem {smith:product} S.~M.~Smith,
	`A product for permutation groups and topological groups',
	{\em Duke Math.~J.} 
	166 (2017) 2965--2999.
%
\bibitem {willis94} G.~Willis,
	`The structure of totally disconnected, locally compact groups',
	{\em Math.~Ann.}
	300 (1994) 341--363.
%
\bibitem {willis01} G.~Willis,
	`Further properties of the scale function on a totally disconnected group',
	{\em J. Algebra}
	237 (2001) 142--164.
%
\bibitem {Wilson:JustInfinite} J.~S.~Wilson,
   `Groups with every proper quotient finite',
   {\em Proc.~Camb.~Philos. Soc.}
   69 (1971) 373--391.
%
\bibitem {Woess} W.~Woess,
	`Topological groups and infinite graphs',
	{\em Discrete Math.}
	95 (1991) 373--384.
\end{thebibliography}
\end{document}